\documentclass[11pt, reqno]{amsart}
\usepackage[english]{babel}
\usepackage{amsthm} 
\usepackage{amssymb,cite}
\usepackage{amsmath}
\usepackage{hyperref}
\usepackage{amsfonts}
\usepackage{stmaryrd}
\usepackage{fancyhdr}
\usepackage{graphicx,color}
\usepackage[left=1in, right=1in, top=1in, bottom=1in]{geometry}
\usepackage{comment}
\usepackage[mathscr]{euscript}
\usepackage{mathrsfs}
\usepackage[shortlabels]{enumitem}
\usepackage{bbm}
\usepackage[normalem]{ulem} 

\newtheorem{theorem}{Theorem}[section]
\newtheorem{prop}[theorem]{Proposition}

\newtheorem{hypothesis}[theorem]{Hypothesis}
\newtheorem{lemma}[theorem]{Lemma}

\newtheorem{corollary}[theorem]{Corollary}
\newtheorem{assumption}[theorem]{Assumption}

\theoremstyle{definition}
\newtheorem{definition}[theorem]{Definition}

\theoremstyle{remark}

\newtheorem{examplex}{Example}[section]
\newenvironment{example}
{\pushQED{\qed}\examplex}
{\popQED\endexamplex}

\newtheorem{remarkx}[theorem]{Remark}
\newenvironment{remark}
{\pushQED{\qed}\remarkx}
{\popQED\endremarkx}

\newcommand{\beg}{\begin{equation}}
\newcommand{\ee}{\end{equation}}

\newcommand{\cA}{\mathcal{A}}
\newcommand{\sX}{\mathscr{X}}
\newcommand{\sP}{\mathscr{P}}
\newcommand{\cH}{\mathcal{H}}
\newcommand{\cP}{\mathcal{P}}
\newcommand{\cL}{\mathcal{L}}
\newcommand{\cPt}{\tilde{\mathcal{P}}}
\newcommand{\cLt}{\tilde{\mathcal{L}}}
\newcommand{\lref}{\lambda_{\mathrm{r}}}
\newcommand{\lswitch}{\lambda_{\mathrm{b}}}

\newcommand{\Lref}{L_{\mathrm{r}}}
\newcommand{\Lswitch}{L_{\mathrm{b}}}

\newcommand{\Ham}{E^{\sX,\sP}}

\newcommand{\PDMP}{\mathrm{PDMP}}

\newcommand{\sH}{\mathscr{H}}
\newcommand{\Proj}{\mathrm{Proj}}

\newcommand{\kref}{k_r}
\newcommand{\ksw}{k_b}
\newcommand{\Dext}{D_{\mathrm{ext}}}
\newcommand{\DCo}{D_{C_0}}
\newcommand{\DL}{D_{L^2_\mu}}
\newcommand{\cylfunc}{\mathcal{F}C_b^\infty}
\newcommand{\cylfunccomp}{\mathcal{F}C_c^\infty}

\newcommand{\R}{\mathbb{R}}
\newcommand{\N}{\mathbb{N}}
\newcommand{\Tr}{\mathrm{Tr}}
\newcommand{\Sv}{\Sigma_v}
\newcommand{\C}{{\Sigma}}

\newcommand{\refr}{\mathrm{refr}}

\newcommand{\Qrefr}{Q_{\mathrm{r}}}
\newcommand{\Qrefl}{Q_{\mathrm{b}}}
\newcommand{\ha}{\mathcal{H}^{\alpha}}

\newcommand{\cN}{{\mathcal N}}

\newcommand{\E}{{\mathbb E}}

\newcommand{\cVZZ}{\mathcal{V}_{\mathrm{ZZS}}}
\newcommand{\cVBP}{\mathcal{V}_{\mathrm{BPS}}}
\newcommand{\LX}{L_{\mathfrak{X}}}
\newcommand{\fX}{\mathfrak{X}}

\newcommand{\cZ}{{\mathcal Z}}

\newcommand{\les}{\lesssim}

\newcommand{\h}{\mathcal{H}}

\renewcommand{\P}{\mathbb{P}}

\newcommand{\lv}{\left\vert}
\newcommand{\rv}{\right\vert}
\newcommand{\ran}{\rangle}
\newcommand{\lan}{\langle}
\newcommand{\bee}{\begin{equation}}
\newcommand{\eee}{\end{equation}}

\newcommand{\lt}{\tilde{\lambda}}
\newcommand{\vperp}{v^N_\perp}
\newcommand{\vN}{v^N}
\newcommand{\xperp}{x^N_\perp}
\newcommand{\xN}{x^N}
\newcommand{\cB}{\mathcal{B}}

\newcommand{\core}{\mathscr{C}}
\begin{document}
\title{Infinite Dimensional Piecewise Deterministic Markov Processes}

\author[Paul Dobson, Joris Bierkens]{Paul Dobson$^{(1)}$, Joris Bierkens$^{(2)}$}
\

\address{(1) School of Mathematics and Maxwell Institute for Mathematical Sciences, University of
Edinburgh, Edinburgh EH9 3FD, UK, pdobson@ed.ac.uk}
\address{(2) Delft Institute of Applied Mathematics, TU Delft, Mekelweg 4, 2628 CD Delft, the Netherlands, joris.bierkens@tudelft.nl}

\maketitle

\begin{abstract}
In this paper we aim to construct infinite dimensional versions of well established Piecewise Deterministic Monte Carlo methods, such as the Bouncy Particle Sampler, the Zig-Zag Sampler and the Boomerang Sampler.
In order to do so we provide an abstract infinite-dimensional framework for Piecewise Deterministic Markov Processes (PDMPs) with unbounded event intensities. We further develop exponential convergence to equilibrium of the infinite dimensional Boomerang Sampler, using hypocoercivity techniques. Furthermore we establish how the infinite dimensional Boomerang Sampler admits a finite dimensional approximation, rendering it suitable for computer simulation.
\end{abstract}

	\vspace{5pt}

\vspace{5pt}
{\sc AMS Classification (MSC 2020).} 60J25 (primary); 
60G53, 65C05, 46N30 (secondary).

\tableofcontents

\section{Introduction}

A \emph{piecewise deterministic Markov process (PDMP)} in a topological space $\mathcal Z$ is a Markov process with deterministic trajectories, with random jumps at random times. The deterministic trajectories are described by a deterministic semigroup flow $\varphi_t : \mathcal Z \rightarrow \mathcal Z$, $t \geq 0$. 
The random times are distributed according to a space-dependent event rate $\lambda : \mathcal Z \rightarrow [0,\infty)$. At these random times the process jumps according to a Markov transition operator $Q : \mathcal Z \times \mathcal B(Z) \rightarrow [0,1]$. Together the flow $(\varphi_t)_{t \geq 0}$, the jump intensity $\lambda$ and the transition operator fully describe the dynamics of the piecewise deterministic Markov process.

PDMPs have been introduced in the probability literature in the work by  Davis \cite{Davis1984, DavisMarkov}, motivated as providing a versatile probabilistic model in operations research. Later PDMPs have found many other applications in, e.g., mathematical biology \cite{Fontbona2012}, finance and statistical survival analysis \cite{MR2189574}. Recent years have seen a strongly emerging interest in PDMPs  from the field of computational statistics \cite{bierkens2019zig,Doucet}, where PDMPs are considered as viable alternatives for classical Markov Chain Monte Carlo methods such as the Metropolis-Hastings algorithm and the Gibbs sampler.

Traditionally PDMPs have been considered on (subsets of) a finite dimensional space $\mathcal Z$. However PDMPs can naturally be defined on, e.g., nonlinear manifolds or infinite dimensional spaces. In this work we carry out the latter task of defining and analyzing PDMPs on infinite dimensional Banach spaces. To our knowledge the only work concerning PDMPs on infinite dimensional spaces so far \cite{Riedler, Riedler2012}, restricts the jump intensities to be globally bounded. One of the key aspects in this work is therefore the extension to unbounded intensities, which is highly relevant from a practical perspective, especially for applications in computational statistics. For much of this work we have in mind an infinite-dimensional extension of piecewise deterministic Monte Carlo methods; in particular extensions of the Zig-Zag process \cite{bierkens2019zig}, the Bouncy Particle Sampler \cite{Doucet} and the Boomerang Sampler \cite{bierkens2020boomerang}. As it turns out these processes are not all easily extendable to infinite dimensions.

\subsection{Structure of this work}

In Section~\ref{sec:construction} we provide the general construction of a PDMP assuming values in a Banach space. Also, we construct the accompanying extended generator and establish the Feller property under the stated assumptions.

Next in Section~\ref{sec:PDMPforsampling} we determine a general condition for an infinite dimensional PDMP to have an invariant measure with a Gibbs density relative to a reference measure.

In Section~\ref{sec:PDMP-examples} we discuss several examples of infinite dimensional extensions of (traditionally) finite dimensional PDMPs: the Zig-Zag Sampler \cite{bierkens2019zig}, Bouncy Particle Sampler \cite{Doucet}, and the Boomerang Sampler \cite{bierkens2020boomerang}. As it turns out the Zig-Zag Sampler can be put into an infinite dimensional framework, but the Bouncy Particle Sampler does not satisfy the conditions of the general theory: a natural bound to use for contour reflections is the Cameron-Martin norm however this is only  finite on a space of measure zero. The Boomerang Sampler is particularly well suited for extension to infinite dimensions, as it allows a Gaussian reference measure which remains well-defined in an infinite dimensional space.
In the remainder of this work, we will therefore focus on the Boomerang Sampler. In Section~\ref{sec:core} we establish that cylindrical functions (see \eqref{eq:cylindricalfunction} for a definition) are a core for the generator of the Boomerang Sampler. In Section~\ref{sec:expconv} we provide conditions under which the Boomerang Sampler converges exponentially fast in $L^2(\mu)$, with $\mu$ the stationary measure, based on the hypocoercivity approach initiated by \cite{Dolbeault}, while being careful about its extension to an infinite dimensional setting.

Finally in Section~\ref{sec:finite-dimensional-approximation} we consider how infinite dimensional PDMPs may be approximated in finite dimensions, allowing for direct computer simulation. This may be motivated from the viewpoint of e.g. \cite{Stuart}, where the approach in Bayesian inverse problems is to formulate these in an infinite dimensional setting, design a suitable infinite dimensional sampling algorithm, and only at the time of implementation consider a suitable finite dimensional approximation. In Section \ref{sec:proof} we have collected the remaining proofs.

\subsection{Notation}\label{sec:notation}

For a given topological space, measurability is always considered relative to the Borel $\sigma$-algebra, unless stated otherwise.

Given two normed vector spaces $H_1,H_2$ we define the following function spaces:
\begin{itemize}
	\item The set of all bounded and measurable (respectively continuous) functions $f:H_1\to H_2$ we denote $B_b(H_1;H_2)$ (resp. $C_b(H_1;H_2)$) and we endow this space with the supremum norm, if $H_2=\R$ then we write $B_b(H_1)$ (resp. $C_b(H_1)$);
	\item For $k\geq 1$ or $k=\infty$ we denote by $C_b^k(H_1)$ to be the space of $k$ times differentiable bounded functions with bounded derivatives up to order $k$ (here and throughout when we say differentiable we mean in the sense of the Fr\'echet derivative). Similarly we denote by $C_c^k(H_1)$ to be the set of $k$ times continuously differentiable functions with compact support.
	\item Given a measure $\mu$ we denote the space of square integrable functions $f:H_1\to H_2$ by $L_\mu^2(H_1;H_2)$, if $H_2=\R$ write $L_\mu^2(H_1)$ and will often abbreviate this to $L_\mu^2$;  
\end{itemize}
	
As explained in the introduction, we will mostly be working in appropriate Hilbert spaces. However, the results of Section \ref{sec:construction} are stated in the setting of more general Banach spaces.  Throughout the paper we denote by   $(\mathcal Z, \| \cdot\|_{\mathcal Z})$ a (possibly) infinite dimensional Banach space and 
let ($\cH, \langle \cdot, \cdot \rangle, \| \cdot \|$) be an infinite dimensional separable Hilbert space.  Consider the probability measure $\pi$ on $\cH$, defined as follows:  
\begin{equation}\label{targetmeasure}
\frac{d\pi}{d\pi_0}  \propto \exp({-\Phi}), \qquad \pi_0:=\cN(0,\C).
\end{equation}
That is, $\pi$ is  absolutely continuous with respect to a Gaussian measure $\pi_0$ with mean zero and covariance operator $\C$. 
Here $\Phi$ is a real valued functional defined on $\cH$. { Measures of the form \eqref{targetmeasure}  naturally arise in Bayesian nonparametric statistics and in the study of conditioned diffusions \cite{Stuart,Hair:etal:05}.} 
  The covariance operator $\C$ is a positive definite, self-adjoint, trace class operator on 
$\cH$, with eigenbasis $\{\gamma_j^2, e_j\} $:
\begin{equation}\label{cphi}
\C e_j= \gamma_j^2 e_j, \quad \forall j \in\mathbb{N},
\end{equation}
and we assume that the set $\{e_j\}_{j \in \N}$ is an orthonormal basis for $\cH$, so that every $x \in \cH$  can be expressed as  $x = \sum_{j\geq1} \; x_j e_j$, where $x_j:=\langle x,e_j\rangle$ and the norm of $x \in \cH$ is
$$
\| x\|= \left(\sum_{j=1}^{\infty} \lv x_j \rv^2 \right)^{1/2}.
$$
For a function $\phi:\h\to\R$ we will write $\partial_i \phi(x):= \langle \nabla \phi(x), e_i \rangle = \lim_{h \rightarrow 0} \frac{\phi(x+h e_i) - \phi(x)}{h}$, whenever the limit exists.

If $A, B \in \R$ we write $A \les B$ if there exists a constant $K>0$ such that $A \leq K B$.
Finally, for any $a \in \R$, $a^+ = \max(a,0)$ and $a^- = \max(-a,0)$ denote the positive and negative part of $a$, respectively. 

For a Hilbert space $\h$ let $\cylfunc(H)$ (respectively $\cylfunccomp$) denote the set of all functions $F$ of the form
	\begin{equation}\label{eq:cylindricalfunction}
	F(z) = f(\langle h_1, z\rangle_{H}, \ldots, \langle h_n, z\rangle_{H})
	\end{equation}
	for some $n\in \N, h_1,\ldots, h_n\in H$ and $f\in C_b^\infty(\R^n)$ (respectively $f\in C_c^\infty(\R^n)$).

\section{Construction of infinite-dimensional PDMPs}\label{sec:construction}
Let us first describe how we can construct a PDMP on an infinite dimensional space. We will follow the construction given by \cite{DavisMarkov}. Let $\cZ$ be a Banach space, equipped with its Borel $\sigma$-algebra. We can define a PDMP by its characteristics $(\fX,\lambda,Q)$, where $\fX$ denotes the vector field of the deterministic flow, $\lambda$ denotes the jump intensity and $Q$ denotes the jump kernel. More specifically, $\fX: \mathcal Z \rightarrow \mathcal Z$ is a vector field on $\cZ$  which generates a flow map $\{\varphi_t\}_{t\geq 0}$, i.e.,  $\varphi_t : \mathcal Z \rightarrow \mathcal Z$ for each $t \geq 0$, and for each $t \geq 0$,
\[ \frac{d}{dt} \varphi_t(z) = \fX \varphi_t(z), \quad \varphi_0(z) = z.\] 
Note that we are assuming that $\varphi_t$ is well-defined for all $t\geq 0$, that is, there is no finite time explosion of the flow $\varphi_t(z)$ for any $z\in \cZ$. The function $\lambda$ maps from $\cZ$ into $[0,\infty)$ and has the property that for each $z\in \cZ$ there exists some $\varepsilon(z)>0$ such that $s\mapsto \lambda(\varphi_s(z))$ is integrable on $[0,\varepsilon(z))$. $Q$ is a transition kernel such that 
\begin{enumerate}
\item For every Borel measurable set $A\subseteq \cZ$ the map $z\mapsto Q(z,A)$ is measurable.
\item There are no phantom jumps, i.e. $Q(z,\{z\})=0$ for all $z\in H$.
\end{enumerate}

Let $(\Omega,\mathcal{F},\mathbb{P})$ denote the Hilbert cube, that is the canonical space for a sequence of i.i.d. uniformly distributed random variables $\{U_n\}_{n=1}^\infty$ taking values on $[0,1]$. Fix an initial condition $z\in \cZ$, and set 
\begin{equation*}
F(t,z) = \exp\left(-\int_0^t\lambda(\varphi_s(z))ds\right).
\end{equation*}
That is the survivor function of the first jump time $T_1$. Set 
\begin{equation*}
\psi_1(u,z) = \inf\{t: F(t,z)\leq u\}
\end{equation*}
where we use the convention that $\inf\emptyset=+\infty$. Define $T_1(\omega)=\psi_1(U_1(\omega),z)$. Now by \cite[Lemma 2.1.1]{Riedler} there exists a Borel measurable function $\psi_2^z:[0,1]\to \cZ$ such that the pushforward of the Lebesgue measure with respect to $\psi$ on $H$ equals $Q(z,\cdot)$. We define the sample path up to the first jump time by
\begin{equation*}
Z_t(\omega) = \begin{cases}
\varphi_t(z), & 0\leq t< T_1(\omega)\\
\psi_2^{\varphi_{T_1(\omega)}(z)}(U_2(\omega)), & t=T_1(\omega).
\end{cases}
\end{equation*}
We now iterate this procedure to construct $Z_t(\omega)$ for all $t\geq 0$ provided $\sup_n T_n=+\infty$, that is the process is {\em non-explosive}. 
If $\lambda$ is bounded then the process is non-explosive. In order to consider unbounded $\lambda$ we will split the transition kernel $Q$ into two parts $\Qrefr$ and $\Qrefl$ which will later correspond to a refreshment kernel and a reflection kernel respectively.

We will associate with this PDMP (at least formally) a generator
\begin{equation}\label{eq:formalgen}
    \cL f(z)=\LX f(z) + \lref(z) \int (f(y)-f(z)) \Qrefr(z,dy) + \lswitch(z) \int (f(y)-f(z)) \Qrefl(z,dy) , \quad z\in \cZ.
\end{equation}
Here $\LX$ is a linear differential operator on a space of functions from $\cZ$ to $\R$ defined by
\begin{equation*}
\LX f(z) = \left. \frac{d}{dt}f(\varphi_t(z)) \right|_{t = 0} = d f(z; \fX(z)),
\end{equation*}
where $df(z, \cdot)$ denotes the Fr\'echet derivative of $f$ at $z$, whenever this is well-defined. 



\begin{assumption}[Wellposedness of abstract PDMP]
\label{ass:general}
With the notation introduced so far, we make the following assumptions. 
\begin{enumerate}[(i)]
\item The intensities $\lref : \cZ \rightarrow [0,\infty)$ and $\lswitch : \cZ \rightarrow [0,\infty)$ are measurable.
\item for any fixed  $z\in\cZ$, $\Qrefl(z, \cdot), \Qrefr(z, \cdot)$ are Markov transition kernels on $\cZ$, such that $z \mapsto \Qrefl f(z)$ and $z \mapsto \Qrefr f(z))$ are measurable for every bounded measurable function $f$ on $\cZ$\,;
\item the reflection Markov kernel $\Qrefl$ is non-expansive with respect to the norm $\|\cdot\|_{\cZ}$ i.e.
\[ \Qrefl(z, \{ y \in \cZ: \|y\|_{\cZ} \leq \|z\|_{\cZ}\}) = 1 \quad \text{for all $z \in \cZ$} \, ;\] \label{ass:reflnormpreserving}
\item the refreshment Markov kernel $\Qrefr $ is locally uniformly bounded in probability, i.e., for all $\varepsilon > 0$ there exists a $R > 0$ such that 
\[ \Qrefr(z, B_R) > 1 - \varepsilon \quad \text{for all $z \in B_R$}, \]
where $B_R$ denotes the open ball with radius $R$ in $\cZ$ with respect to the $\|\cdot\|_{\cZ}$-norm,\label{ass:refrboundedinprob}
\item
the total reflection intensity $\lswitch(z)$ is bounded on bounded sets in $\cZ$, i.e. for all $R > 0$ 
\[ \sup_{z\in B_R} \lswitch(z) < \infty \, ;\]\label{ass:lboundedonbounded}
\item the total refreshment intensity $\lref(z)$ is globally bounded: there is a constant $M > 0$ such that $\lref(z) \leq M$ for all $z \in \cZ$;\label{ass:refrboundedrate}
\item there is a continuous function $t \mapsto c_t \in (0,\infty)$ such that the flow $\varphi_t(z)$ grows at a rate $c_t$ with respect to the $\lVert\cdot\rVert_\mathcal{\cZ}$-norm, i.e for $z \in \mathcal Z$,
\begin{align*}
    \lVert \varphi_t(z) \rVert_{\cZ}&\leq c_t(1+\lVert z\rVert_{\cZ}).
\end{align*}\label{ass:flowbound}
\item the functions $\lref,\lswitch$ are continuous and for any continuous function $f$ we have that $\Qrefl f,\Qrefr f$ are both continuous.
\end{enumerate} 
\end{assumption}

\begin{remark}[On Assumption \ref{ass:general}] \label{note:ZZSorBPSwellposed}
These conditions are intended to be as general as possible. From a practical viewpoint it seems that (iii) is the most difficult condition to satisfy. In particular, we will see in Section~\ref{sec:PDMP-examples} that this condition is not satisfied by the Bouncy Particle Sampler.
\end{remark}
The purpose of this section is to prove the following statement.
\begin{theorem}
\label{thm:normedspace}
Suppose Assumption~\ref{ass:general} is satisfied. There exists a cadlag Markov Process process $(Z_t)_{t \geq 0}$ in $\cZ$ with characteristics $(\fX,\lambda,Q)$.

\end{theorem}
\begin{proof} At the beginning of this section we gave a construction of the PDMP; however this construction is only valid for $t\leq \sup_nT_n$, therefore it remains to show that $\sup_nT_n=\infty$. This is a consequence of Proposition \ref{lem:noexplosion} below. The proof that this construction of a PDMP is a strong Markov process is analogous to \cite[Theorem 25.5]{DavisMarkov}.
\end{proof}

\begin{prop}
	\label{lem:noexplosion}
Suppose Assumption~\ref{ass:general} is satisfied. Write $T_{\mathrm{explode}} = \sup_{i\geq 1} T_i$. Then
\begin{enumerate}[(i)]
\item For all $\varepsilon > 0$, $t \geq 0$ and initial condition $z_0 \in \cZ$ there exists an $R > 0$ such that $\P_{z_0}(T_{\mathrm{explode}} > t \text{ and } Z_s\in B_R \ \text{for all $s \in [0,t]$}) > 1 - \varepsilon$.
\item $T_{\mathrm{explode}} = \infty$ almost surely.
\item For all $t \geq 0, r\geq 0$  \[\lim_{n\to\infty} \sup_{z_0\in B_r}\P_{z_0} \left( T_{n} > t \right) =1.\]
\end{enumerate}
\end{prop}

Statement (ii) of Proposition \ref{lem:noexplosion} is equivalent to saying that only finitely many events can occur in any finite time interval. 

\begin{proof}
Suppose $t \geq 0, n\in \mathbb{N}$ and $\varepsilon >0$, and let $r > 0$ sufficiently large so that $z_0 \in B_r$. 
Let $M$ be an upper bound for the total refreshment intensity, using Assumption~\ref{ass:general} \ref{ass:refrboundedrate}.
Let $N$ denote a Poisson random variable with parameter $M t$ and let $\kref \in \N$ such that $\P(N \geq \kref) < \varepsilon/3$.  By Assumption~\ref{ass:general} \ref{ass:refrboundedrate} with probability at least $1-\varepsilon/3$ there will be fewer than $\kref$ jumps due to the refreshment kernels.
Using Assumption~\ref{ass:general} \ref{ass:refrboundedinprob}, let $R'$ be sufficiently large so that $r<R'$ and $\Qrefr(z,B_{R'}) > (1 - \varepsilon/3)^{1/\kref}$ for all $z \in B_{R'}$, so that with probability at least $1- \varepsilon/3$ all the observed refreshments map into $B_{R'}$.
Conditional on the event that all refreshments map into $B_{R'}$, since all reflections are norm preserving (using Assumption \ref{ass:general} \ref{ass:reflnormpreserving}) it follows from Assumption \ref{ass:general} \ref{ass:flowbound} that $\|Z_t\| \leq c_t(1+R')$.  Since $\lswitch$ is bounded on bounded sets (Assumption~\ref{ass:general} \ref{ass:lboundedonbounded}), it follows that there is a constant $\ksw$ such that with probability at least $1 - \varepsilon/3$ there are at most $\ksw$ reflection events.
We conclude that, for all $\varepsilon > 0$ and $t \geq 0$ we have 
\[ \P_{z_0} \left(T_{\ksw + \kref} > t \text{ and } Z_s \in B_{c_t(1+R')}  \ \text{for all $s \in [0,t]$}\right) > 1 - \varepsilon,\]
establishing the first statement of the lemma. The second statement follows from the first since for all $\varepsilon$ there exists $R > 0$ such that
\[ \P_{z_0}(T_{\mathrm{explode}} > t)\geq \P_{z_0}(T_{\mathrm{explode}} > t \text{ and } Z_t \in B_R) > 1 - \varepsilon\] where $t$ is arbitrary and  $\varepsilon$ can be chosen arbitrarily small. Moreover, as $R=c_t(1+R')$ is independent of the choice of $z_0$ in $B_r$ we have that
\[ \sup_{z_0\in B_r}\P_{z_0} \left( T_{\ksw + \kref} > t \right) > 1 - \varepsilon.\]
As $\kref$ and $\ksw$ are deterministic and depend only on $\varepsilon, r, R'$ for each $\varepsilon>0$ we may take $n$ sufficiently large that $n\geq \kref+\ksw$ which gives
\[\lim_{n\to\infty} \sup_{z_0\in B_r}\P_{z_0} \left( T_n > t \right) > 1 - \varepsilon.\]
Now since $\varepsilon$ was arbitrary we get the third statement of the lemma.
\end{proof}

Define the semigroup $\{\cP_t\}_{t\geq 0}$ on the space $B_b(\cZ)$, the set of all bounded and measurable functions $f:\cZ\to \R$ endowed with the supremum norm, by setting 
\begin{equation}\label{eq:semigroupdef}
\cP_tf(z) = \mathbb{E}_z[f(Z_t^z)].
\end{equation}

\begin{lemma}\label{lem:Feller}
	If Assumption \ref{ass:general} holds, then $\{\cP_t\}_{t\geq 0}$ defined by \eqref{eq:semigroupdef} is a Feller semigroup, i.e. $\cP_t(C_b(\cZ))\subseteq C_b(\cZ)$.
\end{lemma}

\begin{proof}[Proof of Lemma \ref{lem:Feller}]
	Note that $\cP_t$ is a contraction with respect to the supremum norm so it is sufficient to show that $\cP_tf$ is continuous for any $f\in C_b(\cZ)$. We will follow the strategy of \cite[Theorem 27.6]{DavisMarkov}. Fix $f\in C_b(\cZ)$ and define for any $g\in C_b([0,\infty)\times \cZ)$
	\begin{equation*}
	Gg(t,z) = \mathbb{E}_z[f(Z_t)\mathbbm{1}_{t<T_1} + g(t-T_1,Z_{T_1})\mathbbm{1}_{t\geq T^1}].
	\end{equation*}
	Here $\mathbbm{1}_A$ denotes the characteristic function on the set $A$.
	With this notation we can also express $Gg(t,z)$ as
	\begin{align*}
	Gg(t,z) = f(\varphi_t(z)) e^{-\int_0^t \lambda(\varphi_s(z))ds} + \int_0^t Qg(t-s,\cdot)(\varphi_s(z)) \lambda(\varphi_s(z))e^{-\int_0^s \lambda(\varphi_r(z))dr} ds.
	\end{align*}
	Notice that $Gg(t,x,v)$ is continuous in $(x,v)$ provided $f,g$ are both continuous. Following the proof of \cite[Lemma 27.3]{DavisMarkov} we have 
	\begin{equation*}
	\lvert G^ng(t,z)-\cP_tf(z)\rvert \leq 2c \mathbb{P}_z(T_n\leq t)
	\end{equation*}
	where $c = \max(\lVert f\rVert_\infty, \lVert g\rVert_{\infty})$. By Proposition \ref{lem:noexplosion} we have that $G^ng(t,z)$ converges to $\cP_tf(z)$ uniformly on bounded sets; therefore $\cP_tf(z)$ is continuous.
\end{proof}

 \subsection{Extended generator}
 
 For PDMPs the domain of the extended generator can be explicitly characterised as shown in \cite{DavisMarkov, Riedler}
 \begin{definition}
 	Let $\Dext$ denote the set of measurable functions $f:H\to \R$ with the property that there exists a measurable function $h:\cZ\to\R$ such that $t\mapsto h(Z_t)$ is integrable almost surely and the process
 	\begin{equation}
 	C_t^f=f(Z_t)-f(Z_0)-\int_0^th(Z_s)ds
 	\label{eq:defofCtf}\end{equation}
 	is a $\mathbb{P}_z$-local martingale. Then we set $\cL f=h$ and call $(\cL,\Dext)$ the extended generator.
 \end{definition}

\begin{theorem}\label{thm:extgen}
	Let $Z_t$ denote the PDMP with characteristics $(\fX,\lambda,Q)$ and assume that Assumption \ref{ass:general} holds. Then the domain of the extended generator $\Dext$ is given by all measurable functions $f$ such that
	\begin{enumerate}[(i)]
		\item For each $z\in \cZ$ the function $t\mapsto f(\varphi_t(z))$ is differentiable for almost every $t$;\label{item:diffoff}
		\item $(z,t,\omega)\mapsto f(z)-f(Z_{t-}(\omega))$ is a valid integrand for the compensating measure $\tilde{p}$. Here $\tilde{p}(t,A) = \int_0^t Q(Z_s,A)\lambda(Z_s)ds$.\label{item:integrabilityoff}
	\end{enumerate}
Then, for $f\in \Dext$ the extended generator is given by 
\begin{equation}\label{eq:extgen}
\cL f(z) = \LX f(z) + \lambda(z) \int_\cZ (f(y) - f(z)) Q(z,dy).
\end{equation}
\end{theorem}
This follows from the proof of \cite[Theorem 2.2.1]{Riedler}. We note that \cite{Riedler} assumes that $\lambda$ is bounded, however for characterising the extended generator they only use that $\lambda$ is bounded to ensure the process is well-defined. Since we have established the PDMP is well-defined the proof of \cite[Theorem 2.2.1]{Riedler} holds. From \cite[Remark 26.16]{DavisMarkov} if $f$ is bounded and measurable then condition (2) is satisfied. 

Recall the semigroup $\{\cP_t\}_{t\geq 0}$ defined by \eqref{eq:semigroupdef} defined on the space $B_b(\cZ)$. Note that $\{\cP_t\}_{t\geq 0}$ is a contraction semigroup however it need not be a strongly continuous semigroup, therefore we shall consider a smaller space on which $\{\cP_t\}_{t\geq 0}$ is strongly continuous. By Lemma \ref{lem:Feller} if Assumption \ref{ass:general} holds then $\{\cP_t\}_{t\geq 0 }$ is a Feller semigroup, i.e. for every $t\geq 0$ we have that $\cP_t(C_b(\cZ))\subseteq C_b(\cZ)$. Set $C_0(\cZ)=\{f\in C_b(\cZ): \lim_{t\to 0}\lVert \cP_tf-f\rVert_\infty=0\}$. Note this space is closed by the same arguments used in \cite[pg 28-29]{DavisMarkov} and as $\{\cP_t\}_{t\geq 0}$ is a Feller semigroup $\cP_t(C_0(\cZ))\subseteq C_0(\cZ)$ for every $t\geq 0$. On the space $C_0(\cZ)$ the semigroup $\{\cP_t\}_{t\geq 0}$ is strongly continuous. Let $(\cL,\DCo)$ be the generator of $\{\cP_t\}_{t\geq 0}$ on the space $C_0(\cZ)$. We can see from the definition of the extended generator that $\cL$ is the restriction\footnote{Let $f\in \DCo$ and set $C_t^f$  to be defined as in \eqref{eq:defofCtf}. Then we have
\begin{equation*}
    \mathbb{E}[C_t^f\rvert \mathscr{F}_s] = \cP_{t-s}f(Z_s)-f(Z_0)-\int_s^t \cP_{r-s}\cL f(Z_s)dr-\int_0^s\cL f(Z_r)dr.
\end{equation*}
As $f\in \DCo$ we can write $\int_s^t \cP_{r-s}\cL f(Z_s)dr=\cP_{t-s}f(Z_s)-f(Z_s)$ which gives
\begin{equation*}
    \mathbb{E}[C_t^f\rvert \mathscr{F}_s] = f(Z_s)-f(Z_0)-\int_0^s\cL f(Z_r)dr=C_s^f.
\end{equation*}
Therefore $f\in \Dext$.} of the extended generator $L$ to the set $\DCo$.

\section{Invariant measure of PDMP}\label{sec:PDMPforsampling}

In this section we give conditions to ensure a PDMP has an invariant measure $\mu$. The measures we consider will be absolutely continuous with respect to a probability measure $\mu_0$ and we will assume there exists a lower semicontinuous function $\Phi:\cZ \to \R$ which is bounded from below such that
\begin{equation}\label{eq:defofmu}
    \frac{d\mu}{d\mu_0}(z) = \frac{e^{-\Phi(z)}}{\int e^{-\Phi(z)} \mu_0(dz)}.
\end{equation}

\begin{example}[Inverse problem for a diffusion coefficient]
	This example is based on \cite[Section 3.3]{Stuart}. Consider the inverse problem of determining the diffusion coefficient from observations of the solutions of the PDE
	\begin{eqnarray}\label{eq:PDEfordiffcoef}
	\frac{d}{dt}\left(\kappa(t)\frac{dv}{dt}\right)=0,\\
	v(0)=0, v(1)=1.
	\end{eqnarray}
	We make observations $\{v(t_k)\}_{k=1}^q$ subject to Gaussian measurement error. Write observations as
	$$
	y_j=v(t_j)+\eta_j, j=1,\ldots,q
	$$
	where $\eta_j$ are i.i.d. and distributed according to $N(0,1)$. To ensure that $\kappa$ is strictly positive we set $u(x)=\ln(\kappa(x))$ and view $u\in L^2((0,1))$. Now define $J_x:L^\infty((0,1)) \to \R$ by 
	$$
	J_t(w) = \int_0^t \exp(-w(z))dz.
	$$
	Then the solution of \eqref{eq:PDEfordiffcoef} may be written as
	\begin{equation*}
	v(x)=\frac{J_t(u)}{J_1(u)}.
	\end{equation*}
	Set
	\begin{equation*}
	\Phi(u) = \frac{1}{2}\sum_{j=1}^q\left\lvert y_j-\frac{J_{t_j}(u)}{J_1(u)}\right\rvert^2
	\end{equation*}
	
	Now we place a Gaussian prior $\pi_0$ on the space $L^2((0,1))$ with mean zero and covariance operator $\beta (-\Delta)^{-\alpha}$. Here we set the domain of $\Delta$ to be
\begin{equation*}
	\begin{cases}
	\Delta u =u'',
	D(\Delta) = \left\{u \in H^2((0,1)) : u(0)=u(1), \int_0^1 u(s)ds =0\right\}.
	\end{cases}
\end{equation*}
	From \cite[Lemma 6.25]{Stuart} we have that $\pi_0(C([0,1]))=1$ provided $\alpha>1/2$. Then by \cite[Theorem 3.4]{Stuart} the posterior distribution $\pi^y(du)=\mathbb{P}(du\lvert y)$ given the data $y$ is absolutely continuous with respect to $\pi_0$ and
	\begin{equation*}
	\frac{d\pi^y}{d\pi_0}(u) \propto \exp(-\Phi(u)).
	\end{equation*}
\end{example}


\begin{theorem}\label{thm:invmeas}
	Assume that Assumption \ref{ass:general} holds. Suppose $\mathcal Z$ is equipped with a probability measure $\mu_0$. Define a transition kernel $Q^*$ by
	\begin{equation}\label{eq:Qadjoint}
	\int_{\cZ} \int_\cZ g(y,z) Q^*(y,dz)\mu_0(dy) = \int_\cZ \int_\cZ g(y,z) Q(z,dy)\mu_0(dz) 
	\end{equation}
	for all $g\in C_b(\cZ\times \cZ)$. Let $\mu$ denote a probability measure satisfying $\mu\propto e^{-\Phi}\mu_0$. Assume that
	\begin{align}\label{eq:absinvmeasurecond}
	e^{\Phi(z)}\LX ^*(e^{-\Phi(\cdot)})(z)+\int_\cZ \lambda(y)e^{\Phi(z)-\Phi(y)}Q^*(z,dy) -\lambda(z) =0.
	\end{align}
 Here $\LX ^*$ denotes the adjoint of $(\LX, D(\LX)) $ as an operator acting on the space $L_{\mu_0}^2$. Then $\mu$ is formally an invariant measure for the semigroup $\cP_t$ in the sense that
	\begin{equation}\label{eq:Lmu=0onC0}
	\int_\cZ \cL f(z) \mu(dz) = 0
	\end{equation}
	for all $f\in D(\LX)\cap \DCo$.
\end{theorem} 

In order to verify that $\mu$ is an invariant measure it is necessary that \eqref{eq:Lmu=0onC0} holds for all $f$ in a core for $(\cL,\DCo)$. In general for PDMP it is extremely difficult to determine a core for the generator, in Section \ref{sec:core} we prove that sufficiently smooth functions are a core for the Boomerang Sampler and find the invariant measure. 

\begin{proof}[Proof of Theorem \ref{thm:invmeas}]

	Fix $f\in D(\LX)\cap \DCo$ then we have
\begin{align*}
	\int_\cZ Lf(z) \mu(dz) &= \int_\cZ f(z) \LX ^*(e^{-\Phi(\cdot)})(z) \mu_0(dz) + \int_\cZ\int_\cZ \lambda(z)f(y)Q(z,dy) e^{-\Phi(z)} \mu_0(dz) \\
	&\quad - \int_\cZ \lambda(z) f(z) e^{-\Phi(z)}\mu_0(dz)\\
	&= \int_\cZ f(z)\left[ e^{\Phi(z)}\LX ^*(e^{-\Phi(\cdot)})(z)+\int_\cZ \lambda(y)e^{\Phi(z)-\Phi(y)}Q^*(z,dy) -\lambda(z)\right] \mu(dz)\\
	&=0.
\end{align*}
\end{proof}

\begin{example}\label{ex:detrefl}
	Suppose the transition kernel $Q$ corresponds to a deterministic transition, i.e. $Q(z,dy) = \delta_{R(z)}(dy)$ for some invertible function $R$. Assume that
	\begin{enumerate}
		\item The measure $\mu_0$ is invariant under $R$, i.e. $\mu_0\circ R=\mu_0$;
		\item The mapping $\Phi$ is invariant under $R^{-1}$, i.e. $\Phi\circ R^{-1}=\Phi$;
		\item The map $z\mapsto e^{-\Phi(z)}$ belongs to the domain of $\LX ^*$, and 
		\begin{equation}\label{eq:detswitchrelcond}
		e^{\Phi}\LX ^*e^{-\Phi} + \lambda\circ R^{-1}-\lambda=0.
		\end{equation}
	\end{enumerate}
	Then $\mu(\cL f)=0$ for all $f\in \DCo$. Indeed, in this case $Q^*(y,dz) = \delta_{R(y)}$ and we can rewrite the left hand side of \eqref{eq:absinvmeasurecond} as
	\begin{equation*}
	 e^{\Phi(z)}\LX ^*(e^{-\Phi(\cdot)})(z)+\lambda(R^{-1}(z))e^{\Phi(z)-\Phi(R^{-1}(z))} -\lambda(z).
	\end{equation*}
	Using that $\Phi$ is invariant under $R^{-1}$ this simplifies to
	\begin{equation*}
	 e^{\Phi(z)}\LX ^*(e^{-\Phi(\cdot)})(z)+\lambda(R^{-1}(z)) -\lambda(z)
	\end{equation*}
	which is equal to zero by \eqref{eq:detswitchrelcond} therefore the measure $e^{-\Phi}\mu_0$ is an invariant measure.
\end{example}

\begin{example}[Finite dimensional Zig-Zag process]\label{ex:finitedimzigzag}
	We now consider a finite dimensional example, set $\cZ=\R^d\times\R^d$, and for any $z\in \R^d\times\R^d$ we write $z=(x,p)$ for some $x\in\R^d, v\in\R^d$. Suppose we wish to sample from the measure $\pi(dx)=e^{-\Phi(x)} dx$, which we extend to a measure $\mu$ on $\R^d\times\R^d$ by $\mu(dx,dv)=\pi(dx)\nu_0(dv)$ where $\nu_0$ is the uniform distribution on $\{1,-1\}^d$. Let
	\begin{equation*}
	\sX\left(\begin{array}{c}
	x\\v
	\end{array}\right) = \left(\begin{array}{c}
	v\\0
	\end{array}\right)
	\end{equation*}
	and set
	\begin{equation*}
	Q(z,dy) = \sum_{i=1}^d\frac{\lambda_i(z)}{\lambda(z)} \delta_{F_iz}
	\end{equation*}
	where $F_i(x,v)= (x,v_1,\ldots,v_{i-1},-v_i,v_{i+1},\ldots,v_d)$, $\lambda(z) = \sum_{i=1}^d \lambda_i(z)$ and $\lambda_i(x,v) = ( v_i \partial_i\Phi(x))^+$. Note that $\LX $ has an invariant measure $\mu_0=\nu_0(dv)dx$ which is not a probability measure. By \cite[Theorem 27.6]{DavisMarkov} the semigroup associated to this process is a Feller semigroup.
	
	For this example we have that
	\begin{equation*}
	Q^*(z,dy) = \sum_{i=1}^d\frac{\lambda_i(F_iz)}{\lambda(F_iz)} \delta_{F_iz}
	\end{equation*}
	in which case, showing that \eqref{eq:absinvmeasurecond} holds reduces to
	\begin{equation*}
	    \partial_i\Phi(x) v_i+ \lambda_i(F_iz) -\lambda_i(z)=0.
	\end{equation*}
	It is immediate to check that this holds for our choice of $\lambda_i$.
\end{example}
\section{Examples of infinite dimensional PDMPs}\label{sec:PDMP-examples}
\subsection{Zig-Zag Sampler in infinite dimensional Hilbert spaces}\label{sec:ZZ}

In this section we present an infinite dimensional version of the Zig-Zag Sampler, which we refer to as IDZZ,  evolving in the Hilbert space $(\h, \langle \cdot, \cdot \rangle)$. This algorithm is designed to target the measure $\pi$ as defined in \eqref{targetmeasure}, however as in the finite dimensional ZZ we extend the space to be of the form $(x,v)$. That is, we set the state space to be
\begin{equation}\label{defofV}
\cZ = \mathcal H \times\cVZZ
\end{equation}
where
\begin{equation*}
\cVZZ=\left\{ v \in \cH : 
 \C^{-1}v\in \h \right\},
\end{equation*}
and we recall (from Section~\ref{sec:notation}) that $\Sigma$ is a self-adjoint trace class operator on $\mathcal H$. We equip $\cVZZ$ with the inner product $(v,w) \mapsto \langle \C^{-1} v, \C^{-1} w \rangle$, making it into a Hilbert space.
Note that in order to ensure well-posedness of the IDZZ we require that the auxiliary variable $v$ belongs to a smoother space than $\h$, indeed the space $\cVZZ$ is also contained within the Cameron-Martin space of $\pi_0$ and hence is a null set with respect to $\pi_0$. On this extended state space $\mathcal Z$ we design the IDZZ to have invariant measure $\mu$ defined by \eqref{eq:defofmu} where
\begin{equation}\label{def:muv}
 \mu_0= \pi_0 \times \nu_{a}, \,\,\, \mbox{ where } \,\,\, \pi_0 \sim \mathcal N(0,\Sigma), \,\, 
\nu_{a} = \bigotimes_{i=1}^{\infty} (\tfrac 1 2  \delta_{-a_i} + \tfrac 1 2 \delta_{a_i}) \, , a\in \cVZZ.
\end{equation}
Note here that different choices of $a$ can lead to different invariant measures so the process we define is not ergodic on the space $\h\times\cVZZ$. This could be resolved by restricting $\cVZZ$ to the space
\begin{equation*}
\{v\in \h: v=(v_i)_{i=1}^\infty,  v_i\in \{a_i,-a_i\}\}
\end{equation*}
for some fixed $a=(a_i)_{i=1}^\infty\in \cVZZ$. However to remain consistent with the framework of this paper it is more convenient to have that $\cVZZ$ is a Hilbert space.

In this section we will make the following assumptions on the functional $\Phi$. Suppose $\Phi$ is everywhere Fr\'echet differentiable. For each $x \in \h$ the Fr\'echet derivative $d\Phi(x)$ is identified with an element $\nabla \Phi(x)$ of $\h$.
\begin{assumption} \label{ass:1}
The functional $\Phi$ satisfies the following:
\begin{enumerate}
\item {\bf Domain of $\Phi$:} the functional $\Phi$ is defined everywhere on $\h$ and 
$$
0\leq \Phi(x) \les 1+ \lVert x\rVert^2.
$$
\item {\bf Derivatives of $\Phi$:} \label{item:derivativesofPhi}
The function $\Phi$ is differentiable  and $\nabla_x\Phi$  is locally (on bounded sets) Lipschitz and grows at most linearly:
\begin{equation}\label{eq:C2}
\|{ \nabla_x \Phi(x)}\| \lesssim 1+\lVert x\rVert \,.
\end{equation}
\item {\bf Growth of $\Phi$:} The functional $\Phi$ is can be written as $\Phi=\Phi_1+\Phi_2$ where $\Phi_1$ is convex and bounded from below, and $\Phi_2$ is bounded with bounded first and second order derivatives. 
\end{enumerate}
\end{assumption}

The generator of the IDZZ is then defined as follows:
\begin{equation} \label{eq:zigzaggenerator-noniid}
 \cL f(x,v) = \langle \nabla f(x), v \rangle + \sum_{i=1}^{\infty} \lambda_i(x,v) \left( f(x,F_i v) - f(x,v) \right), \quad (x,v) \in \h \times \cVZZ
\end{equation}
where $f$ is in a class of suitably smooth functions. The switching intensities are given by
\[ \lambda_i(x,v) = \left( v_i \left(\frac {x_i} {\gamma_i^2} + \partial_i \Phi(x) \right) \right)^+.\]

In terms of the abstract framework in Section~\ref{sec:construction}, we have 
\begin{align*}&\lambda(x,v) = \sum_{i=1}^\infty \lambda_i(x,v), \quad  Q_i(x,v,dw) = \delta_{F_i v}(dw),  \\ &Q(x,v,dy,dw)=\frac{1}{\lambda(x,v)}\sum_{i=1}^\infty \lambda_i(x,v)\delta_x(dy) Q_i(x,v,dw). 
\end{align*}

\begin{remark}
	As $\lambda$ is defined by an infinite series we must verify that it is well defined for all $(x,v)\in \h\times\cVZZ$. Note that 
	\begin{align}\label{eq:boundforlambdaZZ}
	\sum_{i=1}^{\infty} \lambda_i(x,v) & \leq \sum_{i=1}^{\infty}  \frac{\lvert v_i\rvert}{\gamma_i^2} |x_i| + \sum_{i=1}^{\infty} \lvert v_i\rvert |\partial_i \Phi(x) | \leq \lVert \C^{-1} v\rVert \lVert x\rVert + \lVert v\rVert \lVert \nabla\Phi(x)\rVert.
	\end{align}
	Therefore $\lambda$ is well defined.
\end{remark}


%
%

\begin{theorem}\label{thm:ZZnonGausscase}
Let Assumption \ref{ass:1}  hold. Then there exists a piecewise deterministic Markov process in $\h \times \mathcal V$ with (formal) generator~\eqref{eq:zigzaggenerator-noniid}. Moreover, for any $a\in \cVZZ$ the product measure $\mu$ given by \eqref{eq:defofmu} with $\mu_0$ given by \eqref{def:muv} is invariant for such a process.  
\end{theorem}


\begin{proof}
    The proof is deferred to Section \ref{sec:proofs_zzs}.
\end{proof}

\subsubsection{Choice of velocities}

In the following example we discuss how to choose optimal velocity magnitudes $(a_i)$ for the infinite dimensional zig-zag process, in order to minimize the computational complexity for estimation of the functional $x \mapsto \|x \|_r^2$, where for $r \ge 0$
\[ \|x\|_r^2 := \langle \Sigma^{-r} x, x \rangle.\]
I.e., we will derive below the limiting normal distribution of
\[ \frac 1 {\sqrt T} \int_0^T \left\{ \|X(s)\|_r^2 - \pi_0( \|\cdot\|_r^2)\right\}\, d s,  \]
as $T \rightarrow \infty$, depending on the choice of $(a_i)$,
and determine how to choose $(a_i)$ in order to minimize the asymptotic mean squared error given a fixed computational budget (i.e., for a fixed maximum number of switches).

Let $\sigma^2_f$ denote the asymptotic variance for a one-dimensional canonical (i.e., no refreshments) Zig-Zag process with position $(X(t))_{t \geq 0}$ and velocities $V(t) \in \{-1,1\}$ targetting the standard normal distribution, evaluated with respect to the function $f$, i.e.
\[ \frac 1 {\sqrt{T}} \int_0^T  \left\{ f(X(s)) \, ds - \pi(f) \right\} \, d s \stackrel{d}{\longrightarrow} \mathcal N(0,\sigma^2_f),\]
where $\pi$ denotes the standard normal distribution. Write $\pi_{\gamma}$ for the distribution $\mathcal N(0,\gamma^2)$.

\begin{lemma}
\label{lem:asymptoticvariance}
A one-dimensional Zig-Zag process with velocities $\{-a,a\}$ with stationary distribution $\mu_{\gamma,a} = \pi_{\gamma} \otimes \text{Uniform}(\{-a,+a\})$ and canonical switching intensities, evaluated with respect to the function $\widetilde f(x) = c f(x/\gamma)$ has asymptotic variance  $c^2 \gamma \sigma_f^2 / a$. The expected number of switches compared to the canonical Zig-Zag process with unit speed velocities is proportional to $a/\gamma$ (per unit time, in stationarity).
\end{lemma}

\begin{proof}
Let $\widetilde X(t)$ denote the Zig-Zag process as specified. The stationary distribution of $\widetilde X(t)/\gamma$ is then a standard normal distribution, and the process $\widetilde X(t)/\gamma$ has velocities $\pm a/\gamma$. Slowing down the clock by a factor $a/\gamma$, the process $\Xi(t) := \widetilde X(\gamma t/a)/\gamma$ is a Zig-Zag process at unit speed targetting the standard normal distribution. We compute
\begin{align*}
    \frac 1 {\sqrt{T}} \int_0^T \{ c f(\widetilde X(s)/\gamma) - \E_{\pi_{\gamma}} f(\widetilde X/\gamma) \} \, d s &  = \frac 1 {\sqrt{T}} \int_0^T \{ c f(\Xi(a s/\gamma)) - c \pi(f) \} \, d s \\
    & = \frac {\gamma} {a \sqrt{T }} \int_0^{a T/\gamma} \{ c f(\Xi(r)) - c \pi(f) \} \, d r \\
    & = \frac {c \sqrt{\gamma}} {\sqrt{a}} \frac 1 {\sqrt{\widetilde T}} \int_0^{\widetilde T} \{  f(\Xi(r)) - \pi(f) \} \, d r \\
    & \stackrel{d}{\longrightarrow} \mathcal N(0, c^2 \gamma \sigma^2_f/a).
    \end{align*}
The number of switches follows directly from the slowing down factor.
\end{proof}

\begin{example}
\label{ex:asymptoticvariance-1dgaussian}
Define $\nu^2$ to be the asymptotic variance of the canonical Zig-Zag process with unit velocities targetting the standard normal distribution, evaluated with respect to the function $x \mapsto x^2$. By \cite[Example 1]{BierkensDuncan2016}, we have $\nu^2 = 4\sqrt{2/\pi}$, but the exact value is of no importance here. If we consider now a zigzag process targetting $\mathcal N(0,\gamma^2)$ with speeds $\{-a,a\}$ evaluated with respect to the function $x \mapsto x^2/\gamma^{2r} = c x^2/\gamma^2$ for $c = \gamma^{2(1-r)}$, we obtain from Lemma~\ref{lem:asymptoticvariance} that the associated asymptotic variance is $\gamma^{5 -4r} \nu^2/a$. 
\end{example}

We will now consider an infinite dimensional Zig-Zag process targeting the function $x \mapsto \|x\|_r^2$, for $r > 0$. To ensure that all the required operations are welldefined, we introduce the following assumption.

\begin{assumption}
\label{ass:functional-smoothness}
Suppose $r \in [0,\tfrac 1 2)$ is such that $\sum_{i=1}^{\infty} \gamma_i^{2-4r} < \infty$.
\end{assumption}

\begin{example}
In the case of a Wiener process in $L^2[0,1]$, we have $\gamma_i^2 = (i-1/2)^{-2} \pi^{-2} \sim i^{-2}$. Consequently, for any $r <1/4$ we have
$\sum_{i=1}^{\infty} \gamma_i^{2-4r} < \infty,$
so that Assumption~\ref{ass:functional-smoothness} is satisfied.
\end{example}

\begin{prop}
Suppose Assumption~\ref{ass:functional-smoothness} is satisfied. Consider the zigzag process of Theorem~\ref{thm:ZZnonGausscase} with $\Phi(x) = 0$ for all $x$. The asymptotic variance $\sigma_r^2$ of the Zig-Zag process with respect to the function $x \mapsto \|x\|^2_{r}$ is given by $\sigma^2_r = \nu^2\sum_{i=1}^{\infty} \gamma_i^{5-4r}/a_i$, where $\nu^2$ is as defined in Example~\ref{ex:asymptoticvariance-1dgaussian}. The expected number of switches per unit time interval is proportional to $\sum_{i=1}^{\infty} a_i/\gamma_i$.
\end{prop}

\begin{proof}
Note that the generator of the Zig-Zag process can be interpreted as a sum of infinitely many one-dimensional Zig-Zag processes.
By this factorization property and Example~\ref{ex:asymptoticvariance-1dgaussian}, the asymptotic variance decomposes as
\[ \sigma_r^2 = \sum_{i=1}^{\infty} \gamma_i^{5-4r} \nu^2/a_i.\]
The proportionality factor of the number of switches is a direct consequence of the second statement of Lemma~\ref{lem:asymptoticvariance}.
\end{proof}

Suppose now that we have a fixed computational budget of (approximately) $N$ operations, and we wish to minimize the approximate standard error $\epsilon_{r,T} = \sigma_r/\sqrt{T}$, where $T$ is the length of the time interval we will simulate within $N$ operations. We may then phrase the question of choosing optimal velocities $(a_i)$ as follows:
\[
    \text{minimize} \quad  \frac 1 T \nu^2 \sum_{i=1}^{\infty} \gamma_i^{5-4r}  \quad 
    \text{subject to} \quad  T \sum_{i=1}^{\infty} a_i/\gamma_i \leq N.
\]
Resolving the constraint by setting $T$ at its maximal possible value, $T = \frac{N}{\sum_{i=1}^{\infty} a_i/\gamma_i}$, the problem transforms into the minimization problem

\[ \text{minimize} \quad \frac{\nu^2}{N} \left(\sum_{i=1}^{\infty} \gamma_i^{5-4r}/a_i \right) \left(\sum_{i=1}^{\infty} a_i/\gamma_i \right).\]
Minimization with respect to $a_i$ (using a Cauchy-Schwarz argument) yields
\[ a_i = C\gamma_i^{3-2r}\] for some constant $C$, where the choice of the constant $C$ does not affect the minimization objective. 

We find that, for this choice of $(a_i)$, the expected number of switches is finite,
\[ N = T \sum_{i=1}^{\infty} a_i/\gamma_i = T \sum_{i=1}^{\infty} \gamma_i^{2-2r}  \le T \left( \max_j \gamma_j^{2r}\right) \sum_{i=1}^{\infty} \gamma_i^{2-4r}<\infty \]
by Assumption~\ref{ass:functional-smoothness}.
Furthermore the Zig-Zag process is well-defined using Theorem~\ref{thm:ZZnonGausscase} since
\[ \| \Sigma^{-1} v \|^2 = \sum_{i=1}^{\infty} \frac{a_i^2}{\gamma_i^4} = \sum_{i=1}^{\infty} \gamma_i^{2-4r} < \infty \]
by Assumption~\ref{ass:functional-smoothness}.

Note that, for large $r$, the functional $x \mapsto \|x\|_r^2$ is sensitive to the high frequency components of $x$ (i.e. the large indices in the expansion $x = \sum_{i=1}^{\infty} x_i e_i$). Also, for large $r$, the high frequency velocities $a_i = \gamma_i^{3-2r}$ are relatively larger, so that the high frequency components are explored more.

We have shown how, for a particular example of a functional $x \mapsto \|x\|_r^2$ relative to a Gaussian target we may choose velocities to minimize the statistical error when using the IDZZ as a sampling method, for a fixed computational budget. It is an interesting research problem, but beyond the scope of this work, how this may be extended to other functionals, general target distributions and other sampling methods.

\subsection{Bouncy Particle Sampler in infinite dimensional Hilbert spaces}\label{sec:BPS}

 In this section we introduce  an extension of the BP algorithm which is well-defined in the infinite-dimensional Hilbert space 
 \begin{equation*}
 \cZ=\cH\times \cVBP.
 \end{equation*} 
 As with IDZZ the velocity component needs to be smoother than the position component, set
 \begin{equation*}
     \cVBP = \{v\in \h: \C^{-1}v\in \h\}.
 \end{equation*}
 We will design the IDBP to have invariant measure $\mu$ defined by \eqref{eq:defofmu} where
\begin{equation}\label{def:muvBPS}
 \mu_0= \pi_0 \times \nu_0, \,\,\, \mbox{ where } \,\,\, \pi_0 \sim \mathcal N(0,\C), \,\, 
\nu_0 \sim \mathcal N(0,\C^{\zeta})
\end{equation}
and $\zeta\geq 2$.
Note that as $\zeta\geq 2$ then $\nu_0(\cVBP)=1$ therefore $\mu$ can be viewed as a probability measure on $\h\times\cVBP$.

 The generator of IDBP is defined, on a set of  suitably regular functions, as  
\beg\label{Linfbouncy}
\cL f(x,v) = \langle v, \nabla_x f(x,v) \rangle + \lambda(x,v) [ f(x,R(x) v) - f(x,v)]  + \lambda_{\refr} \int_{\cVBP}\left[f(x,w)-f(x,v)\right] \nu_0(dw), 
\ee
where 
\begin{equation}\label{paraminfbouncy}
\lambda(x,v) := \langle v, \nabla \Psi \rangle^+,  \quad R(x)v := v - \frac {2 \langle v,   \nabla \Psi(x) \rangle}{\|\C^{\zeta/2} \nabla \Psi(x) \|^2}  \C^{\zeta} \nabla \Psi(x),
\end{equation}
$\lambda_{\refr}>0$ is a positive constant, $\Psi$ is been defined by
\begin{equation}\label{def:Psi}
\Psi(x) := \Phi(x)+\frac{1}{2}\langle\C^{-1}x,x\rangle \,.
\end{equation}
\begin{remark}\label{note:psinotdefined}
Note that $\Psi(x)$ is not well defined for all $x\in \h$ since $\C^{-1}$ is an unbounded operator however this is to be viewed as just a formal definition. Indeed, the IDBPS does not rely on $\Psi$ itself but only on $\nabla\Psi$, which always appears with a term to smooth it. In Proposition \ref{propIDBP} we show that all the terms that appear are well defined for all $x\in \h$ and $v\in \cVBP$.
\end{remark}

We will show in Proposition \ref{propIDBP} that the resulting algorithm generates a continuous-time dynamics $(X_t,V_t) \in \h \times \cVBP$ with invariant distribution  $\mu$. Here well-posedness is in the sense of Proposition \ref{lem:choiceofalpha}  and Proposition \ref{propIDBP}  below. 


\begin{prop}\label{lem:choiceofalpha}
 Let Assumption \ref{ass:1} hold.  For any $\zeta \geq 4$ 
 and for every $x \in \h$, $v\in \cVBP$, we have
\begin{enumerate}
\item the intensity $\lambda(x,v)$ is well defined (in the sense that
  the scalar product  $\langle v,\nabla \Psi(x)\rangle$ is finite);
\item $R(x)v \in \cVBP$.
\end{enumerate}
\end{prop}

\begin{proof}
    The proof is deferred to Section \ref{sec:proofs_bps}.
\end{proof}

\begin{remark}
To explain the choice of scaling of the covariance operator in \eqref{Linfbouncy}-\eqref{paraminfbouncy} and at the same time show properties of IDBP, let us consider the following (formal) generator
\bee\label{Linfbouncyp}
\cL f(x,v) = \langle v, \nabla_x f(x,v) \rangle + \lambda(x,v) [ f(x,R(x) v) - f(x,v)]  + \lambda_{ref} \left[(Qf)(x,v)-f(x,v)\right], 
\eee
with 
\bee\label{paraminfbouncyp}
\lambda(x,v) := \langle \C^{\eta} v, \nabla \Psi \rangle^+,  \quad R(x)v := v - \frac {2 \langle v,  \C^{\epsilon} \nabla \Psi \rangle}{\| \C^{\gamma} \nabla \Psi \|^2}  \C^{\beta} \nabla \Psi,
\eee
where $\delta, \eta, \epsilon, \gamma, \beta$ are arbitrary positive parameters. Then in order to have that $R(x)$ is an involution, the measure $\nu_0$ is invariant under $R(x)$ and that $\mu$ is an invariant measure it is necessary that $\beta=2\gamma$, $\epsilon=\eta=0$, which gives \eqref{Linfbouncy}. 
\end{remark}

\begin{prop}\label{propIDBP}
With the notation introduced above, 
\begin{enumerate}
\item The reflection operator is involutive, i.e. 
$R(x)[R(x)v]=v$, for every $x\in \h,v \in \cVBP$; 
\item The reflection operator satisfies the following property
\bee\label{step2}
\lan  R(x)v, \nabla \Psi \ran = - \lan  v, \nabla \Psi \ran \,.
\eee
\item The Gaussian measure is invariant under reflections, i.e. a centered measure $\nu_0$ with covariance $\C^\zeta$ and the measure $\nu_R:=\nu_0 \circ R(x)$ coincide on $\h$; 
\item The measure $\mu$ is invariant for the operator $\cL$ in \eqref{Linfbouncyp}-\eqref{paraminfbouncyp}. 
\end{enumerate}
\end{prop}

\begin{proof}[Proof of Proposition \ref{propIDBP}]
    The proof is deferred to Section \ref{sec:proofs_bps}.
\end{proof}


\begin{remark}
    We have shown that if IDBPS is well posed then it has the correct invariant measure. However as commented in Note \ref{note:ZZSorBPSwellposed} the conditions given in Assumption \ref{ass:general} are not satisfied and hence we can not apply Theorem \ref{thm:normedspace} in this case. Therefore it remains to show that IDBPS is well posed, in particular one needs to show that it is not possible to have an infinite number of reflections in a finite time interval.  This problem can be alleviated by replacing the reflection operator $R(x)$ by $R(x) v= -v$, i.e. to introduce pure reflections in the IDBPS. Establishing wellposedness for $R(x)$ as given by~\eqref{paraminfbouncy} is beyond the scope of this work; however we will give here a heuristic argument to show why we can expect IDBPS to exist under Assumption \ref{ass:1}.

    We estimate the expected number of reflections per unit time by considering the rate in stationarity, that is we will bound
    \begin{equation*}
        \mathbb{E}_\mu[\lambda(X,V)] = \mathbb{E}_\mu[\langle V, \nabla \Phi(X)+\C^{-1}X \rangle^+] .
    \end{equation*}
    By Assumption \ref{ass:1} we have that $\nabla\Phi$ grows at most linearly and therefore we have the bound 
     \begin{equation*}
        \mathbb{E}_\mu[\lambda(X,V)] \leq \mathbb{E}_\mu[\lVert V\rVert (1+\lVert X\rVert)+\lVert \C^{-1}V\rVert \lVert X\rVert ] .
    \end{equation*}
    By independence of $X,V$ and using that $\C^{-1}V$ is Gaussian with covariance $\C^{\zeta-2}$, these terms are bounded. 
\end{remark}

\subsection{Boomerang Sampler}\label{sec:Boomerang}
	
	We now introduce the Boomerang Sampler which differs from IDZZ and IDBP by having circular deterministic dynamics instead of linear motion. 
As with IDZZ and IDBP we will extend the state space $\h$ to include a velocity component and we define the IDB on the infinite dimensional Hilbert space
\begin{equation*}
    \cZ=\h\times \h.
\end{equation*}
In contrast to the IDZZ and IDBP settings we may allow the velocity to belong to the same space as the position component. We design the IDB to have invariant measure given by \eqref{eq:defofmu} where $\mu_0=\pi_0\times \nu_0$ for some Gaussian measures $\pi_0=\mathcal{N}(0,\Sigma_x), \nu_0=\mathcal{N}(0,\Sv)$.
	
	
\subsubsection{Definition and Well Posedness of the Boomerang Sampler}\label{subsec:wellposed}
	
	The dynamics of the Boomerang Sampler are composed of three different mechanisms: the deterministic flow which is chosen to be a Hamiltonian flow, designed to preserve the reference measure $\mu_0$; reflections which will preserve the same Hamiltonian as the deterministic flow; and refreshment which will happen at a constant rate. First we describe the deterministic dynamics. Fix two bounded positive self -adjoint linear operators $\sX, \sP$ and take $\varphi$ to be the flow map associated to the Hamiltonian dynamics described by
		\begin{equation}\label{eq:Hamdyn}
		\frac{d}{dt}\varphi_t(x,v) = \left(\begin{array}{cc}
		0 & \sP\\
		-\sX & 0
		\end{array}\right)\varphi_t(x,v).
		\end{equation}
Recall that $\varphi_t$ takes values in $\cH\times \cH$ so we can view $\varphi_t(x,v)$ as two dimensional vector whose components belong to $\cH$. These dynamics correspond to the Hamiltonian 
		\begin{equation}
			\Ham(x,v):=\langle \sX x, x\rangle_{\cH} + \langle \sP v, v\rangle_{\cH}.   \label{eq:Hamdef}
		\end{equation}
		There are two natural choices for $\Ham(x,v)$, one is the the $\cH\times \cH$ norm which corresponds to taking $\sX=\sP=1$ and we will see that this requires that $\Sigma_x=\Sv$. Another natural choice is to take $\sX=\Sigma_x^{-1}$ and $\sP=\Sv^{-1}$ however the Hamiltonian is then only defined on the Cameron Martin Space which has measure zero with respect to $\mu$. This gives many technical difficulties and even well-posedness of the process may not hold (see Example \ref{ex:weaksoln})  so we shall restrict ourselves to the case where $\sX,\sP$ are bounded operators which implies that the Hamiltonian is everywhere defined.
		Define the operators 
		\begin{align}
		\LX  f(x,v) &= \left.\frac{d}{dt}\right\rvert_{t=0}f(\varphi_t(x,v))=\left\langle \left(\begin{array}{c}
		\sP v\\
		-\sX x
		\end{array}\right),\left(\begin{array}{c}
		\nabla_x f(x,v)\\
		\nabla_v f(x,v)
		\end{array}\right)\right\rangle_{\cH^2}\label{eq:L0def}\\
		\Lref f(x,v) &= \lref \int_{\cH} [f(x,w)-f(x,v)] \nu_0(dw)\label{eq:Lrefdef}\\
		\Lswitch f(x,v) &= \sum_{i=1}^\infty \lambda_i(x,v) [f(x,R_i(x)v)-f(x,v)].\label{eq:Lswdef}
		\end{align}
		Here $\lref$ is a positive constant, $\lambda_i:\cH\times \cH\to [0,\infty)$ for each $i\in \mathbb{N}$ and for each $i\in \mathbb{N}, x\in \cH$, $R_i(x)$ is a bounded linear operator. We formally summarise the Markov process $(X_t,V_t)$ by the generator
		\begin{equation*}
		\cL = \LX  +\Lref+\Lswitch.
		\end{equation*}
		Such a process exists by Proposition \ref{lem:noexplosion} provided the following Hypothesis \ref{hyp:WPass} holds. This process is a PDMP defined on $\mathcal Z=\cH\times \cH$ with characteristics 
		\begin{align}\label{eq:characteristics}
		\mathfrak X  & : 
		\begin{pmatrix}x \\ v \end{pmatrix}\mapsto \begin{pmatrix} 
		\sP v\\
		-\sX x
		\end{pmatrix}, \\
		\nonumber \lambda(x,v) & := \sum_{i=1}^{\infty} \lambda_i(x,v)+\lref, \\ 
		\nonumber Q((x,v), A)& :=\frac{1}{\sum_j\lambda_j(x,v)+\lref}\sum_{i=1}^\infty\lambda_i(x,v)\delta_{(x, R_i(x)v)}(A) + \frac{\lref}{\sum_j\lambda_j(x,v)+\lref}(\delta_x \times \nu_0)(A).
		\end{align}

	From now on we will assume the following hypothesis holds.
	\begin{hypothesis}\label{hyp:WPass}
		Assume that:
		\begin{enumerate}
			\item The intensities $\lambda_n$ and $\lambda$ are continuous and bounded by the Hamiltonian $\Ham(x,v)$ (recall this is defined by \eqref{eq:Hamdef}), that is
			\begin{align}
			\sum_{n=1}^\infty\lambda_n(x,v) &\leq C(1+\Ham(x,v)).\label{eq:lambdaupperbound}
			\end{align}
			\item The flow map $\varphi_t$ is given by the solution of \eqref{eq:Hamdyn} for some bounded linear operators $\sX, \sP$ which are self-adjoint and commute;
			\item Each reflection operator $R_i(x)$ is a linear bounded operator for any $x\in\cH$, and satisfies
			\begin{align}
			\langle \sP R_i(x)v, R_i(x)v\rangle_{\cH} \leq \langle \sP v, v\rangle_{\cH},& & \text{ for all } x,v\in \cH.\label{eq:RpreservesHam}
			\end{align}
			Moreover, assume that $(x,v)\mapsto R_i(x)v$ is continuous. 
		\end{enumerate}
	\end{hypothesis}

\begin{remark}\label{note:wellposed} We shall now comment on each of these assumptions in turn.
	\begin{enumerate}
		\item Since $\sX, \sP$ are bounded $\varphi$ is well defined and for each initial condition $(x,v)$ the map $t\mapsto \varphi_t(x,v)$ is smooth.
		\item Requiring that the reflections $R_i$ do not increase the Hamiltonian allows us to prove that the process is non-explosive, i.e. $\sup_i T^i =\infty$ almost surely. The strategy for showing this relies on the fact that except for refreshment jumps the dynamics is bounded by the Hamiltonian $\Ham(x,v)$ therefore the jump rates are bounded since \eqref{eq:lambdaupperbound} holds. 
		\item Under Hypothesis \ref{hyp:WPass} we have that Assumption \ref{ass:general} also holds, in particular the Boomerang Sampler is well-defined by Theorem \ref{thm:normedspace}. Since $\sP,\sX$ are bounded we can work with the seminorm $\Ham$ in which case Assumption \ref{ass:general} \ref{ass:reflnormpreserving} follows from \eqref{eq:RpreservesHam} and Assumption \ref{ass:general} \ref{ass:flowbound} is satisfied with $c$ constant since the flow $\varphi_t$ preserves the Hamiltonian. 
	\end{enumerate}
\end{remark}

	\begin{definition}\label{def:Boomerang}
		Given a Hilbert space $\cH$, a covariance operator $\Sigma_x$ and a function $\Phi$ we say that the process $\{X_t,V_t\}_{t\geq 0}$ is a {\em Boomerang Sampler with characteristics } $(\Sv,\sX,\sP,\{\lambda_i\}_i,\lref,\{R_i\}_i)$ if Hypothesis \ref{hyp:WPass} is satisfied and $(X_t,V_t)$ is the PDMP constructed in Section \ref{sec:construction} with characteristics given by \eqref{eq:characteristics}. The generator of this process acts on sufficiently smooth functions by
		\begin{align*}
		\cL f(x,v) &= \langle \sP v , \nabla_x f(x,v)\rangle_{\cH} -\langle \sX x , \nabla_v f(x,v)\rangle_{\cH} +  \lref \int_{\cH} f(x,w)-f(x,v) \, \nu_0(dw)\\
		& + \sum_{i=1}^\infty\lambda_i(x,v) [f(x,R_i(x)v)-f(x,v)]
		\end{align*} 
		
		We call the Boomerang Sampler with characteristics $(\Sigma_x,1,1,(\langle \nabla_x\Phi(x),v\rangle_{\cH})_+,\lref,v\mapsto-v)$ the {\em pure reflection Boomerang Sampler}, here $\lambda_i=0$ for $i>1$ so we have not included them in the notation.
	\end{definition}
	
The following example shows that if we drop the assumption that $\sX,\sP$ are bounded then we can still define a weak solution to the dynamics. For simplicity we will take $\sX=\sP=\Sigma_x^{-1}$.

\begin{example}\label{ex:weaksoln}
	Suppose that $\sX$ is an (unbounded) self-adjoint positive operator which is diagonalisable, that is there exists an orthonormal basis of eigenvectors. Then there is a process $(X_t,V_t)$ such that
	\begin{eqnarray}
	X_t-X_0&=\sX\int_0^t V_s \, ds\\
	V_t-V_0&=-\sX\int_0^t X_s \, ds.
	\end{eqnarray}
	We will construct this process by converting the problem into an infinite system of two dimensional ODEs.	Fix an orthonormal basis $\{e_k\}_{k=1}^{\infty}$ of eigenvectors of $\sX$ and  let $\gamma_k^2$ be the eigenvalue associated to $e_k$. 
		
		Consider the system
		\begin{align*}
		\frac{d}{dt}X_t^i &= \gamma_i^2V_t^i, \quad
		\frac{d}{dt}V_t^i = -\gamma_i^2X_t^i
		\end{align*}
		with initial condition $X_0^i=x^i, V_0^i=v^i$.
		The solution of this system is given by
		\begin{equation}\label{eq:componentwisesoln}
		\left(\begin{array}{c}
		X_t^i\\V_t^i
		\end{array}\right) = \left(\begin{array}{c}
		x\cos(\gamma_i^2 t) + v\sin(\gamma_i^2 t)\\-x\sin(\gamma_i^2 t)+v\cos(\gamma_i^2 t)
		\end{array}\right).
		\end{equation}
		We also have that
		$$\lvert X_t^i\rvert^2+ \lvert V_t^i\rvert^2 = \lvert x^i\rvert^2+ \lvert v^i\rvert^2 $$
		In particular, this gives us that $\sum_{i=1}^NX_t^i e_i$ and $\sum_{i=1}^NV_t^i e_i$ converge to some $X_t,V_t\in\cH$ for each $t\geq 0$. So we have
		\begin{equation*}
		X_t - X_0 = \lim_{N\to \infty} \sum_{i=1}^N (X_t^i-x^i) e_i = \lim_{N\to \infty} \sum_{i=1}^N \int_0^t  V_s^i \gamma_ie_i \, ds = \lim_{N\to \infty} \sX\left(\sum_{i=1}^N \int_0^t  V_s^i  e_i \, ds\right)
		\end{equation*}
		Since $\sX$ is a closed operator $ \int_0^t  V_s  \, ds, \int_0^t  X_s  \, ds \in D(\sX)$ and 
		$$
		X_t-X_0 = \sX \int_0^t V_sds, \quad 
		V_t-V_0 = \sX \int_0^t X_sds.
		$$
 Therefore we have weak solution to \eqref{eq:Hamdyn} but a weak solution is the most we can expect. To have a strong solution we require that $X_t\in D(\sX)$ and $V_t\in D(\sP)$. However $X_t,V_t\in D(\sX)$ if and only if
	$$
	\sum_{i=1}^N \gamma_i^2(\lvert X_t^i\rvert^2+\lvert V_t^i\rvert^2)<\infty.
	$$ 
	Using the expressions for $X_t^i,V_t^i$ given by \eqref{eq:componentwisesoln} we can rewrite this as
	$$
	\sum_{i=1}^N \gamma_i^2(\lvert X_t^i\rvert^2+\lvert V_t^i\rvert^2) = \sum_{i=1}^N \gamma_i^2(\lvert x^i\rvert^2+\lvert v^i\rvert^2).
	$$
	Therefore $X_t,V_t\in D(\sX)$ if and only if $x,v\in D(\sX)$. Now as $\sX$ is a closed densely defined operator we have that $D(\sX)=\cH$ if and only if $\sX$ is bounded.
\end{example}

\subsubsection{Invariant measure}

In the previous section we constructed the $\PDMP$ and showed that the construction was well defined for all $t\geq 0$, now we show that it has the desired invariant measure. Recall the definition of the semigroup was given by \eqref{eq:semigroupdef}, here we take $\cZ=\cH\times\cH$ and $Z_t=(X_t,V_t)$.

	\begin{hypothesis}\label{hyp:IMass}
	We shall assume that:
	\begin{enumerate}
		\item Hypothesis \ref{hyp:WPass} holds;
		\item The operators $\sX,\sP$ satisfy the consistency condition:
		\begin{equation}\label{eq:cons}
		\sP \Sv = \Sigma_x\sX.
		\end{equation}
		\item The function $\Phi$ is satisfies \ref{ass:1} and the following relation is satisfied:
		\begin{equation}\label{eq:switchrel}
		\sum_{n=1}^\infty(\lambda_n(x,v)-\lambda_n(x,R_n(x)v))= \langle \nabla_x\Phi(x), \sP v\rangle_\cH;
		\end{equation}
		\item The reflection operators $R_i(x)$ are involutive (i.e. $R_i(x)R_i(x)v=v$ for all $x,v\in\cH$) and invariant under the measure $\nu_0$, that is
		\begin{align}\label{eq:Rinvnu0}
		\int_{\cH} f(R_i(x)v) \nu_0(dv) = \int_{\cH} f(v) \nu_0(dv),& & \text{ for all } f\in L_{\nu_0}^2, x\in \cH, i\geq 1.
		\end{align}
	\end{enumerate}
\end{hypothesis}

\begin{prop}\label{prop:IM}
	Fix a Hilbert space $\cH$, a covariance operator $\Sigma_x$ and a potential function $\Phi$ and let $\{X_t,V_t\}_{t\geq 0}$ be a Boomerang Sampler with characteristics $(\Sv,\sX,\sP,\{\lambda_i\}_i,\lref,\{R_i\}_i)$. Assume that Hypothesis \ref{hyp:IMass} holds then $\mu$ is formally an invariant measure for $\{X_t,V_t\}_{t\geq 0}$ in the sense that \eqref{eq:Lmu=0onC0} for all $f$ sufficiently smooth.
\end{prop}

\begin{proof}[Proof of Proposition \ref{prop:IM}]
    The proof is deferred to Section \ref{sec:proofs_Boomerang}.
\end{proof}

\subsubsection{Pure Reflection Boomerang Sampler}\label{sec:Boomerang_PR}

In this section we take $\lambda_n=0$ for $n>1$ and consider choices of $\lambda=\lambda_1$ and $R=R_1$. Observe that we may always choose $R(x)v=-v$, due to the symmetry of $\nu_0$ and the Hamiltonian $\Ham$ we have that this choice preserves both $\nu_0$ and $\Ham$. With this choice of $R$ we must take $\lambda$ to be
\begin{equation*}
\lambda(x,v) = (\langle \nabla_x\Phi(x), \sP v\rangle_{\cH})^+ + \gamma(x,v)
\end{equation*}
where $\gamma(x,v)=\gamma(x,-v)$ for all $x,v\in \cH$ and is non-negative. Let us consider another choice of $R(x)$. To ensure that $R(x)$ preserves $\nu_0$ it is sufficient to test \eqref{eq:Rinvnu0} for any $f_q$ of the form $f_q(v)=\exp(i\langle v,q\rangle_{\cH})$ for every $q\in \cH$. In which case, $f_q(R(x)v) =\exp(i\langle R(x)v,q\rangle_{\cH}) = f_{R(x)^*q}(v)$, i.e.
\begin{equation*}
\hat{\nu_0}(q) = \hat{\nu_0}(R(x)^*q)
\end{equation*}
where $\hat{\nu_0}$ denotes the Fourier transform of $\nu_0$. Using that $\nu_0$ is normally distributed with covariance operator $\Sv$ we can rewrite this as
\begin{equation*}
\exp\left(-\frac{1}{2}\langle \Sv q,q\rangle_{\cH}\right) = \exp\left(-\frac{1}{2}\langle \Sv R(x)^*q,R(x)^*q\rangle_{\cH}\right).
\end{equation*}
Therefore the measure $\nu_0$ is invariant under $R(x)$ if and only if
\begin{equation}\label{eq:invarianceforreflcond}
\langle \Sv q,q\rangle_{\cH} = \langle R(x)\Sv R(x)^* q,q\rangle_{\cH}
\end{equation}

\begin{remark}
	In finite dimensions we may take 
\begin{equation}\label{eq:findimref}
R(x)v = v-2\frac{\langle \nabla_x \Phi(x), \sP v\rangle}{\langle \Sv \sP\nabla_x\Phi(x), \sP\nabla_x\Phi(x)\rangle} \Sv\sP\nabla_x\Phi(x).
\end{equation}
Indeed, for this choice we have
$$
R(x)^*v = v -2 \frac{\langle \Sv\sP \nabla_x\Phi(x),v\rangle}{\langle \Sv \sP\nabla_x\Phi(x), \sP\nabla_x\Phi(x)\rangle} \sP\nabla_x\Phi(x).
$$
Now we can check that \eqref{eq:invarianceforreflcond} holds.
\begin{align*}
\langle \Sv R(x)^* v,R(x)^*v\rangle_{\cH} &= \left\langle \Sv v -2 \frac{\langle \Sv \sP\nabla_x\Phi(x), v\rangle}{\langle \Sv\sP \nabla_x\Phi(x),\sP\nabla_x\Phi(x)\rangle}\Sv  \sP\nabla_x\Phi(x), \right. \\
&\left. v -2 \frac{\langle \Sv \sP\nabla_x\Phi(x), v\rangle}{\langle \Sv\sP \nabla_x\Phi(x), \sP\nabla_x\Phi(x)\rangle} \sP\nabla_x\Phi(x)\right\rangle_{\cH}\\
&=\langle \Sv v, v\rangle_{\cH} -4 \frac{\langle \Sv\sP \nabla_x\Phi(x), v\rangle}{\langle \Sv\sP \nabla_x\Phi(x), \sP\nabla_x\Phi(x)\rangle}\langle\Sv  \sP\nabla_x\Phi(x), v\rangle_{\cH}\\
& +4 \frac{\langle \Sv\sP \nabla_x\Phi(x), v\rangle^2}{\langle \Sv \sP\nabla_x\Phi(x),\sP \nabla_x\Phi(x)\rangle^2} \langle\Sv\sP\nabla_x\Phi(x), \sP\nabla_x\Phi(x)\rangle_{\cH}\\
&=\langle \Sv v, v\rangle_{\cH}.
\end{align*}
However in infinite dimensional space \eqref{eq:RpreservesHam} does not hold. Instead \eqref{eq:RpreservesHam} holds in finite dimensions when the inner product is taken to be the Cameron Martin space inner product which is equivalent to the original inner product since in this case $\Sigma_x^{-1}, \Sv^{-1}$ are bounded.
\end{remark}

\subsubsection{Factorised Boomerang Sampler}\label{sec:Boomerang_factorised}

For this section we consider an alternate choice for $\lambda_n,R_n$. As in the pure reflection Boomerang Sampler we take $\sX=\sP=1, \Sigma_x=\Sv=\Sigma$. Let $\{e_i\}_{i=1}^\infty$ be an orthonormal basis of eigenvectors of $\Sigma$ and set $R_i(x)v= v-2\langle v,e_i\rangle_{\cH}e_i$, and set $$\lambda_i(x,v)=(v_{i}\partial_{x_i}\Phi(x))^++\gamma_i(x,v)$$ where $v_i=\langle v,e_i\rangle_{\cH}$ and $\gamma:\h\times\h\to[0,\infty)$ satisfies $\gamma_i(x,v)=\gamma_i(x,R_iv)$. 

Assume that $\Phi\in L^1_{\pi_0}$ is smooth and $\nabla_x\Phi$ is globally Lipschitz then Hypothesis \ref{hyp:IMass} is satisfied and we have that $\mu$ is an invariant measure by Proposition \ref{prop:IM}.

\section{Core for the generator of Boomerang Sampler}\label{sec:core}

In this section we will give conditions under which we have a core for the generator $\cL$ in $L_\mu^2$. Even in finite dimensional situations it is difficult to determine when a set is core for the generator of a PDMP, conditions are given in \cite{durmus2018piecewise} for the set $C_c^1$ to be a core in $C_0(\R^d)$. In particular they verify these conditions for BPS and in \cite{andrieu2019peskun} they verify the conditions of ZZS. In both cases, the strategy is to first consider the case of a smooth intensity rate $\lambda$ and show that $\cP_t(C_c^1(\R^d))\subseteq C_c^1(\R^d)$ then extend these results to the canonical intensity rate. See \cite[Section 4.1]{ADW} for a discussion of cores for PDMP in a finite dimensional setting. 

For the remaining sections we will only consider the Boomerang Sampler and will concentrate on the Pure Reflection Boomerang Sampler and Factorised Boomerang Sampler. Let us first introduce a framework which will include both Pure Reflection Boomerang Sampler and Factorised Boomerang Sampler. For this section we restrict to the Boomerang Sampler with characteristics $(\C,1,1,\{\lambda_i\}_i,\lref,\{R_i\}_i)$ where 
\begin{align} \label{eq:lambdaBSframework}
    \lambda_i(x,v)&=\frac{1}{2}(\langle\nabla_x\Phi(x),v-R_nv\rangle_{\cH})^+.
\end{align}
Here $R_i$ is a linear operator on $\h$, with $R_n^2=1$, $\lVert R_nv\rVert=\lVert v\rVert$ for any $v\in \cH$, satisfies \eqref{eq:invarianceforreflcond} and 
\begin{equation}
    \sum_{n=1}^\infty \lVert (1-R_n)v\rVert^2=4\lVert v\rVert^2. \label{eq:Rassump4}
\end{equation}
In this setting the generator acts on smooth functions by
\begin{align}
\cL f(x,v) &= \langle v,\nabla_xf(x,v)\rangle_{\cH} - \langle x,\nabla_vf(x,v)\rangle_{\cH} + \lref \int_{\cH} [f(x,w)-f(x,v)]\nu_0(dw) \nonumber\\
&+ \sum_{n=1}^\infty\lambda_n(x,v)[f(x,R_nv)-f(x,v)].\label{eq:simgen}
\end{align}

Observe that here $\lambda$ is the canonical rate, that is there is no $\gamma_i$ terms, this is to simplify the exposition however the same results hold for any bounded function $\gamma_i$ with $\gamma(x,R_i(v))=\gamma_i(x,v)$.

\begin{lemma}\label{lem:wellposedframework}
    Suppose that $\Phi$ satisfies Assumption \ref{ass:1}. Let $(X_t,V_t)$ be the Boomerang Sampler with  characteristics $(\C,1,1,\{\lambda_i\}_i,\lref,\{R_i\}_i)$ where $\lambda_i$ are given by \eqref{eq:lambdaBSframework}, and let $R_i$ satisfy the above conditions. Then Hypothesis \ref{hyp:IMass} holds, hence $(X_t,V_t)$ is well-defined.
\end{lemma}

In this section we aim to show the set $\cylfunccomp(\h\times\h)$ is a core for the generator $\cL$ as given by \eqref{eq:simgen}. The proof that $\cylfunccomp(\h\times\h)$ is a core is split into 4 steps:
\begin{itemize}
        \item \textbf{ Step 1: }We consider a PDMP $(\tilde{X}_t,\tilde{V}_t)$ with smooth intensities and no refreshments i.e. $(\tilde{X}_t,\tilde{V}_t)$ is a Boomerang Sampler with characteristics $(\C,1,1,\{\lt_i\}_i,0,\{R_i\}_i)$
where
\begin{align}
    \tilde{\lambda}_n(x,v) &= -\log(\phi(\exp(-\langle\nabla_x\Psi(x),v-R_nv\rangle_{\cH}))), \quad \phi(r) = r(1+r)^{-1}.\label{eq:smoothlambda}
\end{align}
It is a simple calculation to verify that Hypothesis \ref{hyp:WPass} is satisfied so we have that $(\tilde{X}_t,\tilde{V}_t)$ admits a strongly continuous semigroup $\cPt_t$ on the space $C_0(\h\times\h)$ and with generator given by $(\cLt,D(\cLt))$. Moreover for sufficiently smooth $f\in D(\cLt)$ we have
\begin{equation}\label{eq:smoothenedgen}
    \cLt f(x,v) = \LX f(x,v)+ \sum_{n=1}^\infty\tilde{\lambda}_n(x,v)[f(x,R_nv)-f(x,v)]. 
\end{equation} This choice of $\tilde{\lambda}_n$ is motivated by \cite{andrieu2019peskun} where they use an analogous choice to obtain a smooth and strictly positive intensity for a finite dimensional ZZS. The advantage of working with $(\tilde{X}_t,\tilde{V}_t)$ is that the semigroup is differentiable since the intensity is differentiable, and as there are no refreshments the process remains on a sphere in $\h\times\h$. For this PDMP we prove that the set 
\begin{equation}\label{eq:core}
	\core = \{f\in  C_b^1(\cH\times\cH):  \exists r>0 \text{ with } \mathrm{supp}(f)\subseteq B_r\}
	\end{equation}
	 is a core for $\cLt$, see Theorem \ref{thm:core}.
	 
	 \item \textbf{ Step 2: } We show that since $\core$ is a core for $(\cLt,D(\cLt))$ we also have that $\core$ is a core for $(\cL,\DCo)$. This is proven in Corollary \ref{cor:core} and the proof follows from showing that the difference between $\cLt$ and $\cL$ is a bounded operator.
	 
	 \item \textbf{ Step 3: } Prove that $\mu$ is an invariant measure for $(X_t,V_t)$ and that we can extend $\cP_t$ to a strongly continuous semigroup (which as an abuse of notation we still denote as $\cP_t$) on $L_\mu^2$, and we denote its generator by $(\cL, \DL)$.
	 
	 \item \textbf{ Step 4: } We show that $\cylfunccomp(\h\times\h)$ is a core for $(\cL, \DL)$.
	 
\end{itemize}
	

\begin{theorem}\label{thm:smoothcore}
    Let $\{\cPt_t\}$ be the semigroup corresponding to the generator \eqref{eq:smoothenedgen}, assume that $\Phi$ satisfies Assumption \ref{ass:1} and that $\phi$ is twice continuously differentiable with $\nabla_x^2\Phi$ bounded on bounded sets. 
    For any $f\in C_b^1(\cH\times\cH)$ with $\mathrm{supp}(f)\subseteq B_r$ for some $r>0$ then there exists a constant $C(r)$ which depends on $r$ such that
\begin{align}\label{eq:derivativebound}
	  \lVert\nabla \cPt_tf\rVert_{\infty}& \leq \lVert\nabla f \rVert_{\infty}+\lVert f\rVert_{\infty}C(r)t^{\frac{1}{2}}.
	\end{align}
	Moreover, $\core$ is a core for $(\cLt,D(\cLt))$.
\end{theorem}

\begin{proof}[Proof of Theorem \ref{thm:smoothcore}]
    First we show that \eqref{eq:derivativebound} holds. Fix $f\in \core$. In order to prove that $\cPt_tf$ is differentiable we will condition on the number of jumps. We can then express the law of $(X_t,V_t)$ explicitly and differentiate this directly. For this approach we need to be able to differentiate the intensity rate $\lambda_i$ and the reflection operator $R_iv$. For this reason we replaced $\lambda_i$ with the differentiable rate $\lt_i$. Since $R_i$ and $\fX$ are linear operators we have that $(x,v)\mapsto R_iv$ and $(x,v)\mapsto\varphi_t(x,v)$ are both differentiable. Define the operator $\cB_i$ to be the reflection operation and $\cB_i'$ as the derivative of $\cB_i$, that is,
    \begin{align*}
        \cB_i g(t,x,v)= g(t,x,R_iv)\\
        \cB'_i = \left(\begin{array}{cc}
        1 & 0\\
        0 & R_i
\end{array}\right).
\end{align*}
Recall we denote by $T_i$ the $i$-th event time. We shall set $J_i=n$ if at $T_i$ the velocity changes according to $R_n$. Note that the probability of more than one event occurs at some time $T_i$ is zero so we don't allow this event to occur.
Define the processes $\xi_t$, $C_t$ by
\begin{align}\label{eq:defofxi}
\xi_t&=\sum_{m=1}^{N_t}C_{T_m}\nabla \log\left(\lambda_{J_m}\right)(X_{T_m},V_{T_m-})-\sum_{i=1}^\infty\left(\int_0^tC_s \nabla_x\lt_i(X_s,V_s)ds\right)  \\
C_t&=\nabla \varphi_{T_1}(x,v)\cB_{J_1}'\cdots \nabla \varphi_{T_{N_t}-T_{N_t-1}}(X_{T_{N_t-1}},V_{T_{N_t-1}})\cB_{J_{N_t}}'\nabla\varphi_{t-T_{N_t}}(X_{T_{N_t}},V_{T_{N_t}})\label{eq:defofC}
\end{align}
and note that $C_t$ is a random orthogonal matrix. Note here if $t<T_1$ then we set $C_t=\nabla \varphi_{t}(x,v)$. With this notation we can write the derivative of the semigroup as
	\begin{align}\label{eq:derivativeofBS}
\nabla \cPt_tf(x,v)&=\mathbb{E}[C_t\nabla f(X_t,V_t)]+\mathbb{E}\left[\xi_t f(X_t,V_t)\right].
\end{align}
The proof of \eqref{eq:derivativeofBS} is deferred to Proposition \ref{prop:derformula}. 
Observe that since $C_t$ is an orthogonal matrix we can bound the first term by the supremum-norm of $\nabla f$. Therefore to show \ref{eq:derivativebound} it remains to find a bound for $\xi_t$, i.e. it suffices to show
\begin{equation}\label{eq:boundonxi}
    \sup_{x,v\in B_r}\mathbb{E}\left[\lVert \xi_t\rVert_{\cH\times\cH}^2\right] \leq C(r)^2t
\end{equation}
where $r>0$ is such that $\mathrm{supp}(f)\subseteq B_r$. The proof of this bound is deferred to Proposition \ref{prop:coresmoothcase}. 

In order to prove $\core$ is a core for $(\cLt,D(\cLt))$ it is sufficient to show that the semigroup $\cPt_t$ preserves $\core$ then by \cite[Proposition 3.3]{ethier} we have that $\core$ is a core for $\cLt$. We have that $\cPt_t$ preserves $C_b(\cH\times\cH)$ by Lemma \ref{lem:Feller}, and since there are no refreshments the process preserves the norm $\lVert\cdot\rVert_\cH$ hence the semigroup also preserves the property of having bounded support. It remains to show that $\cPt_tf\in C_b^1(\h\times\h)$ for any $f\in \core$. From \eqref{eq:derivativeofBS}, $C_t$ is orthogonal, the support of $f$ is contained in $B_r$, and Cauchy Schwartz we have
	\begin{align}\label{eq:estimatefornablaf}
	  \lVert\nabla \cPt_tf\rVert_{\infty}&\leq \lVert\nabla f \rVert_{\infty}+\lVert f\rVert_{\infty}\sup_{x,v\in B_r}\mathbb{E}_{x,v}\left[\lVert\xi_t\rVert_{\h\times\h}^2\right]^{\frac{1}{2}} \leq \lVert\nabla f \rVert_{\infty}+\lVert f\rVert_{\infty}C(r)t^{\frac{1}{2}}. 
	\end{align}
	Therefore $\core$ is a core for $(\cLt,\DCo)$.
\end{proof}

\begin{corollary}\label{cor:core}
    Suppose the assumptions of Theorem \ref{thm:smoothcore} hold. Then $\core$ is a core for $(\cL, \DCo)$.
\end{corollary}

\begin{proof}
Note that 
		\begin{align}
		(\cLt-\cL)f(x,v) &=   \lref \int_{\cH} [f(x,w)-f(x,v)]\nu_0(dw) \nonumber+ \sum_{i=1}^\infty(\lt_i(x,v)-\lambda_i(x,v))[f(x,R_i v)-f(x,v)].
		\end{align}
		Observe that $(\cLt-\cL)f$ corresponds to a Boomerang Sampler with characteristics $(\C,0,0,\{\lt_i-\lambda_i\}_i,0,\{R_i\}_i)$ and hence corresponds to a strongly continuous contraction semigroup, in particular by the Hille-Yosida theorem $\cLt-\cL$ is dissapitive.
		Now since $\cLt-\cL$ is a bounded and dissipative operator, by \cite[Theorem 2.7]{Engel} we have that $(\cLt, \DCo)$ generates a contraction semigroup. Since $\core\subseteq D(\tilde{\cL})$ is a core for $(\tilde{\cL},D(\cLt))$ and $\core\subseteq \DCo$ then using that $\cL$ is closed we see that $D(\tilde{\cL})\subseteq D(\cL)$. Therefore the semigroup generated by $(\cLt,\DCo)$ is an extension of the semigroup generated by $(\cLt,D(\cLt))$ and therefore must coincide and in particular we have $D(\cL)=D(\cLt)$. Therefore we have that $\core$ is a core for $(\cL, \DCo)$.
\end{proof}

\begin{corollary}\label{cor:invmeas}
    Suppose that the assumptions of Theorem \ref{thm:smoothcore} hold, the assumptions of Theorem \ref{thm:invmeas} hold then $\mu$ is an invariant measure for $\cP_t$. Moreover, there is an extension of $\cP_t$ to $L_\mu^2$ which is strongly continuous and has generator given by $(\cL, \DL)$.
\end{corollary}

\begin{proof}[Proof of Corollary \ref{cor:invmeas}]
By Proposition \ref{prop:IM} if Hypothesis \ref{hyp:IMass} holds then \eqref{eq:Lmu=0onC0} for all $f\in \DCo\cap D(\LX)$, therefore it is remains to show that $(X_t,V_t)$ satisfies Hypothesis \ref{hyp:IMass} and that $\core \subseteq D(\LX)$. Hypothesis \ref{hyp:IMass} is verified in Lemma \ref{lem:wellposedframework}. To see that $\core\subseteq D(\LX)$, fix $f\in \core$ and observe that by Taylor's theorem for any $t>0$ there exists some $s\in (0,t)$
\begin{equation*}
    \frac{1}{t}(f(e^{t\fX}x)-f(x))-\LX f(x)= \frac{\partial}{\partial t}f(e^{t\fX}x)\rvert_{t=s}-\LX f(x) =\langle \fX(e^{s\fX}x), \nabla f( e^{s\fX}x)\rangle -\langle \fX(x), \nabla f(x)\rangle
\end{equation*}
Note that the right hand side is bounded and converges to zero as $t\to 0$ therefore by the dominated convergence theorem
\begin{equation*}
    \lim_{t\to 0}\lVert\frac{1}{t}(f(e^{t\fX}x)-f(x))-\LX f(x)\rVert_{L_{\mu_0}^2} =0.
\end{equation*}
Hence $\mu$ is an invariant measure for $\cP_t$.

To see that $\cP_t$ extends to a $C_0$-semigroup on $L_\mu^2$ follows from the proof of \cite[Theorem 5.8]{DaPrato}. The only difference is that \cite[Theorem 5.8]{DaPrato} is stated for semigroups defined on $C_b(\cH\times \cH)$. However it is clear that the proof still holds for semigroups defined on $C_0(\cH\times \cH)$ provided $C_0(\cH\times \cH)$ is dense in $L_\mu^2$. We show that $C_0(\cH\times \cH)$ by showing that $\core \subseteq \DCo$. Fix $f\in \core$ then by Theorem \ref{thm:extgen} we have
\begin{equation*}
    \cP_tf(x,v)-f(x,v)=\int_0^t \cP_s\cL f(x,v) ds.
\end{equation*}
Therefore to show that $\cP_tf-f$ converges to zero as $t\to 0$ it remains to show $\cL f(x,v)$ is bounded uniformly in $(x,v)$. Let $K>0$ be such that $\mathrm{supp}(f)\subseteq B_K$ then 
\begin{equation*}
    \lVert \cL f\rVert_\infty \leq 2K\lVert \nabla f\rVert_\infty  + 2\lref \lVert f\rVert_\infty + 2\lVert f\rVert_\infty\sup_{(x,v)\in B_K}\sum_{n=1}^\infty\lambda_n(x,v).
\end{equation*}
By Hypothesis \ref{hyp:WPass} which holds by Lemma \ref{lem:wellposedframework} we have that $\sum \lambda_n$ is bounded on $B_k$ so we have that $\lVert \cL f\rVert_\infty <\infty$. Therefore $\core\subseteq C_0(\h\times\h)$. Since $\core$ is dense in $L_\mu^2$ we have that $C_0(\h\times\h)$ is dense in $L_\mu^2$. As $\{\cP_t\}_{t\geq 0}$ is strongly continuous on $L_\mu^2$ it has a generator given by $(\cL,\DL)$ which is an extension of $(\cL,\DCo)$.
\end{proof}

\begin{theorem}\label{thm:core}
    Suppose that the assumptions of Corollary \ref{cor:invmeas} hold then $\cylfunccomp(\h\times\h)$ is a core for $(\cL, \DL)$.
\end{theorem}

\begin{proof}[Proof of Theorem \ref{thm:core}]
    	Note that $\cylfunc(H)$ and $\cylfunccomp$ are both dense in $L_\mu^2$ (and $L_{\mu_0}^4$), see \cite[Lemma 2.2]{DaPratoLunardi}.
    \begin{equation*}
        e^{t\LX }f(x,v)=f(\varphi_t(x,v))
    \end{equation*}
    Since $e^{t\LX }$ preserves the set $\cylfunccomp(\h\times\h)$ we have by \cite[Proposition 3.3]{ethier} that this set is a core for $\LX $. 
    
    Fix $f\in \core$, it is immediate to check that $f\in D_{L_{\mu_0}^4}(\LX )$ so there exists $f_n\in \cylfunccomp(\h\times\h)$ such that $f_n\to f$ and $\LX f_n\to \LX f$ in $L_{\mu_0}^4$. As $\Phi$ is bounded from below we also have that $f_n\to f$ and $\LX f_n\to \LX f$ in $L_\mu^2$ and that $f_n\to f$ in $L_\mu^4$. Now
    \begin{align*}
        \lVert \cL f-\cL f_n\rVert_{L_{\mu}^2} &\leq \lVert \LX  f-\LX f_n\rVert_{L_{\mu}^2}+\lVert \sum_{i=1}^\infty\lambda_i (f(\cdot,R_i\cdot)-f_n(\cdot,R_i\cdot))\rVert_{L_{\mu}^2}\\
        &+\lVert \sum_{i=1}^\infty\lambda_i (f-f_n)\rVert_{L_{\mu}^2} + 2\lref \lVert f- f_n\rVert_{L_\mu^2}.
    \end{align*}
    Using Cauchy-Schwartz, and that $\mu_0$ is invariant under $R_i$  we can bound this by
    \begin{align*}
        \lVert \cL f-\cL f_n\rVert_{L_{\mu}^2} \leq \lVert \LX  f-\LX f_n\rVert_{L_{\mu}^2}+2\lVert \sum_{i=1}^\infty\lambda_i\rVert_{L_{\mu}^4}^{\frac{1}{2}}\lVert (f-f_n)\rVert_{L_{\mu}^4}^{\frac{1}{2}}+2\lref \lVert f-f_n\rVert_{L_\mu^2}.
    \end{align*}
    Note here $\lambda\in L_\mu^4$ by \eqref{eq:lambdaframeworkupperbound}.
    Letting $n$ tend to $\infty$ we have that $\cL f_n\to \cL f$ in $L_\mu^2$. That is the closure $\overline{\cylfunccomp(\h\times\h)}$ of $\cylfunccomp(\h\times\h)$ in the graph norm of $\cL$ with respect to the $L_\mu^2$ topology contains $\core$. Now $\core$ is a core for $(\cL,\DCo)$ and since convergence with the supremum norm implies convergence in $L_\mu^2$-norm we have that $\DCo\subseteq \overline{\cylfunccomp(\h\times\h)}$. Again by \cite[Proposition 3.3]{ethier} we have that $\DCo$ is a core for $(\cL,\DL)$  and hence $\cylfunccomp(\h\times\h)$ is a core for $(\cL,\DL)$. 
\end{proof}

\section{Exponential convergence to equilibria}\label{sec:expconv}

In this section we will employ the strategy of \cite{Andrieu,Dolbeault} to give conditions under which we have exponential convergence to equilibria. Let $(X_t,V_t)$ be the Boomerang Sampler with  characteristics $(\C,1,1,\{\lambda_i\}_i,\lref,\{R_i\}_i)$ where $\lambda_i$ are given by \eqref{eq:lambdaBSframework}, and let $R_i$ satisfy the following conditions. 
\begin{hypothesis}
Suppose $R_i$ is a linear operator on $\h$, with $R_n^2=1$, $\lVert R_nv\rVert=\lVert v\rVert$ for any $v\in \cH$, satisfies \eqref{eq:invarianceforreflcond} and \eqref{eq:Rassump4}. Moreover suppose that for any $v\in \h$
\begin{align} 
\sum_{n=1}^\infty (1-R_n)v &= 2v, \label{eq:Rassump1}\\ 
\frac{1}{4}\sum_{n=1}^\infty (1-R_n)^*\Sigma (1-R_n)&=\Sigma,\label{eq:Rassump2}\\
(1-R_n)^*\Sigma (1-R_m)&=0 \text{ for } m\neq n,\label{eq:Rassump3}
\end{align}
Here $(1-R_n)^*$ denotes the adjoint of $(1-R_n)$ as a linear operator from $\cH$ to $\cH$.
\end{hypothesis}

For simplicity we have canonical rates, that is we have set $\gamma_i=0$ where $\gamma_i$ is as in Section \ref{sec:Boomerang_PR} and \ref{sec:Boomerang_factorised}. Note that if $\Phi$ satisfies Assumption \ref{ass:1} then Hypothesis \ref{hyp:IMass} is satisfied and hence Hypothesis \ref{hyp:WPass} also, see Lemma \ref{lem:wellposedframework}. 

\begin{example}[Pure Reflection Boomerang Sampler]
We can express the Pure Reflection Boomerang Sampler in the above notation by setting $R_1=-1, R_n=1$ for any $n>1$. 
\end{example}

\begin{example}[Factorised Boomerang Sampler]
Alternatively we can also express the Factorised Boomerang Sampler by taking $R_nv=v-2\langle v,e_n\rangle_{\cH}e_n$. Indeed in this case \eqref{eq:Rassump3} follows immediately as $e_n,e_m$ are orthogonal for $n\neq m$ and diagonalise $\Sigma$, moreover we have
\begin{equation*}
    (1-R_n)^*\Sigma(1-R_m)v=4v_n\gamma_n^2e_n\delta_{n,m}.
\end{equation*}
Taking $n=m$ and summing over $n$ gives
\begin{equation*}
    \sum_{n=1}^\infty(1-R_n)^*\Sigma(1-R_n)v=4\sum_{n=1}^\infty v_ne_n\gamma_n^2 = 4\Sigma v.
\end{equation*}
Finally checking \eqref{eq:Rassump1}
\begin{equation*}
    \sum_{n=1}^\infty (1-R_n)v = 2\sum_{n=1}^\infty v_ne_n=2v.
\end{equation*}
\end{example}

The conditions we will require for exponential convergence can be expressed in terms of the second order differential operator defined on smooth functions $g:\cH\to\R$ by
\begin{equation}\label{eq:defofA}
\cA g(x) = \Tr(\Sigma\nabla^2_x g(x))-\langle x,\nabla_x g(x) \rangle_{\cH} -\left\langle \nabla_x\Phi(x),\Sigma \nabla_x g(x) \right\rangle_{\cH}.
\end{equation}
Note that this is an operator only on the $x$ variable.

\begin{theorem}\label{thm:Hypocoecivity}
 Let $\cP_t$ be the semigroup corresponding to the Boomerang Sampler which is described by the generator \eqref{eq:simgen}. Assume that Assumption \ref{ass:1} is satisfied, and that $\Phi$ is twice continuously differentiable with bounded Hessian, and there exist $c_2\geq 0, c_1>0$ such that
	\begin{equation}\label{eq:boundonLaplacianU}
	\cA \Phi(x) \leq c_2-c_1 \lVert \Sigma^\frac{1}{2}\nabla_x\Phi(x)\rVert_\cH^2.
	\end{equation}
	Also assume that $e^{-\Phi}\in L^1(\pi_0)$ where $\pi_0=\nu_0=\mathcal{N}(0,\Sigma)$, and set $\pi$ to be the measure defined by \eqref{targetmeasure} and $\mu=\pi\times \nu_0$. Then there exist constants $C, \kappa>0$ such that
	\begin{equation}\label{eq:expdecay}
	\lVert \cP_t f -\mu(f)\rVert_{L_\mu^2}\leq C e^{-\kappa t} \lVert f-\mu(f)\rVert_{L_\mu^2}
	\end{equation}
	for all $f\in L_\mu^2$.
\end{theorem}

\begin{corollary}\label{cor:hypocoercive}
Let $\Phi$ be a twice continuously differentiable function satisfying Assumption~\ref{ass:1}, has bounded Hessian and is such that $\langle \nabla \Phi(x), x \rangle \ge C$ for some $C \in \R$. Then $\Phi$ satisfies the assumptions of Theorem \ref{thm:Hypocoecivity}. In particular, this is satisfied if $\Phi$ is convex, bounded from below and has bounded Hessian.
\end{corollary}

\begin{proof}[Proof of Corollary \ref{cor:hypocoercive}]
	To show that \eqref{eq:boundonLaplacianU} holds, observe that
	\begin{equation*}
	\cA \Phi(x) = \Tr(\Sigma\nabla^2_x \Phi(x))-\langle x,\nabla_x \Phi(x) \rangle_{\cH} -\left\langle \nabla_x\Phi(x),\Sigma \nabla_x \Phi(x) \right\rangle_{\cH}.
	\end{equation*}
	since $\Phi$ has bounded Hessian we have that $\Tr(\Sigma\nabla^2_x \Phi(x))$ is bounded so it is sufficient that $\langle x,\nabla_x \Phi(x) \rangle_{\cH} \geq C$ for some $C\in \R$. We now show this holds if $\Phi$ is convex, and bounded from below. Fix $x\in \cH$, by Taylor's theorem there exists $z(x)$ such that
	\begin{equation*}
	\Phi(0) = \Phi(x) - \langle\nabla_x \Phi(x), x \rangle_{\cH} + \frac{1}{2} \langle \nabla_x^2\Phi(z(x)) x, x\rangle_{\cH}.
	\end{equation*}
	Using that $\Phi$ is convex
	\begin{equation*}
	\Phi(0) \geq \Phi(x) - \langle\nabla_x \Phi(x), x \rangle_{\cH} .
	\end{equation*}
	Finally rearranging we have
	\begin{equation*}
\langle\nabla_x \Phi(x), x \rangle_{\cH} \geq \Phi(x) - 	\Phi(0) 
	\end{equation*}
	which is bounded from below as $\Phi$ is bounded from below.
\end{proof}

\begin{example}\label{ex:diffusionbridge}

Fix $T>0$ and consider the one-dimensional SDE with initial and terminal condition
\begin{equation}\label{eq:diffusionbridge}
\begin{cases}	
dY_t=b(Y_t)dt+ dW_t, & t\in (0,T)\\
Y_0=Y_T=0.
\end{cases}
\end{equation}
Here $W_t$ is a one-dimensional Brownian motion, $b:\R\to\R$ is a smooth function and $\sigma$ is some positive constant. We shall assume that $b\in C_b^3(\R)$, i.e. that $b$ is bounded and has bounded first and second order derivatives. Denote by $\pi$ the law of $Y$ defined on the space $L^2([0,T])$, the aim is to sample from $\pi$. We shall do so by using a Girsanov transformation to write the measure $\pi$ as a Radon-Nikyodym derivative of the measure $\pi_0$, where $\pi_0$ is the law of the one-dimensional Brownian Bridge on $[0,T]$.  Set $\cH=L^2([0,T])$ and define the operator
\begin{equation}
	\begin{cases}
	Ax(\xi)=\frac{1}{2}x''(\xi), & x\in D(A),\\
	D(A) = H^2(0,1)\cap H_0^1(0,1),
	\end{cases}
\end{equation}
where $H^k$ are the usual Sobolev spaces and
\begin{equation*}
	H_0^1(0,T)=\{x\in H^1(0,T): x(0)=x(T)=0\}.
\end{equation*}
Let $X_t$ be a one-dimensional Brownian bridge, then set
\begin{equation*}
\mathcal{E}_t=\exp\left(\int_0^t b(X_s) dX_s - \frac{1}{2}\int_0^t b(X_s)^2ds\right).
\end{equation*}
By Girsanov's theorem the process
\begin{equation*}
 W_t := X_t-\int_0^tb(X_s)ds
\end{equation*}
is a Brownian motion under $\pi$ where $\frac{d\pi}{d\pi_0}=\mathcal{E}_T$. Note that by rearranging we have that $X_t$ solves \eqref{eq:diffusionbridge} under $\pi$. Now let $B=\int_0^t b(y) dy$, then by It\^o's formula we have
\begin{equation*}
B(X_T)-B(X_0) - \frac{1}{2}\int_0^T b'(X_s)ds = \int_0^T b(X_s) dX_s.
\end{equation*}
Recall that $X_T=X_0=0$ so we can now write $\mathcal{E}_T$ as
\begin{equation*}
\mathcal{E}_T=\exp\left(- \frac{1}{2}\int_0^T b'(X_s)ds - \frac{1}{2}\int_0^t b(X_s)^2ds\right).
\end{equation*}
Therefore we have that
\begin{align*}
\frac{d\pi}{d\pi_0}(x) &= \exp(-\Phi(x))\\
\Phi(x)&=\frac{1}{2}\int_0^T b'(X_s)ds + \frac{1}{2}\int_0^T b(X_s)^2ds
\end{align*}
Note that $\Phi$ is smooth and
\begin{align*}
\nabla_x\Phi(x)(t) &= \frac{1}{2}b''(x_t) + b(x_t)b'(x_t)\\
\nabla_x^2\Phi(x)(t) &= \frac{1}{2}b'''(x_t) + b(x_t)b''(x_t) + 2b'(x_t)^2.
\end{align*}
%
As $b\in C_b^3(\R)$ we have that $\Phi$ is bounded with bounded first and second order derivatives, therefore by Remark \ref{note:boundedresest} below 
all the assumptions of Theorem \ref{thm:Hypocoecivity} are satisfied so the Boomerang Sampler converges exponentially in the sense of \eqref{eq:expdecay}.

To keep calculations short we are only working in 1-d, but similar arguments can be applied in higher dimensions. In \cite[Equation 1.5]{Beskos} it is shown that if we consider the SDE
\begin{equation*}
dX_t=f(X_t)dt+B dW_t
\end{equation*} 
where $f=-BB^T\nabla_xV$ then the corresponding $\Phi$ is given by
\begin{equation*}
\Phi(x) = \int_0^t \frac{1}{2}\lvert B^{-1}f(x_t) \rvert^2 + \frac{1}{2}(\nabla_x\cdot f)(x_t) dt.
\end{equation*}

\end{example}


\subsection{Hypocoercivity}\label{sec:hypocoercivity}

We shall prove Theorem \ref{thm:Hypocoecivity} by applying the Abstract Hypocoercivity Theorem introduced in \cite{Dolbeault}. We shall use the formulation from \cite{eisenhuth2021hypocoercivity} which has been developed for infinite dimensional stochastic differential equations. 

\begin{assumption}\label{ass:Hypo_data}
\begin{enumerate}[(\textbf{D\arabic*})]
    \item Set $H=L^2_\mu$ and let $(L,D(L))$ be a linear operator on $H$ generating a $C_0$-semigroup $\{P_t\}_{t\geq 0}$. Where $\{P_t\}_{t\geq 0}$ is conservative, i.e. $P_t1=1$ for all $t\geq 0$ and has invariant measure $\mu$.\label{ass:hypo_setup}
    \item Let $\core\subseteq D(L)$ be a dense subspace of $H$ which is a core for the operator $(L,D(L))$.\label{ass:hypo_core}
    \item Let $(S,D(S))$ be symmetric and let $(A,D(A))$ be a closed and antisymmetic operator on $H$ such that $\core\subseteq D(S)\cap D(A)$ as well as $L\rvert_\core=S-A$.\label{ass:hypo_decomp}
    \item Let $\Pi:H\to H$ be an orthogonal projection which satisfies $\Pi(H)\subseteq D(S), S\Pi=0$ as well as $\Pi(\core)\subseteq D(A),A\Pi(\core)\subseteq D(A)$. \label{ass:hypo_projection}
\end{enumerate}
\end{assumption}

\begin{assumption}\label{ass:Hypo_hypo}
\begin{enumerate}[(\textbf{H\arabic*})]
\item Assume that $\Pi A\Pi\rvert_\core=0$. \label{ass:hypo_APA}
\item \label{ass:hypo_Scoercive}There exists $\lambda_m>0$ such that
\begin{equation*}
    -\langle Sf,f\rangle_H\geq \lambda_m \lVert (1-\Pi) f\rVert^2 \text{ for all } f\in \core.
\end{equation*}
\item \label{ass:hypo_coercive_PI}Define $G=\Pi A^2\Pi$ on $\core$. Assume that $(G,D)$ is essentially self-adjoint on $H$ and assume that there exists $\lambda_M$ such that
\begin{equation*}
    \lVert A\Pi f\rVert_H^2 \geq \lambda_M \lVert \Pi f\rVert_H^2 \text{ for all } f\in \core.
\end{equation*}
\item \label{ass:hypo_remainder} Define $B$ to be the unique extension to a bounded linear operator on $H$ of the operator $(B,D(A\Pi)^*)$ where
\begin{equation*}
    B=(I+(A\Pi)^*(A\Pi))^{-1}(A\Pi)^* \text { on } D((A\Pi)^*).
\end{equation*}
Assume that the operators $(BS,\core)$ and $(BA(1-\Pi),\core)$ are bounded and there exist constants $c_1,c_2<\infty$ such that for all $f\in\core$
\begin{equation*}
    \lVert BSf\rVert_H\leq c_1\lVert (1-\Pi)f\rVert_H, \quad \text{ and } \quad \lVert BA(1-\Pi)f\rVert_H\leq c_2\lVert (1-\Pi)f\rVert_H.
\end{equation*}
\end{enumerate}
\end{assumption}

\begin{theorem}\label{thm:abstractHypocoecivity}
Assume that Assumption \ref{ass:Hypo_data} and \ref{ass:Hypo_hypo} holds. Then there exist positive constants $C, \kappa>0$ such that for any function $f\in L^2_\mu$, and $t\geq 0$,
	\begin{equation*}
	\lVert P_t f -\mu(f)\rVert_{\sH} \leq Ce^{-\kappa t} \lVert f -\mu(f)\rVert_\sH
	\end{equation*}
	where $\mu(f) = \int f d\mu$.
\end{theorem}

Then Theorem \ref{thm:Hypocoecivity} follows from Theorem \ref{thm:abstractHypocoecivity} once we verify these conditions for the Boomerang Sampler. Before we verify these conditions we state the following lemma which is quoted from \cite[Lemma 1]{eisenhuth2021hypocoercivity}.


\begin{lemma}\label{lem:isserlis}
    For the Gaussian measure $\nu_0$ and $v_1,v_2,v_3,v_4\in \cH$ we have that
    \begin{align*}
        \int_{\cH} \langle u,v_1\rangle_{\cH}\langle u,v_2\rangle_{\cH}\nu_0(du) &= \langle \C v_1,v_2\rangle_{\cH},\\
        \int_{\cH} \langle u,v_1\rangle_{\cH}\langle u,v_2\rangle_{\cH}\langle u,v_3\rangle_{\cH}\langle u,v_4\rangle_{\cH}\nu_0(du) &= \langle \C v_1,v_2\rangle_{\cH}\langle \C v_3,v_4\rangle_{\cH}\\
        &+\langle \C v_1,v_3\rangle_{\cH}\langle \C v_2,v_4\rangle_{\cH}+\langle \C v_1,v_4\rangle_{\cH}\langle \C v_2,v_3\rangle_{\cH}.
    \end{align*}
\end{lemma}


\begin{proof}[Proof of Theorem \ref{thm:Hypocoecivity}] To simplify notation we will only consider $f$ such that $\mu(f)=0$, the result then follows by setting $\tilde{f}=f-\mu(f)$. Let us first consider Assumption \ref{ass:Hypo_data}, note that \ref{ass:hypo_setup} is immediate, and \ref{ass:hypo_core} follows from Theorem \ref{thm:core}. 
 
Define $(S, D(S))$ and $(A,D(A))$ as the closure of the operators $(S,\cylfunc(\h\times\h)), (A,\cylfunc(\h\times\h))$ where
\begin{align}
Sf(x,v)&:=\frac{1}{2}(\cL+\cL^*)f(x,v)= \sum_{n=1}^\infty\lambda_n^e(x,v)[f(x,R_n v)-f(x,v)]+\Lref f(x,v), \quad f\in \cylfunc(\h\times\h)\label{eq:defofS}\\
Af(x,v)&:=\frac{1}{2}(\cL-\cL^*)f(x,v)= \LX f(x,v) + \sum_{n=1}^\infty\lambda_n^o(x,v)[f(x,R_nv)-f(x,v)],\quad f\in \cylfunc(\h\times\h). \label{eq:defofT}
\end{align}
 Here
\begin{align}
\lambda_n^e(x,v) &= \frac{1}{2}(\lambda_n(x,v)+\lambda_n(x,R_nv))= \frac{1}{4}\lvert \langle\nabla_x\Phi(x),v-R_nv\rangle_{\cH}\rvert+\gamma_n(x,v),\\
\lambda_n^o(x,v) &= \frac{1}{2}(\lambda_n(x,v)-\lambda_n(x,R_nv))=\frac{1}{4} \langle\nabla_x\Phi(x),v-R_nv\rangle_{\cH}.\label{eq:defofolambdao}
\end{align}
For $f\in L_\mu^2$ set
\begin{equation}\label{eq:defofPi}
    \Pi(f)(x) = \int_{\h} f(x,v) \nu_0(dv).
\end{equation}
Note that since $(S,\cylfunc(\h\times\h))$ and $(A,\cylfunc(\h\times\h))$ are the restriction of the operators $(S^*,D(S^*))$ and $(-A^*, D(A^*))$ to the set $\cylfunc(\h\times\h)$ which is dense we have that $(S,\cylfunc(\h\times\h))$ and $(A,\cylfunc(\h\times\h))$ are closable. It is immediate to verify that $\Pi$ is an orthogonal projection on $L_\mu^2$ therefore all of these operators are well-defined and we have that \ref{ass:hypo_decomp} is satisfied. 

Let us now verify \ref{ass:hypo_projection}. Fix $g\in \Pi(L_\mu^2)$ and take $g_n\in \cylfunccomp(\h\times \h)$ such that $g_n\to g$ which exists since $\cylfunccomp(\h\times\h)$ is dense in $L_\mu^2$. Then $\Pi g_n\in \cylfunc(\h\times\h)$, and we have $S(g_n) =0$. Since $(S,D(S))$ is closed we have that $\Pi(L_\mu^2)\subseteq D(S)$ and $S\Pi=0$. Fix $f\in \cylfunc(\h\times\h)$ then $\Pi f\in \cylfunc(\h\times\h)\subseteq D(A)$ and
\begin{equation}\label{eq:TPi}
    A\Pi f = \langle v,\nabla_x\Pi f(x,v)\rangle_{\cH}.  
\end{equation}
It is straightforward to check that $A\Pi f\in \cylfunc(\h\times\h)$ and therefore $A\Pi \cylfunc(\h\times\h)\subseteq D(A)$.

Let us now verify Assumption \ref{ass:Hypo_hypo}. Note that Assumption \ref{ass:Hypo_hypo} \ref{ass:hypo_APA} is verified immediately. Assumption \ref{ass:Hypo_hypo} \ref{ass:hypo_Scoercive} follows from Lemma \ref{lem:coercivityofS} below. 

In order to verify Assumption \ref{ass:Hypo_hypo} \ref{ass:hypo_coercive_PI} we must first derive an expression for the operator $G$. Fix $f\in \cylfunc(\h\times\h)$ and let $g=\Pi f$ then we by applying $A$ to \eqref{eq:TPi} we have
\begin{align}
(A^2g)(x,v)&=  \langle v, \nabla_x\left\langle 
v, \nabla_x g(x) \right\rangle_{\cH}\rangle_{\cH} -\langle x,\nabla_xg(x) \rangle_{\cH}-\sum_{n=1}^\infty\lambda_n^o(x,v)\left\langle (1-R_n)v, \nabla_x g(x) \right\rangle_{\cH}\\
&=\left\langle v, \nabla^2_x g(x)  v\right\rangle_{\cH}-\langle  x,\nabla_x g(x) \rangle_{\cH}-\sum_{n=1}^\infty\lambda_n^o(x,v)\left\langle (1-R_n)v, \nabla_x g(x) \right\rangle_{\cH}\label{eq:T2Pi}
\end{align}
Using \eqref{eq:defofolambdao} and Lemma \ref{lem:isserlis} we have
\begin{align*}
\Pi[\lambda_n^o(x,v)\left\langle (1-R_n)v, \nabla_x\Pi f(x) \right\rangle_{\cH}]&=\frac{1}{4}\Pi[\langle \nabla_x\Phi,(1-R_n)v\rangle_{\cH}\left\langle (1-R_n)v, \nabla_x\Pi f(x) \right\rangle_{\cH}]\\
&=\frac{1}{4}\left\langle (1-R_n)\nabla_x\Phi(x),\Sigma(1-R_n)\nabla_x\Pi f(x) \right\rangle_{\cH}.
\end{align*}
Here we have used that $\Pi$ denotes integration with respect to $\nu_0$ and $\nu_0$ has mean zero and covariance operator $\Sigma$. Summing over $n$ and using \eqref{eq:Rassump2} we have
\begin{align*}
\sum_{n=1}^\infty\Pi[\lambda_n^o(x,v)\left\langle (1-R_n)v, \nabla_x\Pi f(x) \right\rangle_{\cH}]=\left\langle \nabla_x\Phi(x),\Sigma\nabla_x\Pi f(x) \right\rangle_{\cH}.
\end{align*}

By applying $\Pi$ to the first term in \eqref{eq:T2Pi} we find
\begin{align*}
\Pi\left\langle v, \nabla_x^2\Pi f(x)  v\right\rangle_{\cH} 
&=\Tr(\Sigma\nabla^2_x\Pi f(x)).
\end{align*}
Therefore
\begin{equation*}
\Pi A^2\Pi f = \Tr(\Sigma\nabla^2_x\Pi f(x))-\langle x,\nabla_x\Pi f(x) \rangle_{\cH}-\left\langle \nabla_x\Phi(x),\Sigma \nabla_x\Pi f(x) \right\rangle_{\cH}.
\end{equation*}
Therefore, we have $Gf=\Pi T^2\Pi f=\cA g$ for any $f\in \cylfunc(\h\times\h)$.

From \cite[Proposition 4]{eisenhuth2021hypocoercivity} we have that $(G,
\cylfunc(\h))$ satisfies the Poincar\'e inequality for all $g\in \cylfunc(\h)$
\begin{equation*}
    \int_\h (g(x)-\pi(g))^2 \pi(dx) \leq \lambda_1^{-1} \int_\h \langle \Sigma \nabla_x g,\nabla_xg\rangle_\h \pi(dx) =-\langle \cA g,g\rangle_{L^2_\pi}.
\end{equation*}
Note that \cite{eisenhuth2021hypocoercivity} requires that the function $\Phi$ is convex however if $\Phi=\Phi_1+\Phi_2$ with $\Phi_1$ convex and $\Phi_2$ bounded then the Poincar\'e inequality still holds, this can be seen by first applying \cite[Proposition 4]{eisenhuth2021hypocoercivity} for $\Phi_1$ and then by the same argument as in \cite[Proposition 4.2.7]{Bakry} the Poincar\'e inequality holds for $\Phi$.

Therefore taking $g=\Pi f$ for $f\in L_\mu^2$ we have that Assumption \ref{ass:Hypo_hypo} \ref{ass:hypo_coercive_PI} holds.

As explained in \cite[Remark 2.17]{grothaus2014hypocoercivity} we can replace Assumption \ref{ass:Hypo_hypo} \ref{ass:hypo_remainder} with
    \begin{equation}\label{eq:remaindertermsinnerproductform}
        \langle B(S-A(1-\Pi))f,\Pi f\rangle \leq c_3\lVert (1-\Pi)f\rVert \lVert \Pi f\rVert
    \end{equation}
    for some constant $c_3\geq 0$ and for all $f\in \cylfunc(\cH\times\cH)$.
    Therefore it suffices to prove that \eqref{eq:remaindertermsinnerproductform} holds which we defer to Proposition \ref{prop:SAestintermsofresolvent}.

\end{proof}


\begin{lemma}\label{lem:coercivityofS}
	Assume that the assumptions of Theorem \ref{thm:Hypocoecivity} hold and let $S$ and $\Pi$ be the operators defined by \eqref{eq:defofS} and \eqref{eq:defofPi} respectively . Then for any $f\in D(S)$ we have
	\begin{equation*}
	-\langle Sf,f\rangle_{L_\mu^2} \geq \lref \lVert (1-\Pi)f \rVert_{L_\mu^2}^2
	\end{equation*}
	In particular, we have that Assumption \ref{ass:Hypo_hypo} \ref{ass:hypo_Scoercive} of Theorem \ref{thm:abstractHypocoecivity} holds.
\end{lemma}

\begin{proof}[Proof of Lemma \ref{lem:coercivityofS}]
	Note that
	\begin{align*}
	-\langle Sf,f\rangle_{L_\mu^2} &= -\sum_{n=1}^\infty\int_{\cH}\int_\cH \left(\lambda_n^e(x,v)[f(x,R_nv)-f(x,v)]+\Lref f(x,v)\right)f(x,v) \nu_0(dv)\pi(dx).
	\end{align*}
	Since $\nu_0$ and $\lambda_n^e$ are both invariant under $R_n$ we have
	\begin{align*}
	\int_{\cH} \lambda_n^e(x,v)[f(x,R_nv)-f(x,v)]^2 \nu_0(dv) &= \int_{\cH} 2\lambda_n^e(x,v)f(x,v)^2 -2\lambda_n^e(x,v)f(x,R_nv)f(x,v) \nu_0(dv)\\
	&= -2\int_{\cH} \lambda_n^e(x,v)[f(x,R_nv) -f(x,v)]f(x,v) \nu_0(dv).
	\end{align*} 
	Therefore we have for any $n\in\mathbb{N}$
	\begin{equation*}
	\langle \lambda^e(x,v)[f(x,R_nv)-f(x,v)],f\rangle_{L_\mu^2} \leq 0.
	\end{equation*}
	Which gives
	\begin{align*}
	-\langle Sf,f\rangle_{L_\mu^2} &\geq  -\langle\Lref f,f\rangle_{L_\mu^2} \\
	&=   \lref\langle (1-\Pi)f,f\rangle_{L_\mu^2}
	\end{align*}	
	The result follows since $\Pi$ is an orthogonal projection.
\end{proof}

\begin{prop}\label{prop:resest}
	Let $\cP_t$ be the semigroup corresponding to the Boomerang Sampler which is described by the generator \eqref{eq:simgen}. Assume that $\Phi$ is twice continuously differentiable with bounded Hessian and there exist $c_2\geq 0, c_1>0$ such that \eqref{eq:boundonLaplacianU} holds. Also assume that $e^{-\Phi}\in L^1(\pi_0)$ where $\pi_0=\nu_0=\mathcal{N}(0,\Sigma)$, and set $\pi$ to be the measure defined by \eqref{targetmeasure} and $\mu=\pi\times \nu_0$. Fix $g\in \cylfunc(\cH)$ and set $h=g-\cA g$ then we have
	\begin{align}
	\lVert \Sigma^{\frac{1}{2}}\nabla_xg\rVert_{L_\pi^2} &\leq  \lVert h  \rVert_{L_\pi^2}, \label{eq:d1Resolventbound}\\
	\lVert \nabla_x\Sigma\nabla_xg\rVert_{L_\pi^2} &\leq \kappa_1\lVert h  \rVert_{L_\pi^2}, \label{eq:d2Resolventbound}\\
	\left\lVert \langle \Sigma^\frac{1}{2}\nabla_xg , \Sigma^\frac{1}{2}\nabla_x\Phi\rangle_{\cH}\right\rVert_{L_\pi^2} &\leq \left\lVert \lVert \Sigma^\frac{1}{2}\nabla_xg \rVert_{\cH} \Sigma^\frac{1}{2}\nabla_x\Phi\right\rVert_{L_\pi^2(\cH;\cH)}  \leq \kappa_2\lVert h \rVert_{L_\pi^2}. \label{eq:innerprodresolventbound}
	\end{align}
	Here $\kappa_1^2=2+c_\Phi$, $-c_\Phi\leq \nabla_x^2\Phi(x)$ for all $x\in\cH$ and $\kappa_2^2=\frac{4\kappa_1^2}{C_1^2}+\frac{2C_2}{C_1}$.
\end{prop}
We defer the proof of Proposition \ref{prop:resest} to Section \ref{sec:proofsofhypocoercivity}.

\begin{remark}\label{note:boundedresest}
	In the context of Proposition \ref{prop:resest} if we assume that $\nabla_x\Phi$ is bounded then \eqref{eq:innerprodresolventbound} follows immediately from \eqref{eq:d1Resolventbound}. Therefore we can remove the assumption that $\eqref{eq:boundonLaplacianU}$ holds in Theorem \ref{thm:Hypocoecivity} if $\nabla_x\Phi$ is bounded.
\end{remark}

With these estimates to hand we can show that Assumption \ref{ass:Hypo_hypo} \ref{ass:hypo_remainder} of Theorem \ref{thm:abstractHypocoecivity} holds. 
These are based on \cite[Lemma 12 \& 13]{Andrieu}.
\begin{prop}\label{prop:SAestintermsofresolvent}
	Let $S,A,\Pi$ be defined by \eqref{eq:defofS}, \eqref{eq:defofT} and \eqref{eq:defofPi}. Define $B$ as in Theorem \ref{thm:abstractHypocoecivity}. Assume that estimates \eqref{eq:d1Resolventbound}, \eqref{eq:d2Resolventbound} and \eqref{eq:innerprodresolventbound} hold then we have for all $f\in \cylfunc(\h\times\h)$
	\begin{align}
	\lvert \langle BSf,\Pi f\rangle_{L_\mu^2}\rvert &\leq \left(\sqrt{3}\kappa_2+ \frac{1}{\sqrt{2}}\lref\right)\lVert (1-\Pi)f\rVert_{L_\mu^2} \lVert f \rVert_{L_\mu^2},\label{eq:BS}\\
	\lvert\langle BA(1-\Pi) f,\Pi f\rangle_{L_\mu^2}\rvert&\leq \sqrt{2}(\kappa_1+\kappa_2)\left\lVert  (1-\Pi) f\right\rVert_{L_\mu^2}\left\lVert   f\right\rVert_{L_\mu^2}.\label{eq:BA(1-Pi)}
	\end{align}
	Here $\kappa_2 =\sqrt{\frac{4\kappa_1^2}{c_1^2} + \frac{c_2}{c_1}}$ and $\kappa_1$ is defined as in Proposition \ref{prop:resest}. 
\end{prop}

\begin{proof}[Proof of Proposition \ref{prop:SAestintermsofresolvent}]
    The proof is deferred to Section \ref{sec:proofsofhypocoercivity}.
\end{proof}

\section{Finite Dimensional Approximation}
\label{sec:finite-dimensional-approximation}

In this section we wish to construct a finite dimensional approximation to the Boomerang Sampler. We shall achieve this by taking a finite dimensional approximation of $\Phi$ and then constructing a Boomerang sampler to sample from this measure. Let us assume we have a potential function $\Phi:\cH\to \R$  continuously differentiable and measures $\pi_0,\nu_0,\pi,\mu_0,\mu$ as in Section \ref{sec:PDMPforsampling}. 

Fix an orthonormal basis, $\{e_i\}_{i=1}^\infty$, of eigenvectors of $\Sigma_x$. Denote by $\Proj_N$ the orthogonal projection onto $\cH_N:=\mathrm{span}\{e_1,\ldots, e_N\}$, and define 
\begin{equation*}
\Phi_N(x) = \Phi(\Proj_N(x)).
\end{equation*}
Note that the projection of $\pi_0$ (respectively $\nu_0$) to the space $\cH_N$ is still a Gaussian measure as linear transformations preserve Gaussianity and is centered with covariance operator $\Sigma_{x,N} =\Proj_N\circ \Sigma_x$ (resp. $\Sigma_{v,N} =\Proj_N\circ \Sv$). Now we can consider the problem of how to sample from the measure $\mu_{N}$ defined by
\begin{equation*}
\frac{d\mu_N}{d\mu_{0,N}}(dx,dv)  = \frac{\exp(-\Phi_N(x))}{\int \exp(-\Phi_N(y)) \mu_{0,N}(dy,dw)}
\end{equation*}
where $\mu_{0,N}(dx,dv) = \pi_{0,N}(dx)\nu_{0,N}(dv)$, $\pi_{0,N}=\mathcal{N}(0,\Sigma_{x,N})$ and $\nu_{0,N}=\mathcal{N}(0,\Sigma_{v,N})$. Note that although these measures are supported on the finite dimensional space $\cH_N^2$ we can also view them as measures on $\cH^2$.

Let us introduce some more notation for any $x\in \cH$ we denote by
\begin{align*}
\xN&:=\Proj_Nx=\sum_{n=1}^N x_ne_n,\quad 
x_n:=\langle x,e_n\rangle,\\
\xperp&:=x-\xN =\sum_{n>N} x_ne_n ,
\end{align*}
and similarly we define $\vN:=\Proj_Nv, \vperp:=v-\vN$. 

Let $\{X_t,V_t\}$ denote the Boomerang Sampler whose generator is given by \eqref{eq:simgen}. We shall construct an $N$-dimensional process $\{X_t^N,V_t^N\}_{t\geq 0}$ which takes values in $\cH_N\times\cH_N$, although this process takes values in the space $\cH_N\times\cH_N$ it is convenient for the analysis to view it as a process taking values in $\cH\times\cH$. We first consider two examples before giving a more abstract framework, which is the pure reflection Boomerang Sampler and the Factorised Boomerang Sampler.

\begin{example}\label{ex:finitedimPurerefl}
	Let $\{X_t,V_t\}_{t\geq 0}$ denote the pure reflection Boomerang Sampler whose generator is given by \eqref{eq:simgen}
	\begin{align*}
	\cL_{PR} f(x,v) &= \langle v,\nabla_xf(x,v)\rangle_{\cH} - \langle x,\nabla_vf(x,v)\rangle_{\cH} + \lref \int_{\cH} [f(x,w)-f(x,v)]\nu_0(dw) \\
	&+ (\langle \nabla_x\Phi(x),v\rangle_{\cH})^+[f(x,-v)-f(x,v)].
	\end{align*}
	To construct a finite dimensional approximation we need to specify the deterministic dynamics, the reflections and the refreshments. For the deterministic dynamics act on the first $N$ components in the same way as the full process but leave all other components fixed. Reflections will only reflect the first $N$ components instead of the entire velocity and refreshments will refresh with the measure $\nu_{0,N}$. That is the process ${X_t^N,V_t^N}$ will have generator given by
	\begin{align*}
	\cL_{PR}^N f(x,v) &= \langle \vN,\nabla_xf(x,v)\rangle_{\cH} - \langle \xN,\nabla_vf(x,v)\rangle_{\cH} + \lref \int_{\cH} [f(x,w)-f(x,v)]\nu_{0,N}(dw) \\
	&+ (\langle \nabla_x\Phi_N(x),\vN\rangle_{\cH})^+[f(x,-\vN+\vperp)-f(x,v)].
	\end{align*}
	It is immediate to check that Hypothesis \ref{hyp:IMass} so this process is well posed and has $\mu_{N}$ as an invariant measure.
\end{example}

\begin{example}\label{ex:finitedimfactor}
	Let $\{X_t,V_t\}_{t\geq 0}$ denote the Factorised Boomerang Sampler whose generator is given by \eqref{eq:simgen}
	\begin{align*}
	\cL_{F} f(x,v) &= \langle v,\nabla_xf(x,v)\rangle_{\cH} - \langle x,\nabla_vf(x,v)\rangle_{\cH} + \lref \int_{\cH} [f(x,w)-f(x,v)]\nu_0(dw) \\
	&+ \sum_{n=1}^\infty(\partial_n\Phi(x)v_n)^+[f(x,R_nv)-f(x,v)].
	\end{align*}
	Here $R_nv=v-2\langle v,e_n\rangle_{\cH}e_n$. The finite dimensional approximation is constructed analogously to Example \ref{ex:finitedimPurerefl} except for the reflections. In this case the first $N$ reflections are the same as the Factorised Boomerang, that is the reflection operator $R_n^N$ changes the sign of the $n$-th component for $n\leq N$. All other reflections are set to the identity, that is we only have $N$ reflection operators.  The process ${X_t^N,V_t^N}$ will have generator given by
	\begin{align*}
	\cL_{PR}^N f(x,v) &= \langle \vN,\nabla_xf(x,v)\rangle_{\cH} - \langle \xN,\nabla_vf(x,v)\rangle_{\cH} + \lref \int_{\cH} [f(x,w)-f(x,v)]\nu_{0,N}(dw) \\
	&+ \sum_{n=1}^N(\partial_n\Phi_N(x)v_n)^+[f(x,R_nv)-f(x,v)].
	\end{align*}
	It is immediate to check that Hypothesis \ref{hyp:IMass} so this process is well posed and has $\mu_{N}$ as an invariant measure.
\end{example}

Now we introduce a framework which will allow us to tackle both the situations of Example \ref{ex:finitedimPurerefl} and Example \ref{ex:finitedimfactor} simultaneously. Let $\{X_t,V_t\}_{t\geq 0}$ denote the Boomerang Sampler with generator \eqref{eq:simgen}. We shall assume throughout this section that the assumptions and notation stated at the start of Section \ref{sec:expconv} hold. Let $\{X_t^N,V_t^N\}_{t\geq 0}$ be the Boomerang Sampler with characteristics $(\Sigma_{x,N},\Proj_N,\Proj_N,\{\lambda_N^n\},\lref,\{R_N^n\}_n)$ which samples from the measure $\mu_N$ on the space $\cH^2$. Here $\lambda_n^N$ and $R_N^n$ are approximations of $\lambda_n$ and $R_n$ respectively and we shall require that
\begin{align}
\lVert \sum_{n=1}^\infty(\lambda_n-\lambda_n^N)\rVert^2 &\leq C(\Phi)\sum_{n>N}\gamma^2_n\label{eq:lambdaapprox}\\
\lVert \sum_{n=1}^\infty\lambda_n(x,v) \lVert R_nv-R_n^Nv\rVert_\cH\rVert_{L_\mu^2}^2 &\leq C\lVert \Sigma\nabla_x\Phi\rVert_{L_\pi^2}^2\sum_{n>N}\gamma^2_n\label{eq:Rapprox}.
\end{align}
Moreover we shall assume that $R_n^N\vperp=\vperp$ and that $\lambda_n^N(x,v)=\lambda_n^N(\xN,\vN)$ so that the process $\{X_t^N,V_t^N\}$ only changes the first $N$ components. As in Section \ref{sec:core}, Hypothesis \ref{hyp:IMass} is satisfied for these choices with respect to potential function $\Phi_N$ and by Proposition \ref{prop:IM} we have that $\mu_N$ is an invariant measure for this process. Note this process corresponds to the generator
\begin{align}
\cL_Nf(x,v)&=\langle \vN , \nabla_x f(x,v)\rangle_{\cH} -\langle \xN , \nabla_v f(x,v)\rangle_{\cH} +  \lref \int_{\cH} f(x,w^N+\vperp)-f(x,v) \nu_{0,N}(dw) \nonumber\\
&+ \sum_{n=1}^\infty\lambda_n^N(x,v)[f(x,R_n^Nv)-f(x,v)]\label{eq:gendeffinitedim}
\end{align}
for $x,v\in\cH$ and $f\in D(\cL_N)$. Define $\cP_t^N$ as follows
\begin{equation}\label{eq:semigpdeffinitedim}
\cP_t^Nf(x,v) = \mathbb{E}[f(X_t^{N,x,v}, V_t^{N,x,v})]
\end{equation} 
and recall that $\cP_t$ is the semigroup associated to the process $\{X_t,V_t\}_{t\geq 0}$.

In order to prove convergence of $\cP_t^N$ to $\cP_t$ in a suitable topology we shall rely on the following estimate.

\begin{lemma}\label{lem:estimatediffofoperators}
		Let $\{\cP_t\}$ be semigroup corresponding to the generator \eqref{eq:simgen}. Assume that $\Phi$ is twice continuously differentiable, bounded from below and has bounded Hessian. Let $\{\cP_t^N\}_{t\geq 0}$ be the semigroup defined by \eqref{eq:semigpdeffinitedim} with generator $\cL_N$ defined by \eqref{eq:gendeffinitedim}. Then there exists a constant $C$ which may depend on $\Phi$ but is independent of $t$ and $N$ such for each $f\in C_b^1(\cH^2), N\in \mathbb{N}, t\geq 0$ we have
		\begin{equation*}
		\lVert(\cL_N-\cL)\cP_t^Nf \rVert_{L_\mu^2} \leq C \left( \lVert \nabla_x f \rVert_\infty + \lVert \nabla_v f \rVert_\infty +\lVert f\rVert_\infty\right)\left(\sum_{i=N+1}^\infty \gamma^2_i\right)^\frac{1}{2}.
		\end{equation*}
\end{lemma}

\begin{proof}
	To simplify notation let us set $f_t^N(x,v)=\cP_t^Nf(x,v)$ and notice that as the process $(X_t^N,V_t^N)$ only lives in $\cH_N^2$ and does not depend on $\xperp$ or $\vperp$ we have for $i>N$
	\begin{align}
	\partial_{x_i} f_t^N(x,v) &= \langle \nabla_x f_t^N, e_i\rangle_{\cH}  =\cP_t^N(\partial_{x_i}f)(x,v),\label{eq:tailestderivofNsemigpx}\\
	\partial_{v_i} f_t^N(x,v) &= \langle \nabla_vf_t^N, e_i\rangle_{\cH}  = \cP_t^N(\partial_{v_i}f)(x,v).\label{eq:tailestderivofNsemigpp}
	\end{align}
	In particular , for $f$ with bounded derivatives we have the estimates
	\begin{equation}\label{eq:projnablafN}
	\lVert (1-\Proj_N)\nabla_x f_t^N(x,v)\rVert_{L_\mu^2(\cH^2;\cH)} \leq \lVert \nabla_x f \rVert_\infty, \quad \lVert (1-\Proj_N)\nabla_v f_t^N(x,v)\rVert_{L_\mu^2(\cH^2;\cH)} \leq \lVert \nabla_v f \rVert_\infty
	\end{equation}
	
	Fix $g\in C_b^1(\cH^2)$ then
	\begin{align*}
	\lvert \cL g(x,v)-\cL_Ng(x,v)\rvert  &\leq \lvert \langle \vperp , \nabla_x g(x,v)\rangle_{\cH}\rvert+\lvert \langle \xperp , \nabla_v g(x,v)\rangle_{\cH}\rvert \label{eq:derivterms}\tag{*}\\
	&+  \lref \left\lvert\int_{\cH} g(x,w) \nu_{0}(dw)-\int_{\cH} g(x,w^N+\vperp)\nu_{0,N}(dw_N)\right\rvert \label{eq:refreshterms}\tag{**}\\
	&+ \sum_{n=1}^\infty\lvert \lambda_n(x,v) [g(x,R_nv)-g(x,v)]-\lambda_n^N(x,v) [g(x,R_n^Nv)-g(x,v)]\rvert.\label{eq:reflectionterms}\tag{***}
	\end{align*}
	Replacing $g$ with $f_t^N$ in \eqref{eq:derivterms}-\eqref{eq:reflectionterms}, taking the $L_\mu^2$-norm and using the triangle inequality we have
	\begin{align*}
	&\lVert \cL f_t^N(x,v)-\cL_Nf_t^N(x,v)\rVert_{L_\mu^2}  \leq \lVert \langle \vperp , \nabla_x f_t^N(x,v)\rangle_{\cH}\rVert_{L_\mu^2}+\lVert \langle \xperp, \nabla_v f_t^N(x,v)\rangle_{\cH}\rVert_{L_\mu^2} \label{eq:derivtermsnorm}\tag{$\dagger$}\\
	&+  \lref \left\lVert\int_{\cH} f_t^N(x,w) \nu_{0}(dw)-\int_{\cH} f_t^N(x,w^N+\vperp)\nu_{0}(dw)\right\rVert_{L_\mu^2} \label{eq:refreshtermsnorm}\tag{$\dagger\dagger$}\\
	&+ \left\lVert\sum_{n=1}^\infty \left( \lambda_n(x,v) [f_t^N(x,R_nv)-f_t^N(x,v)]-\lambda_n^N(x,v) [f_t^N(x,R_n^Nv)-f_t^N(x,v)]\right)\right\rVert_{L_\mu^2}.\label{eq:reflectiontermsnrom}\tag{$\dagger\dagger\dagger$}
	\end{align*}
	
	First let us consider \eqref{eq:derivtermsnorm}, by expanding in terms of $\{e_i\}_{i=1}^\infty$ and using \eqref{eq:projnablafN} we have
	\begin{align}
	\lVert \langle \vperp , \nabla_x f_t^N(x,v)\rangle_{\cH}\rVert_{L_\mu^2}^2 &=\int_{\cH^2} \sum_{i=N+1}^\infty v_i^2 (\partial_{x_i}f_t^N)^2 \mu(dx,dv) \leq \lVert \nabla_xf \rVert_\infty^2 \sum_{i=N+1}^\infty \gamma^2_i,\\
	\lVert \langle \xperp , \nabla_v f_t^N(x,v)\rangle_{\cH}\rVert_{L_\mu^2}^2 &=\int_{\cH^2} \sum_{i=N+1}^\infty x_i^2 (\partial_{v_i}f_t^N)^2 \mu(dx,dv) \leq \lVert \nabla_vf \rVert_\infty^2\lVert  \xperp \rVert_{L_\pi^2}^2 \label{eq:estfordagger},
	\end{align}
	Since we are assuming that $\Phi$ is bounded from below we can estimate $\lVert  \xperp \rVert_{L_\mu^2(\cH;\cH)}$, let $C_\Phi=\inf_{x\in \cH}\Phi(x)$ then
	\begin{equation}\label{eq:Gaussiantailx}
	\lVert  \xperp\rVert_{L_\mu^2(\cH;\cH)}^2 \leq e^{-C_{\Phi}} \sum_{i=N+1}^\infty \gamma^2_i.
	\end{equation}
	Therefore applying \eqref{eq:Gaussiantailx} to \eqref{eq:estfordagger} we have
	\begin{align}\label{eq:estfordaggerfinal}
	&\lVert \langle \vperp , \nabla_x f_t^N(x,v)\rangle_{\cH}\rVert_{L_\mu^2}+\lVert \langle \xperp , \nabla_v f_t^N(x,v)\rangle_{\cH}\rVert_{L_\mu^2} \\
	&\leq \left( \lVert \nabla_x f \rVert_\infty + e^{-\frac{1}{2}C_\Phi} \lVert \nabla_v f \rVert_\infty\right)\left(\sum_{i=N+1}^\infty \gamma^2_i\right)^\frac{1}{2}.
	\end{align}
	
	By the Fundamental Theorem of Calculus for any $g\in C_b^1(\cH)$ we have
	\begin{equation}\label{eq:FTC}
	\lvert g(x,w)-g(x,w^N+\vperp)\rvert \leq \int_0^1 \lvert \langle\nabla_v g(x,sw^N_\perp+(1-s)\vperp+w^N), w^N_\perp-\vperp \rangle \rvert ds \leq \lVert (1-\Proj_N)\nabla_vg\rVert_\infty \lVert w^N_\perp-\vperp\rVert_\cH.
	\end{equation}
	If we apply this with the function $g$ replaced by $f_t^N$ then 
	\begin{equation*}
	\lvert f_t^N(x,w)-f_t^N(x,w^N+\vperp)\rvert \leq \lVert\nabla_vf\rVert_\infty \lVert w_\perp^N-\vperp\rVert_\cH.
	\end{equation*}
	Now integrating with respect to $\mu$ we obtain an estimate for \eqref{eq:refreshtermsnorm}
	\begin{align*}
	\lvert\eqref{eq:refreshtermsnorm}\rvert^2 &\leq \lref^2\int_\cH\int_\cH\lvert f_t^N(x,w)-f_t^N(x,w^N+\vperp)\rvert^2\mu(dx,dv)\nu_0(dw) \\
	&\leq \lref^2\lVert\nabla_vf\rVert_\infty^2 \int_\cH\int_\cH\lvert w^N_\perp-\vperp\rvert^2 \nu_0(dw)\nu_0(dv) \leq 2\lref^2\lVert\nabla_vf\rVert_\infty^2 \sum_{i=N+1}^\infty \gamma^2_i.
	\end{align*}
	
	It remains to consider \eqref{eq:reflectiontermsnrom}. 
	\begin{align}
	&\left\lVert \sum_{n=1}^\infty \left(\lambda_n(x,v) [f_t^N(x,R_nv)-f_t^N(x,v)]-\lambda_n^N(x,v) [f_t^N(x,R_n^Nv)-f_t^N(x,v)]\right)\right\rVert_{L_\mu^2}\nonumber\\
	&\leq\lVert  \sum_{n=1}^\infty(\lambda_n(x,v)-\lambda_n^N(x,v)) [f_t^N(x,R_n^Nv)-f_t^N(x,v)]\rVert_{L_\mu^2}\nonumber\\
	&+ \lVert \sum_{n=1}^\infty\lambda_n(x,v) [f_t^N(x,R_nv)-f_t^N(x,R_n^Nv)]\rVert_{L_\mu^2}\nonumber\\
	&\leq 2\lVert f\rVert_\infty\left\lVert \sum_{n=1}^\infty( \lambda_n-\lambda_n^N) \right\rVert_{L_\mu^2}\\
	&+ \left\lVert \sum_{n=1}^\infty\lambda_n(x,v) [f_t^N(x,R_nv)-f_t^N(x,R_n^Nv)]\right\rVert_{L_\mu^2}.
	\end{align}
We can bound these terms using \eqref{eq:lambdaapprox} and \eqref{eq:Rapprox}.
	
	So combining all these estimates we have
	\begin{align*}
	&\lVert \cL f_t^N(x,v)-\cL_Nf_t^N(x,v)\rVert_{L_\mu^2}  \leq C\left( \lVert f\rVert_\infty+\lVert \nabla_x f \rVert_\infty + \lVert \nabla_v f \rVert_\infty + \lVert\nabla_vf\rVert_\infty\right)\left(\sum_{i=N+1}^\infty \gamma^2_i\right)^\frac{1}{2}.
	\end{align*}
\end{proof}

Now we will prove for every $f\in C_b^1(\cH\times \cH)$ that $\cP_t^Nf$ converges to $\cP_tf$ in $L_\mu^2$, uniformly for $t$ in compact intervals. 

\begin{theorem}\label{thm:finitedimconvfiniteT}
	Let $\{\cP_t\}$ be semigroup corresponding to the generator \eqref{eq:simgen}.  Assume that $\Phi$ is twice continuously differentiable, convex, bounded from below and has bounded Hessian. Let $\{\cP_t^N\}_{t\geq 0}$ be defined by \eqref{eq:semigpdeffinitedim}. Then for each $f\in C_b^1(\cH^2)$, $\cP_t^Nf$ converges to $\cP_t$ uniformly in time, that is for any $T>0$ and $f\in C_b^1(\cH^2)$ 
	\begin{equation*}
	\lim_{N\to \infty} \sup_{t\in[ 0,T]} \left\lVert \cP_tf-\cP_tf^N \right\rVert_{L_\mu^2} =0.
	\end{equation*}
	Moreover, fix $T>0$ then for any $f\in C_b^1(\cH^2)$ there is a constant $C=C(f,\Phi)$ such that for $N$ sufficiently large
	\begin{align*}
	\sup_{0\leq t\leq T}\lVert \cP_tf-\cP_t^Nf \rVert_{L_\mu^2}&\leq C(f,\Phi) T \left(\sum_{i=N+1}^\infty \gamma^2_i\right)^{\frac{1}{2}}
	\end{align*}
\end{theorem}

\begin{proof}[Proof of Theorem \ref{thm:finitedimconvfiniteT}]
Fix $f\in C_b^1(\cH^2)$, we can write the difference of the semigroups in terms of their generators using as follows:
\begin{align*}
\lVert \cP_tf-\cP_t^Nf \rVert_{L_\mu^2} &= \lVert \int_0^t \partial_s\cP_{t-s}\cP_s^Nf ds \rVert_{L_\mu^2}\\
&\leq  \int_0^t \lVert\cP_{t-s}(\cL_N-\cL)\cP_s^Nf \rVert_{L_\mu^2}ds 
\end{align*}
Now since $\cP_t$ is a contraction semigroup on $L_\mu^2$ we have
\begin{align*}
\lVert \cP_tf-\cP_t^Nf \rVert_{L_\mu^2} &\leq  \int_0^t \lVert(\cL_N-\cL)\cP_s^Nf \rVert_{L_\mu^2}ds 
\end{align*}
By Lemma \ref{lem:estimatediffofoperators} we bound this with
\begin{align*}
\lVert \cP_tf-\cP_t^Nf \rVert_{L_\mu^2} &\leq  C t\left( \lVert \nabla_x f \rVert_\infty + \lVert \nabla_p f \rVert_\infty +\lVert f\rVert_\infty\right)\left(\sum_{i=N+1}^\infty \gamma^2_i\right)^\frac{1}{2}.
\end{align*}
\end{proof}

\begin{example}\label{ex:explicitrateex}
	Set $\cH=\ell^2$ to be the space of square integrable sequences and fix $s>1$ and define the operator 
	\begin{equation*}
	\Sigma = (n^{-s}x_n)_{n=1}^\infty.
	\end{equation*}
	Note this is a bounded, self-adjoint, positive operator of Trace class. Set $\pi_0=\mathcal{N}(0,\Sigma)$ and define $\Phi(x)=\lVert x\rVert_{\ell^2}^2/2$. Then the measure $\pi$ is given by
	\begin{equation*}
	\frac{d\pi}{d\pi_0}(x) = e^{-\frac{1}{2}\lVert x\rVert_{\ell^2}^2}\prod_{n=1}^\infty (1+n^{-s}).
	\end{equation*}
	 It is immediate to check that $\Phi$ is smooth, bounded below and has bounded Hessian so the conditions of Theorem \ref{thm:finitedimconvfiniteT} are satisfied. Therefore we can construct the pure reflection Boomerang sampler for this setting and a finite dimensional approximation with the canonical basis which converges on compact time intervals. As $\Phi$ is convex we will also obtain convergence uniform in time by Theorem \ref{thm:uniformfinitedimconv}. The rate of convergence is determined by 
	\begin{align*}
	  \sum_{i=N+1}^\infty \frac{1}{i^{s}} =: \zeta(s,N+1).
	\end{align*}
	Here $\zeta$ is the Hurwitz zeta function. To leading order $\zeta(s,N+1) \approx N^{1-s}/(1-s)$ so we may choose $s$ to obtain any polynomial rate. 
\end{example}

So far we have proven convergence for any finite time, but as we are interested in convergence to equilibria uniform in time estimates are more helpful. In order to prove these estimates we shall use the Hypocoercivity of Section \ref{sec:expconv}. The following theorem is inspired by \cite{CDO} in which exponential decay of the derivatives of a semigroup are used to prove uniform in time convergence of the weak error between an Euler scheme approximation and an SDE. This technique is also used in \cite{BDOZ} to show uniform in time convergence between a SDE on a fast dynamical network and an averaged SDE. In both of these papers the results rely upon obtaining exponential decay of the derivatives of the semigroup, conditions for such an estimate are given in \cite{CrisanOttobre}. In contrast for a PDMP the semigroup need not be even differentiable in all directions so we can not hope to have an estimate which decays exponentially. Instead of the derivative estimate we shall use the exponential convergence from the Hypocoercivity theory established in Section \ref{sec:expconv}. This has the advantage that we don't require derivative estimates however we instead need control on the error between the approximate invariant measure $\mu_N$ and the invariant measure for the infinite dimensional process $\mu$.    

\begin{theorem}\label{thm:uniformfinitedimconv}
	Let $\{\cP_t\}$ be semigroup corresponding to the generator \eqref{eq:simgen}. Assume that $\Phi$ is twice continuously differentiable, convex (or bounded with bounded derivative), bounded from below and has bounded Hessian. Let $\{\cP_t^N\}_{t\geq 0}$ be the corresponding to the finite dimensional approximation constructed and suppose that
	\begin{equation}\label{eq:invmeasureconverges}
	 \lim_{N\to\infty}\left\lVert \frac{d\mu}{d\mu_N}-1\right\rVert_{L_{\mu_N}^2} =0.
	\end{equation}
	Then for each $f\in C_b^1(\cH^2)$, $\cP_t^Nf$ converges to $\cP_t$ uniformly in time, that is for any $f$ bounded and measurable
	\begin{equation*}
	\lim_{N\to \infty} \sup_{t\geq 0} \left\lVert \cP_tf-\cP_tf^N \right\rVert_{L_\mu^2} =0.
	\end{equation*}
	Moreover, for any $f\in C_b^1(\cH^2)$ there is a constant $C=C(f,\Phi)$ such that for $N$ sufficiently large
	\begin{align*}
	\lVert \cP_tf-\cP_t^Nf \rVert_{L_\mu^2}&\leq C(f,\Phi)\left(\sum_{i=N+1}^\infty \gamma^2_i\right)^{\frac{1}{2}} +  \left\lVert \frac{d\mu}{d\mu_N}-1\right\rVert_{L_{\mu_N}^2}.
	\end{align*}
\end{theorem}

\begin{proof}[Proof of Theorem \ref{thm:uniformfinitedimconv}]
Fix $f\in C_b^1(\cH^2)$, we can write the difference of the semigroups in terms of their generators as follows:
\begin{align*}
\lVert \cP_tf-\cP_t^Nf \rVert_{L_\mu^2} &= \lVert \int_0^t \partial_s\cP_{t-s}\cP_s^Nf ds \rVert_{L_\mu^2}\\
&\leq  \int_0^t \lVert\cP_{t-s}(\cL_N-\cL)\cP_s^Nf \rVert_{L_\mu^2}ds 
\end{align*}
We wish to use hypocoecivity to control this, to which end we add and subtract $\mu((\cL_N-\cL)\cP_s^Nf )$. Note that $\mu(\cL\cP_s^Nf) =0$ since $\mu$ is an invariant measure for $\cP_t$.
\begin{align*}
\lVert \cP_tf-\cP_t^Nf \rVert_{L_\mu^2}\leq  \int_0^t \lVert\cP_{t-s}(\cL_N-\cL)\cP_s^Nf  - \mu(\cL_N\cP_s^Nf)\rVert_{L_\mu^2} +\lvert\mu(\cL_N\cP_s^Nf)\rvert  ds
\end{align*}

 By Theorem \ref{thm:Hypocoecivity} we have
\begin{align*}
\lVert \cP_tf-\cP_t^Nf \rVert_{L_\mu^2}&\leq  \int_0^t  Ce^{-\kappa (t-s)}\lVert(\cL_N-\cL)\cP_s^Nf - \mu(\cL_N\cP_s^Nf) \rVert_{L_\mu^2} + \lvert\mu(\cL_N\cP_s^Nf)\rvert ds \\
&\leq   \int_0^t Ce^{-\kappa (t-s)}\lVert(\cL_N-\cL)\cP_s^Nf \rVert_{L_\mu^2} + (C+1)\lvert\mu(\cL_N\cP_s^Nf)\rvert ds.
\end{align*}

Let us consider the term $\lvert\mu(\cL_N\cP_s^Nf)\rvert$.
\begin{align*}
\lvert\mu(\cL_N\cP_s^Nf)\rvert &= \lvert\mu(\cL_N\cP_s^Nf) - \mu_N(\cL_N\cP_s^Nf)\rvert\rvert\\
& = \mu_N\left((\cL_N\cP_s^Nf)\left(\frac{d\mu}{d\mu_N}-1\right)\right)\\
&\leq \lVert \cL_N\cP_s^Nf \rVert_{L_{\mu_N}^2} \lVert \frac{d\mu}{d\mu_N}-1\rVert_{L_{\mu_N}^2}
\end{align*}
Under our assumptions we have that $\cP_t^N$ is hypocoercive with constants independent of $N$, therefore
\begin{align*}
\lvert\mu(\cL_N\cP_s^Nf)\rvert \leq Ce^{-\rho t}\lVert \cL_Nf \rVert_{L_{\mu_N}^2} \lVert \frac{d\mu}{d\mu_N}-1\rVert_{L_{\mu_N}^2}.
\end{align*}
Hence
\begin{align*}
\lVert \cP_tf-\cP_t^Nf \rVert_{L_\mu^2}&\leq   \int_0^t Ce^{-\kappa (t-s)}\lVert(\cL_N-\cL)\cP_s^Nf \rVert_{L_\mu^2}ds + \frac{C(C+1)}{\rho}\lVert \cL_Nf \rVert_{L_{\mu_N}^2} \lVert \frac{d\mu}{d\mu_N}-1\rVert_{L_{\mu_N}^2}.
\end{align*}

It remains to show that 
\begin{equation*}
\lim_{N\to \infty}\sup_{t\geq 0}\int_0^t e^{-\kappa (t-s)}\lVert(\cL_N-\cL)\cP_s^Nf \rVert_{L_\mu^2}ds=0.
\end{equation*}
This follows since we can bound the integrand using Lemma \ref{lem:estimatediffofoperators} which gives
\begin{align*}
\int_0^t e^{-\kappa (t-s)}\lVert(\cL_N-\cL)\cP_s^Nf \rVert_{L_\mu^2}ds&\leq \int_0^t e^{-\kappa (t-s)}ds C \left( \lVert \nabla_x f \rVert_\infty + \lVert \nabla_v f \rVert_\infty +\lVert f\rVert_\infty\right)\left(\sum_{i=N+1}^\infty \gamma^2_i\right)^\frac{1}{2}\\
&=\frac{1}{\kappa}(1-e^{-\kappa t}) C \left( \lVert \nabla_x f \rVert_\infty + \lVert \nabla_v f \rVert_\infty +\lVert f\rVert_\infty\right)\left(\sum_{i=N+1}^\infty \gamma^2_i\right)^\frac{1}{2}
\end{align*}
\end{proof}

\begin{example}[Example \ref{ex:explicitrateex} continued]
We showed in Example \ref{ex:explicitrateex} how to construct an example with an arbitrarily chosen polynomial rate of convergence. We wish to continue this example to show in this case that we get the same order of convergence uniform in time. Note that as $\Phi$ is convex we have that all the conditions of Theorem \ref{thm:uniformfinitedimconv} are satisfied provided \eqref{eq:invmeasureconverges} holds. If we show that
\begin{equation*}
    \lVert \frac{d\mu}{d\mu_N} -1\rVert_{L_{\mu_N}^2}^2 \leq C\sum_{i=N+1}^\infty\gamma_i^2
\end{equation*}
holds then \eqref{eq:invmeasureconverges} is satisfied and the rate of convergence of $\cP_t^Nf$ is the same as in Example \ref{ex:explicitrateex}. Note that if we don't require the rate of convergence then it is immediate to check that \eqref{eq:invmeasureconverges} holds since $d\mu/d\mu_N$ is bounded uniformly in $N$ by $2$ so the results follows by the dominated convergence theorem. For this example 
\begin{equation*}
    \frac{d\mu}{d\mu_N} = \frac{d\pi}{d\pi_N} = \frac{\exp(-\lVert \xperp\rVert_\cH^2/2)}{\pi_0(\exp(-\lVert \xperp\rVert_\cH^2/2))}.
\end{equation*}
Then we can write 
\begin{equation*}
    \lVert \frac{d\mu}{d\mu_N} -1\rVert_{L_{\mu_N}^2}^2 =\int_{\cH} \left\lvert \frac{\exp(-\lVert \xperp\rVert_\cH^2/2)}{\pi_0(\exp(-\lVert \xperp\rVert_\cH^2/2))} -1\right\rvert^2 \frac{\exp(-\lVert \xN\rVert_\cH^2/2)}{\pi_0(\exp(-\lVert \xN\rVert_\cH^2/2))} \pi_0(dx) . 
\end{equation*}
Under $\pi_0$ $\xN$ is independent of $\xperp$ so this simplifies to
\begin{equation*}
    \lVert \frac{d\mu}{d\mu_N} -1\rVert_{L_{\mu_N}^2}^2 =\int_{\cH} \left\lvert \frac{\exp(-\lVert \xperp\rVert_\cH^2/2)}{\pi_0(\exp(-\lVert \xperp\rVert_\cH^2/2))} -1\right\rvert^2 \pi_0(dx) =\int_{\cH}  \frac{\exp(-\lVert \xperp\rVert_\cH^2)}{\pi_0(\exp(-\lVert \xperp\rVert_\cH^2/2))^2} \pi_0(dx)-1   . 
\end{equation*}
Again using the independence structure of $\pi_0$ we can rewrite this as
\begin{equation*}
    \lVert \frac{d\mu}{d\mu_N} -1\rVert_{L_{\mu_N}^2}^2 =\prod_{i>N} \left[ \frac{\pi_0(\exp(-x_i^2))}{\pi_0(\exp(-x_i^2/2))^2} \right]-1   . 
\end{equation*}
As $\pi_0$ is a Guassian measure we can evaluate this integrals and obtain
\begin{equation*}
    \lVert \frac{d\mu}{d\mu_N} -1\rVert_{L_{\mu_N}^2}^2 =\prod_{i>N} \left[ \frac{\gamma_i^2+1}{\sqrt{2\gamma_i^2+1}} \right]-1  \sim \frac{1}{2}\sum_{i>N} \gamma_i^4 . 
\end{equation*}
\end{example}

\subsection{Finite dimensional approximation of Pure reflection Boomerang Sampler}

In this section we verify that \eqref{eq:lambdaapprox} and \eqref{eq:Rapprox} both hold for the approximation of the  Pure reflection Boomerang Sampler given in Example \ref{ex:finitedimPurerefl}.


\begin{lemma}
	Let $R_N$ and $\lambda_N$ be as in Example \ref{ex:finitedimPurerefl} then \eqref{eq:lambdaapprox} and \eqref{eq:Rapprox} hold.
\end{lemma}

\begin{proof}
	First consider \eqref{eq:lambdaapprox}. In this case the left hand side of \eqref{eq:lambdaapprox} simplifies to
	\begin{align*}
	\lVert\sum_{n=1}^\infty (\lambda_n-\lambda_n^N)\rVert  &= \lVert (\langle\nabla_x\Phi(x),v\rangle_\cH)_+ - (\langle\nabla_x\Phi(x_N),v_N\rangle_\cH)_+\rVert \\
	&\leq \lVert \langle\nabla_x\Phi(x)- \nabla_x\Phi(x_N),v\rangle_\cH\rVert\\
	&\leq \sqrt{\Tr(\Sigma)}\lVert \nabla_x\Phi(x)-\nabla_x\Phi(x_N)\rVert\\
	&\leq  \sqrt{\Tr(\Sigma)}\lVert \nabla_x^2\Phi\rVert_\infty\lVert x-x_N\rVert.
	\end{align*}
	Then \eqref{eq:lambdaapprox} follows by \eqref{eq:Gaussiantailx}. 
	
	Now consider \eqref{eq:Rapprox} the left hand side of which is
	\begin{align*}
	\sum_{n=1}^\infty\lVert \lambda_n(x,v) \lVert R_nv-R_n^Nv\rVert_\cH\rVert_{L_\mu^2}&=2\lVert (\langle\nabla_x\Phi(x),v\rangle_\cH)_+ \lVert \vperp\rVert_\cH\rVert_{L_\mu^2}\\
	&\leq 2\left(\sum_{i=1}^\infty\sum_{j>N}\int_{\cH^2} (\partial_{x_i}\Phi(x))^2v_i^2 v_j^2\mu(dx,dv)\right)^{\frac{1}{2}}.
	\end{align*}
	Now since $e_i$ are eigenvectors of $\Sigma$ we have that $\int v_i^2v_j^2d\nu=\delta_{ij} 3\gamma_j^4+\gamma^2_i\gamma^2_j(1-\delta_{ij})$ so 
	\begin{align*}
	&\sum_{n=1}^\infty\lVert \lambda_n(x,v) \lVert R_nv-R_n^Nv\rVert_\cH\rVert_{L_\mu^2}\\
	&= 2\left(\sum_{j>N}\gamma^2_j\sum_{\substack{i=1\\i\neq j}}^\infty\gamma^2_i\int_{\cH^2} (\partial_{x_i}\Phi(x))^2\pi(dx)+\sum_{j>N}3\gamma_j^4\int_{\cH^2} (\partial_{x_j}\Phi(x))^2\pi(dx)\right)^{\frac{1}{2}}\\
	&\leq 4\lVert \Sigma\nabla_x\Phi\rVert_{L_\pi^2}\left(\sum_{j>N}\gamma^2_j\right)^{\frac{1}{2}}
	\end{align*}
\end{proof}

\subsection{Finite dimensional approximation of Factorised Boomerang Sampler}

In this section we verify that \eqref{eq:lambdaapprox} and \eqref{eq:Rapprox} both hold for the approximation of the  Factorised Boomerang Sampler given in Example \ref{ex:finitedimfactor}..

\begin{lemma}
	Let $R_N$ and $\lambda_N$ be as in Example \ref{ex:finitedimfactor} then \eqref{eq:lambdaapprox} and \eqref{eq:Rapprox} hold.
\end{lemma}

\begin{proof}
	First consider \eqref{eq:lambdaapprox}. In this case the left hand side of \eqref{eq:lambdaapprox} simplifies to
	\begin{align*}
	\lVert \sum_{n=1}^\infty\lvert\lambda_n-\lambda_n^N\rvert\rVert^2  &=\int_{\cH^2} \sum_{n,m} \lvert \lambda_n(x,v)-\lambda_n^N(x,v)\rvert \lvert \lambda_m(x,v)-\lambda_m^N(x,v)\rvert\mu(dx,dv)\\
	&=\int_{\cH^2} \sum_{n,m\leq N} \lvert \lambda_n-\lambda_n^N\rvert \lvert \lambda_m-\lambda_m^N\rvert d\mu+\int_{\cH^2} \sum_{n,m>N} \lvert \lambda_n-\lambda_n^N\rvert \lvert \lambda_m-\lambda_m^N\rvert d\mu\\
	&+2\int_{\cH^2} \sum_{n\leq N,m>N} \lvert \lambda_n-\lambda_n^N\rvert \lvert \lambda_m-\lambda_m^N\rvert d\mu\\
		&\leq\int_{\cH^2} \sum_{n,m\leq N} \lvert v_nv_m\rvert \lvert \partial_n\Phi(x)-\partial_n\Phi(x_N)\rvert \lvert \partial_m\Phi(x)-\partial_m\Phi(x_N)\rvert d\mu\\
		&+\int_{\cH^2} \sum_{n,m> N} \lvert v_nv_m\rvert \lvert \partial_n\Phi(x)\rvert \lvert \partial_m\Phi(x)\rvert d\mu\\
		&+2\int_{\cH^2} \sum_{n\leq N,m> N} \lvert v_nv_m\rvert \lvert \partial_n\Phi(x)-\partial_n\Phi(x_N)\rvert \lvert \partial_m\Phi(x)\rvert d\mu
	\end{align*}
	Note that $v$ is independent of $x$ under $\mu$ and $\nu_0(\lvert v_nv_m\rvert) \leq \sqrt{\gamma^2_n\gamma^2_m}$. 
		\begin{align*}
	&\lVert \sum_{n=1}^\infty\lvert\lambda_n-\lambda_n^N\rvert\rVert^2  
	\leq\int_\cH \sum_{n,m\leq N} \sqrt{\gamma^2_n\gamma^2_m} \lvert \partial_n\Phi(x)-\partial_n\Phi(x_N)\rvert \lvert \partial_m\Phi(x)-\partial_m\Phi(x_N)\rvert d\pi\\
	&+\int_\cH \sum_{n,m> N} \sqrt{\gamma^2_n\gamma^2_m} \lvert \partial_n\Phi(x)\rvert \lvert \partial_m\Phi(x)\rvert d\pi+2\int_\cH \sum_{n\leq N,m> N} \sqrt{\gamma^2_n\gamma^2_m} \lvert \partial_n\Phi(x)-\partial_n\Phi(x_N)\rvert \lvert \partial_m\Phi(x)\rvert d\pi\\
	&\leq\left(\sum_{n\leq N}  \sqrt{\gamma^2_n} \lVert \partial_n\Phi(x)-\partial_n\Phi(x_N)\rVert_{L_\pi^2}  \right)^2\\
	&+\left( \sum_{n> N} \sqrt{\gamma^2_n} \lVert \partial_n\Phi(x)\rVert_{L_\pi^2}  \right)^2+2 \sum_{n\leq N,m> N} \sqrt{\gamma^2_n\gamma^2_m} \lVert \partial_n\Phi(x)-\partial_n\Phi(x_N)\rVert_{L_\pi^2} \lVert \partial_m\Phi(x)\rVert_{L_\pi^2} 	\end{align*}
	Now use the Mean value theorem
			\begin{align*}
	\lVert \sum_{n=1}^\infty\lvert\lambda_n-\lambda_n^N\rvert\rVert^2  
	&\leq\left(\sum_{n\leq N}  \sqrt{\gamma^2_n} \lVert\nabla_x\partial_n\Phi\rVert_\infty\lVert x-x_N\rVert_{L_\pi^2}  \right)^2+\left( \sum_{n> N} \sqrt{\gamma^2_n} \lVert \partial_n\Phi(x)\rVert_{L_\pi^2}  \right)^2\\
	&+2 \sum_{n\leq N,m> N} \sqrt{\gamma^2_n\gamma^2_m} \lVert\nabla_x\partial_n\Phi\rVert_\infty\lVert x-x_N\rVert_{L_\pi^2} \lVert \partial_m\Phi(x)\rVert_{L_\pi^2} 	\end{align*}
	Using Cauchy-Schwartz
	\begin{align*}
	\lVert \sum_{n=1}^\infty\lvert\lambda_n-\lambda_n^N\rvert\rVert^2  
	&\leq\left(\sum_{n\leq N}  \gamma^2_n\right)\left(\sum_{n\leq N} \lVert\nabla_x\partial_n\Phi\rVert_\infty^2\right)\lVert x-x_N\rVert_{L_\pi^2} ^2+\left( \sum_{n> N} \gamma^2_n  \right)\left( \sum_{n> N} \lVert \partial_n\Phi(x)\rVert_{L_\pi^2}^2  \right)\\
	&+2 \lVert x-x_N\rVert_{L_\pi^2}\sqrt{\sum_{n\leq N} \gamma^2_n}\sqrt{\sum_{n\leq N}\lVert\nabla_x\partial_n\Phi\rVert_\infty^2}\sqrt{\sum_{m> N}\gamma^2_m }\sqrt{\sum_{m> N} \lVert \partial_m\Phi(x)\rVert_{L_\pi^2}^2} 	\\
	&\leq\Tr(\Sigma) \lVert\nabla_x^2\Phi\rVert_\infty^2\lVert x-x_N\rVert_{L_\pi^2} ^2+\left( \sum_{n> N} \gamma^2_n  \right)\lVert \nabla\Phi(x)\rVert_{L_\pi^2}^2 \\
	&+2 \lVert x-x_N\rVert_{L_\pi^2}\sqrt{\Tr(\Sigma)}\lVert\nabla_x^2\Phi\rVert_\infty\sqrt{\sum_{m> N}\gamma^2_m } \lVert \nabla_x\Phi(x)\rVert_{L_\pi^2}
	\end{align*}
	
	Then \eqref{eq:lambdaapprox} follows by \eqref{eq:Gaussiantailx}. 
	
	Now consider \eqref{eq:Rapprox} the left hand side of which is
	\begin{align*}
	&\lVert \sum_{n=1}^\infty\lambda_n(x,v) \lVert R_nv-R_n^Nv\rVert_\cH\rVert_{L_\mu^2}^2\\
	&=\int_{\cH^2} \sum_{n=1}^\infty\sum_{m=1}^\infty\lambda_n(x,v) \lambda_m(x,v)\lVert R_nv-R_n^Nv\rVert_\cH\lVert R_mv-R_m^Nv\rVert_\cH\pi(dx)\nu_0(dv).
	\end{align*}
	Now if $n\leq N$ then $R_n^N=R_n$ so all terms vanish except when $n>N$ and $m>N$. In that case $R_n^N=1$ so $R_nv-R_n^Nv=-2v_ne_n$, thus we have
	\begin{align*}
\lVert \sum_{n=1}^\infty\lambda_n(x,v) \lVert R_nv-R_n^Nv\rVert_\cH\rVert_{L_\mu^2}^2&=4\int_{\cH^2} \sum_{n>N}\sum_{m>N}(\partial_n\Phi(x)v_n)^+ (\partial_m\Phi(x)v_m)^+\lvert v_n\rvert\lvert v_m\rvert\pi(dx)\nu_0(dv),\\
&\leq 4\int_{\cH^2} \sum_{n>N}\sum_{m>N}\lvert\partial_n\Phi(x)\rvert \lvert\partial_m\Phi(x)\rvert\lvert v_n\rvert^2\lvert v_m\rvert^2\pi(dx)\nu_0(dv).
\end{align*}
Now $\int v_i^2v_j^2d\nu=\delta_{ij} 3\gamma_j^4+\gamma^2_i\gamma^2_j(1-\delta_{ij})$, in particular $\int v_i^2v_j^2d\nu\leq  3\gamma^2_i\gamma^2_j$. 
\begin{align*}
\lVert \sum_{n>N}\lambda_n(x,v) \lVert R_nv-R_n^Nv\rVert_\cH\rVert_{L_\mu^2}^2&\leq 12\int_{\cH} \sum_{n>N}\sum_{m>N}\gamma^2_n\gamma^2_m\lvert\partial_n\Phi(x)\rvert \lvert\partial_m\Phi(x)\rvert\pi(dx)\\
&\leq 12\int_{\cH} \left(\sum_{n>N}\gamma^2_n\lvert\partial_n\Phi(x)\rvert \right)^2\pi(dx)\\
&\leq 12\left(\sum_{n>N}\gamma^2_n\right)\int_{\cH} \left(\sum_{n>N}\gamma^2_n\lvert\partial_n\Phi(x)\rvert^2 \right)\pi(dx)\\
&\leq 12\lVert\Sigma\nabla_x\Phi(x)\rVert_{L_\pi^2}^2\left(\sum_{n>N}\gamma^2_n\right) .
\end{align*}
\end{proof}

\section{Proofs}\label{sec:proof}
\subsection{Proofs of Section \ref{sec:PDMP-examples}}\label{sec:proofsofexamples}
\subsubsection{Proofs of Section \ref{sec:ZZ}}\label{sec:proofs_zzs}

\begin{proof}[Proof of Theorem \ref{thm:ZZnonGausscase}]
To prove well-posedness of the piecewise deterministic Markov process we shall apply Theorem \ref{thm:normedspace}, to this end we show that Assumption \ref{ass:general} holds. 
\begin{itemize}
    \item[(i,viii)] 
    Using Fr\'echet differentiability of $\Phi$, this follows from the estimate
\[ |\lambda_i(x,v) - \lambda_i(y,v)| \leq \frac{\lvert v_i\rvert}{\gamma_i^2} |x_i - y_i| + \lvert v_i\rvert |\partial_i \Phi(x) - \partial_i \Phi(y)|.\]
Now summing over $i$ and using the Cauchy-Schwartz inequality we have
\begin{align*}
    \lvert \lambda(x,v)-\lambda(y,v)\rvert \leq \lVert \C^{-1}v\rVert \|x - y\| + \lVert v\rVert \lVert \nabla_x\phi(x)-\nabla_x\phi(y)\rvert.
\end{align*}
Therefore $x\mapsto\lambda(x,v)$ is continuous, continuity in the $v$ component is similar.
\item[(ii)] This is immediate from the definition of $(Q_i)$.
\item[(iii)] Is immediate since the absolute value of velocities are preserved by reflections.
\item[(iv)] There are no refreshments.
\item[(v)] This follows from \eqref{eq:boundforlambdaZZ} and Assumption \ref{ass:1} (3).
\item[(vi)] There are no refreshments.
\item[(vii)] The flow is given by $\varphi_t(x,v)=(x+tv,v)$ so we have that \ref{ass:flowbound} holds with $c_t$ growing linearly.
\end{itemize}
To show that $\mu$ is an invariant measure we shall appeal to Theorem \ref{thm:invmeas}, so we must show that \eqref{eq:absinvmeasurecond} holds. In this case $\LX f=\langle v,\nabla_xf\rangle_\cH$ then by \cite[Proposition 10.20]{DaPrato} the adjoint of $\LX $ in $L_{\pi_0}^2$ is given by
\begin{equation*}
    \LX ^*f = -\langle \nabla_xf,v\rangle_\cH+\langle \Sigma^{-1}v,x\rangle_\cH f.
\end{equation*}
Note that $\Sigma^{-1}v\in \cH$ so the above expression is well defined. Therefore we have
\begin{equation}\label{eq:invmeascondZZ1}
    e^{\Phi(x)}\LX ^*(e^{-\Phi})(x,v) = \left(\langle \nabla_x\Phi(x),v\rangle_\cH+\langle \Sigma^{-1}v,x\rangle _\cH\right) .
\end{equation}
Next we have 
\begin{equation*}
    Q^*((x,v),dw)=\sum_{i=1}^\infty \frac{\lambda_i(x,F_iv)}{\lambda(x,F_iv)}\delta_{F_iv}(dw).
\end{equation*}
Indeed we can verify \eqref{eq:Qadjoint} holds, the left hand side of \eqref{eq:Qadjoint} is
\begin{align*}
    \int_{\cH} \int_\cH g(v,w) Q^*((x,v),dw)\mu_0(dx,dv) = \sum_{i=1}^\infty\int_{\cH} \int_\cH \frac{\lambda_i(x,F_iv)}{\lambda(x,F_iv)}  g(v,F_iv) \mu_0(dx,dv).
\end{align*}
As $F_i$ leave $\mu_0$ invariant we can do a change of variables $v\to F_iv$ for each $i$.
\begin{align*}
    \int_{\cH} \int_\cH g(v,w) Q^*((x,v),dw)\mu_0(dx,dv) &= \sum_{i=1}^\infty\int_{\cH} \int_\cH \frac{\lambda_i(x,v)}{\lambda(x,v)}  g(F_iv,v) \mu_0(dx,dv)\\
    &=\int_{\cH} \int_\cH g(w,v) Q((x,v),dw)\mu_0(dx,dv).
\end{align*}
Thus \eqref{eq:Qadjoint} holds. Now
\begin{align*}
    \int_\cH \lambda(x,w)Q^*((x,v),dw)-\lambda(x,v) = \sum_{i=1}^\infty \frac{\lambda_i(x,F_iv)}{\lambda(x,F_iv)}\lambda(x,F_iv)-\lambda(x,v)
\end{align*}
Now using the definition of $\lambda_i$ we have
\begin{align}\label{eq:invmeascondZZ2}
    \int_\cH \lambda(x,w)Q^*((x,v),dw)-\lambda(x,v) = -\langle v,\nabla_x\Phi(x)\rangle_\cH - \langle \Sigma^{-1}v,x\rangle _\cH. 
\end{align}
Finally combining \eqref{eq:invmeascondZZ1} and \eqref{eq:invmeascondZZ2} we have \eqref{eq:absinvmeasurecond} holds. Then by Theorem \ref{thm:invmeas} we have $\mu$ is formally an invariant measure for the Zig Zag process.
\end{proof}

\subsubsection{Proofs of Section \ref{sec:BPS}}\label{sec:proofs_bps}

\begin{proof}[Proof of Proposition \ref{lem:choiceofalpha}]
Fix $x\in \h$ with $\C^{-1}x\in \h$. Under our assumptions,  $\nabla\Psi$ is  
$$
\nabla \Psi= \nabla \Phi(x)+ \C^{-1}x\,.
$$
Therefore 
\begin{align}\label{sp1}
\langle v,\nabla\Psi\rangle = \langle
v ,\nabla \Phi \rangle+ \langle
 \C^{-1}v ,x\rangle \,.
\end{align}
By Assumption \ref{ass:1} the first addend on the right hand side of the above is well defined and since $\C^{-1}v\in \h$ the second addend is also well defined for all $x\in \h, v\in \cVBP$. Now by continuity we have that $\langle v,\nabla\Psi\rangle$ is well defined for all  $x\in \h, v\in \cVBP$, and hence $\lambda$ is also. 

In order to prove the statement (2) it suffices to show that $\C^{-1}\C^{\zeta} \nabla\Psi(x)\in\h$ for every $x \in \h$.  Indeed, if this is the case, then it is easy to see that also $\lVert\C^{\zeta/2}\nabla\Psi\rVert$ is finite. To show that $\C^{\zeta-1} \nabla\Psi(x)\in\h$ we act as above and write for every $x\in \h,v\in \cVBP$ with $\C^{-1}x\in\h$ that
\begin{align*}
\C^{\zeta-1}\nabla\Psi= \C^{\zeta-1}\nabla \Phi(x)+ \C^{\zeta-2}x 
\end{align*}
Imposing that $\zeta\geq 2$ we ensure this is well defined.
\end{proof}

\begin{proof}[Proof of Proposition \ref{propIDBP}]
To show (1), 
\begin{align*}
R(x)[R(x)v] & = R(x) \left[ v - \frac {2 \langle v,  \nabla \Psi \rangle}{\| \C^{\zeta/2} \nabla \Psi \|^2} \C^{\zeta} \nabla \Psi\right]\\
& = v - \frac {2 \langle v,   \nabla \Psi \rangle}
{\| \C^{\zeta/2} \nabla \Psi \|^2} \C^{\zeta} \nabla \Psi-
2 \frac{\C^{\zeta} \nabla \Psi}{\| \C^{\zeta/2} \nabla \Psi \|^2}
\left( \lan v , \nabla \Psi \ran - 2 \frac{\lan v,  \nabla \Psi \ran }{\|\C^{\zeta/2}\nabla \Psi\|^2}\lan \C^{\zeta} \nabla \Psi,   \nabla \Psi\ran\right)\\
& = v+ \left[-4 \frac { \langle v,   \nabla \Psi \rangle}
{\| \C^{\zeta/2} \nabla \Psi \|^2} + 4
\frac{\lan v, \nabla \Psi \ran}{\|\C^{\zeta/2} \nabla\Psi\|^4}\lan \C^{\zeta}\nabla\Psi, \nabla\Psi\ran
\right] \C^{\zeta} \nabla \Psi  \, =v.
\end{align*}
Therefore the claim follows. As for point (2):
\begin{align*}
\lan R(x)v, \nabla \Psi \ran & =
\lan v, \nabla \Psi \ran - \frac {2 \langle v,   \nabla \Psi \rangle}{\|\C^{\zeta/2} \nabla \Psi \|^2} \langle \C^{\zeta} \nabla \Psi , \nabla \Psi \ran \,.
\end{align*}
Hence, under the constraints on the parameters stated in point (2), \eqref{step2} holds. To prove the statement in point (3), recall that the Fourier transform $\hat{\nu}$ of a measure $\nu$ on $\h$ 
$$
\hat{\nu}(\xi)= \int_{\ha}e^{i \langle z,\xi \rangle}d\nu(z) 
$$
and that a measure $\nu_0$ is a centered  Gaussian with covariance operator $\C^\zeta$ on $\h$ if and only if 
$$
\hat{\nu}_0(\xi) = e^{-\frac{1}{2} \langle \C^\zeta \xi, \xi \rangle} \,.
$$
Because Fourier transforms characterize  measures (see \cite[Proposition 1.7]{DaPrato}), we just need to impose $\hat{\nu}_0=\hat{\nu}_R$. To this end, start by defining the formal adjoint operator $R_x'$ of $R(x)$, namely the operator 
$$
R_x'\xi: =\xi-2 \frac{\langle \xi, \C^{\zeta} \nabla \Psi \rangle}{\|\C^{\zeta/2} \nabla \Psi\|^2}  \nabla \Psi \,.
$$
With this definition one has
$$
\langle R(x) z, \xi \rangle = \langle z, R(x)'\xi\rangle \, .
$$
Let us now calculate $\hat{\nu}_R$:
\begin{align*}
\hat{\nu}_R(\xi) & = \int_{\h}
e^{i \langle z, \xi \rangle }d\nu_R(z) = 
\int_{\h} e^{i \langle R(x)z, \xi \rangle }d\nu_0(z)\\
&= \int_{\h} e^{i \langle z, R_x'\xi \rangle }d\nu_0(z) = e^{-\frac{1}{2}\langle \C^\zeta (R_x'\xi), (R_x'\xi) \rangle } \,
\end{align*}
where in the first equality we have used the involutivity of $R_x$, which gives the standard change of variables formula
$$
\int g(z) \, d (\nu \circ R_x)(z) = \int (g \circ R_x)(z) \,  d\nu(z) \,.
$$
In order to have $\hat{\nu}_R=\hat{\nu}_0$, we then need to have 
$$
\langle \C^\zeta (R_x'\xi), (R_x'\xi) \rangle = \langle \C^\zeta \xi , \xi \rangle, \mbox{ for every } \xi \in \h . 
$$
Because 
\begin{align*}
   \langle \C^\zeta (R_x'\xi), (R_x'\xi) \rangle = \langle \C^\zeta \xi , \xi \rangle -4  \frac{\langle \xi, \C^{\zeta} \nabla \Psi \rangle}{\|\C^{\zeta/2} \nabla \Psi\|^2}\langle \C^\zeta \xi, \nabla \psi\rangle
   +4 \frac{\left\vert\langle \xi, \C^{\zeta} \nabla \Psi \rangle \right\vert^2}{\|\C^{\zeta/2} \nabla \Psi\|^4}
   \langle \C^{\zeta}\nabla\Psi,\nabla\Psi \rangle =\langle \C^\zeta \xi , \xi \rangle , 
\end{align*}
thus the reflection invariance of $\nu_0$ does hold. 

To prove statement (4) we show that \eqref{eq:absinvmeasurecond} holds. Since IDBPS fits in the setting of Example \ref{ex:detrefl} it is sufficient to verify the conditions given in that example, note that the first two conditions are satisfied since the reflection operator only acts on the $v$ component and by statement (3) of this proposition so it sufficies to show that \eqref{eq:detswitchrelcond} holds. As in the proof of invariance for the IDZZS we have that \eqref{eq:invmeascondZZ1} holds so it remains to show that
\begin{equation}\label{eq:suffcondinvBPS}
\langle \nabla_x\Phi(x),v\rangle+\langle \Sigma^{-1}v,x\rangle+ \lambda(x,R(x)v)-\lambda(x,v)=0
\end{equation}
By \eqref{step2} we have that 
\begin{equation*}
    \lambda(x,R(x)v)-\lambda(x,v) = -\langle \C^\eta v, \nabla\Psi\rangle =-\langle \nabla_x\Phi(x),v\rangle-\langle \Sigma^{-1}v,x\rangle.
\end{equation*}
Therefore \eqref{eq:suffcondinvBPS} holds.

\end{proof}

\subsubsection{Proof of Section \ref{sec:Boomerang}}\label{sec:proofs_Boomerang}

\begin{proof}[Proof of Proposition \ref{prop:IM}]
	Note that refreshment clearly preserves the invariant measure so without loss of generality we may set $\lref=0$. We first prove that \eqref{eq:Lmu=0onC0} holds by applying Theorem \ref{thm:invmeas}. To do so we must determine $\LX ^*$ and  $Q^*$. First let us consider $\LX $. The set $\mathbb{C}:=C^1(\cH^2)\cap L_{\mu_0}^2$ is a core\footnote{This follows since $\mathbb{C}$ is invariant under $e^{t\LX }$, the semigroup of $\LX $, and is dense in $L_{\mu_0}^2$ as it contains the Fourier basis. Therefore $\mathbb{C}$ is a core for $\LX $ by \cite[Proposition B.1.10]{Lorenzi}.} for $\LX $ in $L_{\mu_0}^2$. By Lemma \ref{lem:L0skewsym} we have that $\mathbb{C}\subseteq D(\LX ^*)$ and $\LX ^*f=-\LX f$. By Hypothesis \ref{hyp:IMass} we have that $e^{-\Phi}\in \mathbb{C}$ and hence $\LX ^*e^{-\Phi}=-\LX e^{-\Phi}$.
	
	Similar to Example \ref{ex:finitedimzigzag} we have that
	\begin{equation*}
		Q^*(z,dy) = \sum_{i=1}^\infty\frac{\lambda_i(R_iz)}{\tilde{\lambda}(R_iz)} \delta_{R_iz}.
	\end{equation*}
	Now we can write the left hand side of \eqref{eq:absinvmeasurecond} as
	\begin{align}
	&e^{\Phi(z)}\LX ^*(e^{-\Phi(\cdot)})(z)+\int_H \tilde{\lambda}(y)e^{\Phi(z)-\Phi(y)}Q^*(z,dy) -\tilde{\lambda}(z) \\
	&= 	\LX (\Phi)(z)+\sum_{i=1}^\infty\lambda_i(R_iz)e^{\Phi(z)-\Phi(R_iz)} -\tilde{\lambda}(z)\label{eq:testinvmeasass}
	\end{align}
	Recall $\Phi(R_iz)=\Phi(x,R_iv)=\Phi(x)$ as $R_i$ only acts on the $p$ component. Therefore \eqref{eq:testinvmeasass} simplifies to
	\begin{equation*}
	\LX (\Phi)(z)+\tilde{\lambda}(R_iz) -\tilde{\lambda}(z).
	\end{equation*}
	This is equal to zero by \eqref{eq:switchrel} since $\LX \Phi= \langle\nabla_x\Phi,\sP p\rangle_{\cH}$. By Theorem \ref{thm:invmeas} we have that \eqref{eq:Lmu=0onC0} holds for all $f\in C_0(\cH^2)$. We can extend this to hold for all $f\in L_{\mu}^2$ by density.
\end{proof}

\subsection{Proofs of Section \ref{sec:core}}

\begin{proof}[Proof of Lemma \ref{lem:wellposedframework}]
In order to verify that Hypothesis \ref{hyp:IMass} holds. The only assumption which does not follow immediately is \eqref{eq:lambdaupperbound}. By Cauchy-Schwartz and \eqref{eq:Rassump4} we have
\begin{align*}
        \lambda(x)&=\sum_{n=1}^\infty \frac{1}{2}(\langle\nabla_x\Phi(x),v-R_nv\rangle_{\cH})^+ \\
        &\leq \frac{1}{2}\lVert \nabla_x\Phi(x)\rVert_{\h} \sum_{n=1}^\infty\lVert (1-R_n)v\rVert_{\h} \\
        &\leq \frac{1}{2}\lVert \nabla_x\Phi(x)\rVert_{\h} \sqrt{\sum_{n=1}^\infty\lVert (1-R_n)v\rVert_{\h}^2}
        \leq \lVert \nabla_x\Phi(x)\rVert_{\h}\lVert v\rVert_\h.
    \end{align*}
    Then by Assumption \ref{ass:1}
    \begin{equation}\label{eq:lambdaframeworkupperbound}
        \lambda(x) \lesssim  (1+\lVert x\rVert_{\h})\lVert v\rVert_\h
    \end{equation}
    and \eqref{eq:lambdaupperbound} follows.
\end{proof}

\begin{prop}\label{prop:derformula}
	Let $\{\cPt_t\}$ be the semigroup corresponding to the generator \eqref{eq:smoothenedgen} and assume that $\lt_i$ is continuously differentiable for all $i$. Then for any $t>0, f\in \core, x,v \in \cH$ we have that \eqref{eq:derivativeofBS} holds.
%
\end{prop}

\begin{proof}[Proof of Proposition \ref{prop:derformula}]
Define the total rate of events to be
\begin{equation*}
\Lambda_t(x,v) = \sum_{i=1}^\infty\int_0^t\lt_i(\varphi_s(x,v))ds.
\end{equation*}
Set
\begin{equation*}
F_{n,t}=F_{n,t}(x,v)=\mathbb{E}_{x,v}[\mathbbm{1}_{[T_n,T_{n+1})}(t)f(X_t,V_t)]
\end{equation*} 
then we have
\begin{equation*}
\tilde \cP_tf(x,v)= \sum_{n=0}^{\infty} F_{n,t}.
\end{equation*}
Next we observe that for $n\geq 1$ $F_{n,t}=QF_{n-1,\cdot}(t,x,v)$ where
\begin{align}\label{eq:Q}
Qg(t,x,v) &=\mathbb{E}_{x,v}[\mathbbm{1}_{T_1 < t}g(t-T_1,X_{t},V_t)]\nonumber\\
&= \sum_{i=1}^\infty\int_0^t e^{-\Lambda_s(x,v)} \lt_i(\varphi_s(x,v))\cB_i g(t-s, \varphi_s(x,v))ds.
\end{align}

We shall make the following inductive hypothesis.
	\begin{align}\label{eq:indhyp}
\nabla F_{n,t}&=\mathbb{E}[\mathbbm{1}_{[T_n,T_{n+1})}(t) C_t\nabla f(X_t,V_t)]-\mathbb{E}\left[\left(\int_0^tC_s \nabla\lambda(X_s,V_s)ds\right)\mathbbm{1}_{[T_n,T_{n+1})}(t) f(X_t,V_t)\right] \\
&+\mathbb{E}\left[\sum_{m=1}^nC_{T_m-}\nabla \log\left(\lambda_{J_m}\right)(X_{T_m},V_{T_m-})\mathbbm{1}_{[T_n,T_{n+1})}(t) f(X_t,V_t)\right]
\end{align}

We shall first show that \eqref{eq:indhyp} holds for $n=0$. For $n=0$ we can write
\begin{equation*}
F_{0,t}=f(\varphi_t(x,v))e^{-\Lambda_t(x,v)}.
\end{equation*}
We can differentiate this to obtain
\begin{equation*}
\nabla F_{0,t} = \nabla \varphi_t(x,v)(\nabla f)(\varphi_t(x,v))e^{-\Lambda_t(x,v)} -f(\varphi_t(x,v))\nabla \Lambda_t(x,v)e^{-\Lambda_t(x,v)}.
\end{equation*}
Rewriting this in terms of expectations we have
\begin{equation*}
\nabla F_{0,t} = \mathbb{E}[\mathbbm{1}_{[T_0,T_{1})}(t)(\nabla \varphi_t(x,v)(\nabla f)(X_t,V_t)-f(X_t,V_t)\nabla \Lambda_t(x,v))].
\end{equation*}

Let us assume that \eqref{eq:indhyp} holds for some $n\geq 0$. We can differentiate \eqref{eq:Q} to find
\begin{align*}
\nabla Qg(t,x,v) &= \sum_{i=1}^\infty\int_0^t e^{-\Lambda_s(x,v)} \lt_i(\varphi_s(x,v))\nabla \varphi_s(x,v)(\nabla\cB_i g)(t-s, \varphi_s(x,v))ds \\
&-\sum_{i=1}^\infty\int_0^t \nabla \Lambda_s(x,v) e^{-\Lambda_s(x,v)} \lt_i(\varphi_s(x,v))\cB_i g(t-s, \varphi_s(x,v))ds\\
&+ \sum_{i=1}^\infty\int_0^t e^{-\Lambda_s(x,v)} \nabla\varphi_s(x,v)\nabla\lt_i(\varphi_s(x,v))\cB_i g(t-s, \varphi_s(x,v))ds
\end{align*}
We can rewrite these in terms of expectations using the definition of $Q, J_1$.
\begin{align*}
\nabla Qg(t,x,v) &= \mathbb{E}_{x,v}\left[\mathbbm{1}_{T_1 < t} \nabla \varphi_{T_1}(x,v)(\nabla\cB_{J_1} g)(t-T_1,X_{T_1},V_{T_1-})\right] \\
&-\mathbb{E}_{x,v}\left[\mathbbm{1}_{T_1 < t} \nabla\Lambda_{T_1}(x,v)  g(t-T_1, X_{T_1},V_{T_1})\right]\\
&+ \mathbb{E}_{x,v}\left[\mathbbm{1}_{T_1 < t} \nabla \varphi_{T_1}(x,v)(\nabla\log\lt_{J_1})(X_{T_1},V_{T_1-})g(t-T_1, X_{T_1},V_{T_1})\right].
\end{align*}
Now if $J_1=i$ for some $i\geq 1$ we have
\begin{equation*}
(\nabla\cB_{i} g)(s,y,w) = \nabla \left[ g(s,y,R_i w) \right] = \cB'_i (\nabla g)(s,y,R_nw) = \cB'_i \cB_i(\nabla g)(s,y,w) 
\end{equation*}
where 
\begin{equation*}
\cB'_i = \left(\begin{array}{cc}
1 & 0\\
0 & R_i
\end{array}\right).
\end{equation*}
With this notation in mind we can write 
\begin{align*}
\nabla Qg(t,x,v) &= \mathbb{E}_{x,v}\left[\mathbbm{1}_{T_1 < t} \nabla \varphi_{T_1}(x,v)\cB_{J_1}' (\nabla g)(t-T_1,X_{T_1},V_{T_1})\right] \\
&-\mathbb{E}_{x,v}\left[\mathbbm{1}_{T_1 < t} \nabla\Lambda_{T_1}(x,v)  g(t-T_1, X_{T_1},V_{T_1})\right]\\
&+ \mathbb{E}_{x,v}\left[\mathbbm{1}_{T_1 < t} \nabla \varphi_{T_1}(x,v)(\nabla\log\lt_{J_1})(X_{T_1},V_{T_1-})g(t-T_1, X_{T_1},V_{T_1})\right].
\end{align*}

Now setting $g(t,x,v)=F_{n,t}(x,v)$ and using the inductive hypothesis \eqref{eq:indhyp} we have
\begin{align*}
&\nabla Qg(t,x,v) = \mathbb{E}_{x,v}\left[\mathbbm{1}_{[T_n,T_{n+1})}(t) C_t\nabla f(X_t,V_t)]\right]\\
&+\mathbb{E}_{x,v}\left[\sum_{m=2}^nC_{T_m-}\nabla \log\left(\lt_{J_m}\right)(X_{T_m},V_{T_m-})\mathbbm{1}_{[T_n,T_{n+1})}(t) f(X_t,V_t)\right] \\
&-\mathbb{E}_{x,v}\left[\mathbbm{1}_{[T_{n+1},T_{n+2})}(t) \nabla\Lambda_{T_1}(x,v)  f( X_t,V_t)\right]-\mathbb{E}_{x,v}\left[\left(\int_{T_1}^tC_{T_1}\nabla \varphi_s \nabla\lt(X_s,V_s)ds\right)\mathbbm{1}_{[T_{n+1},T_{n+2})}(t) f(X_t,V_t)\right]\\
&+ \mathbb{E}_{x,v}\left[\mathbbm{1}_{[T_1,T_{2})}(t) \nabla \varphi_{T_1}(x,v)(\nabla\log\lt_{J_1})(X_{T_1},V_{T_1-})g(t-T_1, X_{T_1},V_{T_1})\right].
\end{align*}
Combining these terms we have that \eqref{eq:indhyp} holds for $n+1$, therefore holds for all $n\geq 0$.
\end{proof}

\begin{prop}\label{prop:coresmoothcase}
	Let $\{\cPt_t\}$ be the semigroup corresponding to the generator \eqref{eq:smoothenedgen} and assume that $\Phi$ is twice continuously differentiable with both $\nabla_x\Phi,\nabla_x^2\Phi$ bounded on bounded sets. For any $f\in C_b^1(\cH\times\cH)$ with $\mathrm{supp}(f)\subseteq B_r$ for some $r>0$ then there exists a constant $C(r)$ which depends on $r$ and $\Phi$ such that \eqref{eq:boundonxi} holds.
\end{prop}

\begin{proof}[Proof of Proposition \ref{prop:coresmoothcase}]
	We shall define an extended PDMP which will also keep track of component of the latest jump and the pre-multiplier $C_t$, 
	\begin{equation*}
	\tilde{\cZ}_t = (X_t,V_{t},C_t,J_{T_{N_t}})
	\end{equation*} 
	this is a PDMP with rate $\lt$, deterministic motion $\tilde{\varphi}_t$ and jumps according to the Markov kernel $\tilde{Q}$ where
	\begin{align*}
	\tilde{\varphi}_t (x,v,c,j)& = (\varphi_t(x,v),c\nabla \varphi_t(x,v),j)\\
	\tilde{Q}((x,v,c,j),(dy,dw,d\overline{c},di) &= \sum_{n=1}^\infty \frac{\lt_n(x,v)}{\sum_{m=1}^\infty \lt_m(x,v)} \delta_{x,R_nv,c\mathcal{B}_n',m}(dy,dw,d\overline{c},di).
	\end{align*}
	The random counting measure, $\hat{p}$,
	 determined by $\tilde{\cZ}_t$  is 
	 \begin{equation*}
	 \hat{p} = \sum_{n=1}^\infty \delta_{T_n,\tilde{\cZ}_{T_n}},
	 \end{equation*}
	 	where $\delta_z$ here represents the Dirac Delta measure which has $\delta_z(\{z\})=1$. By \cite[Equation (7.32)]{MR2189574} the random counting measure $\hat{p}$ has compensating measure $\tilde{p}$ which  is given by
	\begin{equation*}
	\tilde{p}(dt, d\tilde{z}) = \lt(\tilde{\cZ}_{t-}) \tilde{Q}(\tilde{\cZ}_{t-},d\tilde{z}) dt.
	\end{equation*}
	That is $p-\tilde{p}$ is a local Martingale (with filtration given by the natural filtration of $\tilde{\cZ}_t$).
	
	Notice that for integrable $g:[0,\infty)\times\cH\times \cH\times \cH^{2\times 2}\times \N \to \R$ we have
	\begin{equation*}
	\int g dp = \sum_{n=1}^\infty g(T_n,X_{T_n}, V_{T_n},C_{T_n},J_n) - \int_0^\infty \sum_{m=1}^\infty \lt_m(X_s,V_{s-}) g(s,X_s,R_mV_{s-}, C_{s-}\mathcal{B}_m',m)ds.
	\end{equation*}
	Fix $t>0$ and set $g_t$ to be
	\begin{equation*}
	g_t(s,x,v,c,j) =  c\mathcal{B}_j'\nabla \log\lt_j(x,R_jv) \mathbbm{1}_{s\leq t}.
	\end{equation*}
	With this choice of $g$ we have that
	\begin{align*}
	\int g_t dp &= \sum_{n=1}^{N_t} C_{T_n}\mathcal{B}_{J_n}'\nabla \log\lt_{J_n}(X_{T_n}, R_{J_n}V_{T_{n}}) - \int_0^t \sum_{m=1}^\infty \lt_m(X_s,V_{s-}) C_{s-}\mathcal{B}_m'\mathcal{B}_m'\nabla \log\lt_m(X_t, R_m^2V_{s-})  ds\\
	&=\sum_{m=1}^{N_t} C_{T_m-}\nabla \log\lt_{J_m}(X_{T_n}, V_{T_{n}-}) - \int_0^t  C_s\nabla \lt(X_s,V_{s}) ds =\xi_t
	\end{align*}
	is a local Martingale. Moreover, by \cite[Proposition 4.6.2]{MR2189574} we have
	\begin{equation*}
	\mathbb{E}\left[\left\lVert\int gdp\right\rVert_{\h\times\h}^2\right] = \mathbb{E}[\int \lVert g\rVert_{\h\times\h} ^2d\tilde{p}].
	\end{equation*}
	Therefore, for this choice of $g$ we have
	\begin{align}
	\mathbb{E}\left[\left\lVert\xi_t\right\rVert_{\h\times\h}^2\right] &= \mathbb{E}\left[\int_0^t \sum_{j=1}^\infty  \lt_j(X_s,V_s)\lvert C_s\nabla \log\lt_j(X_s,V_{s})\rvert^2 ds \right]\nonumber\\
	&\leq \mathbb{E}\left[\int_0^t \sum_{j=1}^\infty  \frac{\lvert \nabla \lt_j(X_s,V_{s})\rvert^2}{\lt_j(X_s,V_s)} ds \right] \label{eq:secondmomentofxi}
	\end{align}
	
	Recall $\lt_j$ is given by \eqref{eq:smoothlambda} so 
	\begin{align*}
	\frac{\lvert \nabla \lt_j(x,v)\rvert^2}{\lt_j(x,v)} = \frac{\psi'(\langle\nabla_x\Phi(x), v-R_jv\rangle_{\cH})^2\lvert\nabla \langle\nabla_x\Phi(x), v-R_jv\rangle_{\cH}\rvert^2}{-\log\left(\phi(\exp(-\langle\nabla_x\Phi(x), v-R_jv\rangle_{\cH}))\right)}
	\end{align*} 
	where
	\begin{equation*}
	\psi(s)=\log\left(\phi(\exp(-s))\right).
	\end{equation*}
	Note that for the choice of $\phi(r)=\frac{r}{1+r}$ we have that for all $s\in \R$
	\begin{equation*}
	0\leq \frac{\psi'(s)^2}{-\log(\phi(\exp(-s)))} \leq \frac{1}{2}.
	\end{equation*}
	Using this bound we obtain
	\begin{align*}
	\frac{\lvert \nabla \lt_j(x,v)\rvert^2}{\lt_j(x,v)} \leq \frac{1}{2} \lvert\nabla \langle\nabla_x\Phi(x), v-R_jv\rangle_{\cH}\rvert^2 =  \frac{1}{2} (\lvert\nabla_x^2\Phi(x)(v-R_jv)\rvert^2 + \lvert(1-R_j)\nabla_x\Phi(x)\rvert^2).
	\end{align*} 
	Summing over $j$ we have and using \eqref{eq:Rassump4} we have
	\begin{equation}\label{eq:summinglambda}
	\sum_{j=1}^\infty  \frac{\lvert \nabla \lt_j(x,v)\rvert^2}{\lt_j(x,v)} \leq 2(\lVert\nabla_x^2\Phi(x)\rVert^2\lVert v\rVert^2 + \lVert\nabla_x\Phi(x)\rVert^2).
	\end{equation}
	Therefore
	\begin{equation*}
	    \mathbb{E}\left[\left\lVert\xi_t\right\rVert_{\h\times\h}^2\right] \leq 2\mathbb{E}\left[\int_0^t  (\lVert\nabla_x^2\Phi(X_s)\rVert^2 \lVert V_s\rVert^2+ \lVert\nabla_x\Phi(X_s)\rVert^2) ds \right].
	\end{equation*}
	Since there are no refreshments, plus both $\varphi_t$ and $R_i$ preserve the norm we have that $\lVert (X_s,V_s)\rVert_{\h\times\h}=\lVert (x,v)\rVert_{\h\times\h}$. That is, the process always remains on the sphere centered at zero with radius $\lVert (x,v)\rVert_{\h\times\h}$. Moreover as $f$ has support contained within $B_r$ we need only consider $(x,v)\in B_r$ so $(X_s,V_s)\in B_r$. Since we have that $\nabla_x\Phi, \nabla_x^2\Phi$ and $V_s$ are bounded on bounded sets there exists a positive constant $C(r)$ such that
		\begin{equation}\label{eq:momentboundofxi}
	    \mathbb{E}\left[\left\lVert\xi_t\right\rVert_{\h\times\h}^2\right] \leq C(r)^2 t.
	\end{equation}

\end{proof}

\subsection{Proofs of Section \ref{sec:hypocoercivity}}\label{sec:proofsofhypocoercivity}

\begin{proof}[Proof of Proposition \ref{prop:resest}]
By \cite[Theorem 3.7]{eisenhuth2021essential} we have for any 
$g\in \cylfunc(\h)$ and $h=g-\cA g$ it holds 
\begin{align}
    &\int_\h g^2+\langle \C \nabla_x g,\nabla_x g\rangle_\h d\pi =\int_\h gh d\pi\label{eq:IBPforderivativesofA1}\\
    &\int_\h \lVert\C^{\frac{1}{2}}\nabla_x g\rVert_\h^2 +\Tr[(\C\nabla_x^2 g)^2] +\langle \nabla_x^2\Phi \C\nabla_xg,\C\nabla_x g\rangle_\h d\pi = \int_\h (\cA g)^2 d\pi.\label{eq:IBPforderivativesofA2}
\end{align}
As shown in \cite[Remark 3.8]{eisenhuth2021essential} \eqref{eq:IBPforderivativesofA1} implies
\begin{equation}\label{eq:firstorderboundresolvent}
    \int_\h g^2+\langle \C \nabla_x g,\nabla_x g\rangle_\h d\pi \leq \int_\h h^2 d\pi.
\end{equation}
Note that although \cite{eisenhuth2021essential} typically assumes $\Phi$ is convex the results we used above do not require convexity. Therefore \eqref{eq:d1Resolventbound} follows.

Since $\nabla_x^2\Phi$ is bounded there exists $c_\Phi\geq 0$ such that $\C\nabla_x^2\Phi(x)\geq -c_\Phi$ for all $x\in \h$. Then we can use this bound in \eqref{eq:IBPforderivativesofA2} to obtain
\begin{equation*}
    \int_\h \lVert\C^{\frac{1}{2}}\nabla_x g\rVert_\h^2 +\Tr[(\C\nabla_x^2 g)^2] d\pi \leq  \int_\h (\cA g)^2 d\pi+c_\Phi\int_\h\langle  \nabla_xg,\C\nabla_x g\rangle_\h d\pi.
\end{equation*}
It remains to bound the right hand side. For the first term we use that $\cA g=g-h$ and $\lVert g\rVert_{L_\pi^2}\leq \lVert h\rVert_{L_\pi^2}$. For the second term we use \eqref{eq:firstorderboundresolvent}. These give
\begin{equation*}
    \int_\h \lVert\C^{\frac{1}{2}}\nabla_x g\rVert_\h^2 +\Tr[(\C\nabla_x^2 g)^2] d\pi \leq  (2+c_\Phi)\int_\h  h^2 d\pi.
\end{equation*}

It remains to show that \eqref{eq:innerprodresolventbound} holds.

	Fix $\varphi\in C^2(\cH)$, and note that by Young's inequality we have
	\begin{equation*}
	2\langle \varphi (\Sigma^\frac{1}{2}\nabla_x) \Phi, (\Sigma^\frac{1}{2}\nabla_x)\varphi\rangle_{L_\pi^2} \leq \varepsilon^{-1}\lVert (\Sigma^\frac{1}{2}\nabla_x)\varphi\rVert_{L_\pi^2}^2 + \varepsilon \lVert \varphi (\Sigma^\frac{1}{2}\nabla_x)\Phi\rVert_{L_\pi^2}^2
	\end{equation*}
	where $\varepsilon>0$ will be chosen later. Also we have
	\begin{align*}
	2\langle \varphi (\Sigma^\frac{1}{2}\nabla_x) \Phi, (\Sigma^\frac{1}{2}\nabla_x)\varphi\rangle_{L_\pi^2} &= \langle (\Sigma^\frac{1}{2}\nabla_x) \Phi, (\Sigma^\frac{1}{2}\nabla_x)(\varphi^2)\rangle_{L_\pi^2}\\
	&=\langle -\cA\Phi, \varphi^2\rangle_{L_\pi^2}.
	\end{align*}	
	Now using \eqref{eq:boundonLaplacianU} we get
	\begin{align*}
	2\langle \varphi (\Sigma^\frac{1}{2}\nabla_x) \Phi, (\Sigma^\frac{1}{2}\nabla_x)\varphi\rangle_{L_\pi^2} &\geq C_1\lVert \varphi (\Sigma^\frac{1}{2}\nabla_x)\Phi\rVert_{L_\pi^2}^2 -C_2\lVert \varphi\rVert_{L_\pi^2}^2.
	\end{align*}
	Now by setting $\varepsilon=C_1/2$ we get
	\begin{equation}\label{eq:varphidUestimate}
	\lVert \varphi (\Sigma^\frac{1}{2}\nabla_x)\Phi\rVert_{L_\pi^2}^2  \leq \frac{4}{C_1^2}\lVert (\Sigma^\frac{1}{2}\nabla_x)\varphi\rVert_{L_\pi^2}^2 +\frac{2C_2}{C_1} \lVert \varphi\rVert_{L_\pi^2}^2.
	\end{equation}
	
	Choose $\varphi(x) = \sqrt{\delta+\langle (\Sigma^\frac{1}{2}\nabla_x)g,(\Sigma^\frac{1}{2}\nabla_x)g\rangle_{\h}}$ and observe that
	\begin{align*}
	    \lVert\Sigma^{\frac{1}{2}}\nabla_x\phi(x)\rVert_\cH &= \frac{\lVert\langle (\nabla_x\Sigma\nabla_x)g,(\Sigma^\frac{1}{2}\nabla_x)g\rangle_{\h}\rVert_{\cH}}{\sqrt{\delta+\langle (\Sigma^\frac{1}{2}\nabla_x)g,(\Sigma^\frac{1}{2}\nabla_x)g\rangle_{\h}}} \leq \lVert \nabla_x\Sigma\nabla_x g\rVert_{\h}\frac{\lVert\Sigma^\frac{1}{2}\nabla_xg\rVert_{\h}}{\sqrt{\delta+\langle (\Sigma^\frac{1}{2}\nabla_x)g,(\Sigma^\frac{1}{2}\nabla_x)g\rangle_{\h}}}\\
	    &\leq \lVert \nabla_x\Sigma\nabla_x g\rVert_{\h}.
	\end{align*}
	With this choice of $\varphi$ we have that \eqref{eq:varphidUestimate} gives
	\begin{equation*}
	\left\lVert \lVert (\Sigma^\frac{1}{2}\nabla_x)g \rVert_{\cH} (\Sigma^\frac{1}{2}\nabla_x)\Phi\right\rVert_{L_\pi^2}^2  \leq \frac{4}{C_1^2}\lVert \nabla_x\Sigma\nabla_x g\rVert_{L_\pi^2}^2 +\frac{2C_2}{C_1} \left\lVert (\Sigma^\frac{1}{2}\nabla_x)g \right\rVert_{L_\pi^2}^2+\frac{2C_2}{C_1}\delta.
	\end{equation*}
	Finally by \eqref{eq:d1Resolventbound},\eqref{eq:d2Resolventbound} and taking $\delta\to 0$ we have 
	\begin{equation*}
	\left\lVert \lVert (\Sigma^\frac{1}{2}\nabla_x)u_g \rVert_{\cH} (\Sigma^\frac{1}{2}\nabla_x)\Phi\right\rVert_{L_\pi^2}^2  \leq \left(\frac{4\kappa_1^2}{C_1^2}+\frac{2C_2}{C_1} \right)\lVert g \rVert_{L_\pi^2}^2.
	\end{equation*}
\end{proof}

\begin{proof}[Proof of Proposition \ref{prop:SAestintermsofresolvent}]
    Let us first show that \eqref{eq:BS} holds. Fix $f\in \cylfunc(\h\times\h)$ and observe that
 	\begin{equation*}
 	\lvert \langle BSf,\Pi f\rangle_{L_\mu^2}\rvert  = \lvert \langle f,BA^*\Pi f\rangle_{L_\mu^2}\rvert.
 	\end{equation*}
Since $\Pi$ is an orthogonal projection onto the kernel of $S$ we can rewrite the right hand side of the above expression as
	\begin{equation*}
\lvert \langle BSf,\Pi f\rangle_{L_\mu^2}\rvert  = \lvert \langle (1-\Pi)f,SB^*\Pi f\rangle_{L_\mu^2}\rvert \leq \lVert (1-\Pi)f\rVert \lVert SB^*f\rVert_{L_\mu^2}.
\end{equation*}
    Therefore \eqref{eq:BS} follows provided we have 
        \begin{align}
        \lVert (BS)^*g\rVert_{L_\mu^2}\leq c_4\lVert g\rVert_{L_\mu^2}.\label{eq:remainderformS}
    \end{align}
    
	For $f\in \cylfunc(\cH\times\cH)$ let $u_f\in L^2_\pi(\cH)$ be the function that satisfies $(1-G)u_f=\Pi f$. This is an abuse of notation as we defined $u_g$ above for $g\in L_\pi^2$, as $\Pi f$ can be viewed as belonging to $L_\pi^2$ we should use the notation $u_{\Pi f}$ however to simplify the exposition we drop the $\Pi$ in the notation. Then $SB^*f = SAu_f$. As $u_f$ does not depend on $v$ we can write $SAu_f$ as
	\begin{equation*}
	SAu_f = S\left\langle 
	v, \nabla_xu_f(x) \right\rangle_{\cH} = -\sum_{n=1}^\infty\lambda_n^e(x,v) \left\langle 
	(1-R_n)v, \nabla_xu_f(x) \right\rangle_{\cH}+ \lref (\Pi-1)\left\langle 
	v, \nabla_xu_f(x) \right\rangle_{\cH}.
	\end{equation*}
	Recall that $\lambda_n^e(x,v)= \frac{1}{4}\lvert \langle (1-R_n)v,\nabla_x\Phi(x)\rangle_\cH\rvert+\gamma_n(x,v)$ and $\Pi$ acts by integration with respect to a centred Gaussian measure. This gives
	\begin{equation*}
	SAu_f= -\frac{1}{4}\sum_{n=1}^\infty\lvert \langle (1-R_n)v,\nabla_x\Phi(x)\rangle_\cH\rvert \left\langle 
	(1-R_n)v,  \nabla_xu_f(x) \right\rangle_{\cH} - \lref \left\langle 
	v, \nabla_xu_f(x) \right\rangle_{\cH}.
	\end{equation*}
	Taking the $L_\mu^2$ norm we have
	\begin{equation*}
	\lVert SAu_f \rVert_{L_\mu^2}\leq\left\lVert \frac{1}{4}\sum_{n=1}^\infty \langle (1-R_n)v,\nabla_x\Phi(x)\rangle_\cH\left\langle 
	(1-R_n)v, \nabla_xu_f(x) \right\rangle_{\cH}\right\rVert_{L_\mu^2} + \lref \lVert\left\langle 
	v, \nabla_xu_f(x) \right\rangle_{\cH}\rVert_{L_{\mu}^2}.
	\end{equation*}	
	Using that $\nu_0$ is a Gaussian measure with covariance $\Sigma$ we can write the second term (by Lemma \ref{lem:isserlis}) as
	\begin{equation*}
	\lVert\left\langle 
	v, \nabla_xu_f(x) \right\rangle_{\cH}\rVert_{L_{\mu}^2}^2 =  \int_{\cH}\langle 
	\Sigma\nabla_xu_f(x), \nabla_xu_f(x) \rangle_{\cH} \pi(dx) =  \lVert \Sigma^{\frac{1}{2}}\nabla_xu_f \rVert_{L_\pi^2}^2.
	\end{equation*}
	By Lemma \ref{lem:isserlis} the first term can be written as
	\begin{align*}
	&\lVert \frac{1}{4}\sum_{n=1}^\infty\langle v,(1-R_n)^*\nabla_x\Phi(x)\rangle_\cH\left\langle 
	v,(1-R_n)^* \nabla_xu_f(x) \right\rangle_{\cH}\rVert_{L_\mu^2}^2 \\
	&= \frac{1}{16}\!\!\sum_{n,m=1}^\infty\!\int_{\cH^2} \left[ \langle v,(1-R_n)^*\nabla_x\Phi(x)\rangle_\cH\!\left\langle 
	v,(1-R_n)^* \nabla_xu_f(x) \right\rangle_{\cH} \right. \\
	& \quad \quad \quad \quad \quad \quad \quad \left. \times \langle v,(1-R_m)^*\nabla_x\Phi(x)\rangle_\cH\!\left\langle 
	v,(1-R_m)^* \nabla_xu_f(x) \right\rangle_{\cH} \right] \mu(dx,dv) \\
	&=\frac{1}{16}\sum_{n,m=1}^\infty\int_{\cH}\langle \Sigma(1-R_n)^*\nabla_x\Phi(x),(1-R_n)^* \nabla_xu_f(x) \rangle_{\cH}\langle \Sigma(1-R_m)^*\nabla_x\Phi(x),(1-R_m)^* \nabla_xu_f(x) \rangle_{\cH}\pi(dx)\\
	&+\frac{1}{16}\sum_{n,m=1}^\infty\int_{\cH}\langle \Sigma(1-R_n)^*\nabla_x\Phi(x),(1-R_m)^*\nabla_x\Phi(x)\rangle_\cH\left\langle 
	\Sigma(1-R_n)^* \nabla_xu_f(x) ,(1-R_m)^* \nabla_xu_f(x) \right\rangle_{\cH}\pi(dx)\\
	&+\frac{1}{16}\sum_{n,m=1}^\infty\int_{\cH}\langle \Sigma(1-R_n)^*\nabla_x\Phi(x),(1-R_m)^* \nabla_xu_f(x) \rangle_{\cH}\langle 
	\Sigma(1-R_n)^* \nabla_xu_f(x) ,(1-R_m)^*\nabla_x\Phi(x)\rangle_\cH\pi(dx)\\
	&=\int_{\cH}\langle \Sigma\nabla_x\Phi(x), \nabla_xu_f(x) \rangle_{\cH}^2\pi(dx)\\
&+\frac{1}{16}\sum_{n=1}^\infty\int_{\cH}\langle \Sigma(1-R_n)^*\nabla_x\Phi(x),(1-R_n)^*\nabla_x\Phi(x)\rangle_\cH\left\langle 
\Sigma(1-R_n)^* \nabla_xu_f(x) ,(1-R_n)^* \nabla_xu_f(x) \right\rangle_{\cH}\pi(dx)\\
&+\frac{1}{16}\sum_{n=1}^\infty\int_{\cH}\langle \Sigma(1-R_n)^*\nabla_x\Phi(x),(1-R_n)^* \nabla_xu_f(x) \rangle_{\cH}^2\pi(dx)\\
	\end{align*}
	Now using Cauchy-Schwartz we have
	\begin{align*}
	&\lVert \frac{1}{4}\sum_{n=1}^\infty\langle v,(1-R_n)^*\nabla_x\Phi(x)\rangle_\cH\left\langle 
	v,(1-R_n)^* \nabla_xu_f(x) \right\rangle_{\cH}\rVert_{L_\mu^2}^2 \\
	&\leq \int_{\cH}\lVert \Sigma^{\frac{1}{2}}\nabla_x\Phi(x)\rVert_\cH^2 \lVert \Sigma^{\frac{1}{2}} \nabla_xu_f(x) \rVert_{\cH}^2\pi(dx)\\
	&+\frac{1}{8}\sum_{n=1}^\infty\int_{\cH}\lVert \Sigma^{\frac{1}{2}}(1-R_n)^*\nabla_x\Phi(x)\rVert_\cH^2\lVert 
	\Sigma^{\frac{1}{2}}(1-R_n)^* \nabla_xu_f(x)\rVert_\cH^2\pi(dx).
	\end{align*}
	
	Fix a function $F:\cH\to\cH$ then by \eqref{eq:Rassump2} we have the following bound
	\begin{align}
	\lVert \Sigma^{\frac{1}{2}}(1-R_n)^*F\rVert_\cH^2 &\leq \sum_{n=1}^\infty 	\lVert \Sigma^{\frac{1}{2}}(1-R_n)^*F\rVert_\cH^2\label{eq:Gestimate1}\\
	&=\sum_{n=1}^\infty 	\langle (1-R_n)\Sigma(1-R_n)^*F, F\rangle_\cH\label{eq:Gestimate2}\\
	&= 4\lVert \Sigma^{\frac{1}{2}}F\rVert_\cH^2\label{eq:Gestimate3}
	\end{align}
	
	By setting $F(x)=\nabla_x\Phi(x)$ in \eqref{eq:Gestimate1}-\eqref{eq:Gestimate3} we have $\lVert \Sigma^{\frac{1}{2}}(1-R_n)^*\nabla_x\Phi(x)\rVert_\cH^2 \leq 4\lVert \Sigma^{\frac{1}{2}}\nabla_x\Phi(x)\rVert_\cH^2$ which gives
	\begin{align*}
	&\lVert \frac{1}{4}\sum_{n=1}^\infty\langle v,(1-R_n)^*\nabla_x\Phi(x)\rangle_\cH\left\langle 
v,(1-R_n)^* \nabla_xu_f(x) \right\rangle_{\cH}\rVert_{L_\mu^2}^2 \\
&\leq \int_{\cH}\lVert \Sigma^{\frac{1}{2}}\nabla_x\Phi(x)\rVert_\cH^2 \lVert \Sigma^{\frac{1}{2}} \nabla_xu_f(x) \rVert_{\cH}^2\pi(dx)\\
&+\frac{1}{2}\sum_{n=1}^\infty\int_{\cH}\lVert \Sigma^{\frac{1}{2}}\nabla_x\Phi(x)\rVert_\cH^2\lVert 
\Sigma^{\frac{1}{2}}(1-R_n)^* \nabla_xu_f(x)\rVert_\cH^2\pi(dx)\\
&\leq 3\int_{\cH}\lVert \Sigma^{\frac{1}{2}}\nabla_x\Phi(x)\rVert_\cH^2 \lVert \Sigma^{\frac{1}{2}} \nabla_xu_f(x) \rVert_{\cH}^2\pi(dx).
	\end{align*}
	Therefore we have
	\begin{equation*}
	\lVert SAu_f \rVert_{L_\mu^2}\leq \sqrt{3}\left\lVert \left\lVert 
	\Sigma^\frac{1}{2}\nabla_xu_f\right\rVert_{\cH}\Sigma^\frac{1}{2}\nabla_x\Phi(x)\right\rVert_{L_\pi^2(\cH;\cH)}+ \lref \lVert \Sigma^{\frac{1}{2}}\nabla_xu_f \rVert_{L_\pi^2(\cH;\cH)}.
	\end{equation*}	
	Now by \eqref{eq:d1Resolventbound} and \eqref{eq:innerprodresolventbound} we can bound these terms by
	\begin{equation*}
	\lVert STu_f \rVert_{L_\mu^2}^2\leq \left(\sqrt{3}\left(\frac{4}{c_1^2}+\frac{c_2}{c_1} \right)^\frac{1}{2}+ \frac{1}{\sqrt{2}}\lref\right) \lVert f \rVert_{L_\mu^2}^2.
	\end{equation*}
	That is \eqref{eq:remainderformS} holds.
	
	Now we prove that \eqref{eq:BA(1-Pi)} holds. Observe that 	
	\begin{align*}
	\lvert\langle BA(1-\Pi) f,\Pi f\rangle_{L_\mu^2}\rvert = \lvert\langle (1-\Pi)f, (1-\Pi)AB^*\Pi f\rangle_{L_\mu^2}\rvert &= \lvert\langle (1-\Pi)f, (1-\Pi)A^2u_f\rangle_{L_\mu^2}\rvert \\
	&\leq  \lVert (1-\Pi)f\rVert_{L_{\mu}^2} \lVert (1-\Pi)A^2u_f\rVert_{L_\mu^2}.
	\end{align*}
	Recall from \eqref{eq:T2Pi} that
	\begin{align*}
	A^2u_f=\left\langle  v, \nabla^2_xu_f(x) v\right\rangle_{\cH}-\langle x,\nabla_xu_f(x) \rangle_{\cH}-\frac{1}{4}\sum_{n=1}^\infty \langle\nabla_x\Phi(x),(1-R_n)v\rangle_{\cH}\left\langle (1-R_n)v, \nabla_xu_f(x) \right\rangle_{\cH}.
	\end{align*}
	Applying $(1-\Pi)$ gives
	\begin{align*}
	&(1-\Pi)A^2u_f=\left\langle v, \nabla^2_xu_f(x)  v\right\rangle_{\cH}-\Tr(\Sigma\nabla^2_xu_f(x))\\
	&-\frac{1}{4}\sum_{n=1}^\infty \langle\nabla_x\Phi(x),(1-R_n)v\rangle_{\cH}\left\langle (1-R_n)v, \nabla_xu_f(x) \right\rangle_{\cH}+\frac{1}{4}\sum_{n=1}^\infty \left\langle\Sigma(1-R_n)^*\nabla_x\Phi(x), (1-R_n)^*\nabla_xu_f(x) \right\rangle_{\cH}.
	\end{align*}
	We can simplify this expression using \eqref{eq:Rassump2}
	\begin{align*}
	&(1-\Pi)A^2u_f=\left\langle v, \nabla^2_xu_f(x)  v\right\rangle_{\cH}-\Tr(\Sigma\nabla^2_xu_f(x))\\
	&-\frac{1}{4}\sum_{n=1}^\infty \langle\nabla_x\Phi(x),(1-R_n)v\rangle_{\cH}\left\langle (1-R_n)v, \nabla_xu_f(x) \right\rangle_{\cH}+\left\langle\Sigma\nabla_x\Phi(x), \nabla_xu_f(x) \right\rangle_{\cH}.
	\end{align*}
	
	Using that $\Pi$ is an orthogonal projection we can write
	\begin{align*}
	\lVert \left\langle v, \nabla^2_xu_f(x)  v\right\rangle_{\cH}-\Tr(\Sigma\nabla^2_xu_f(x)) \rVert_{L_{\mu}^2}^2=\int_{\cH^2} \left\langle v, \nabla^2_xu_f(x) v\right\rangle_{\cH}^2 - \Tr(\Sigma\nabla^2_xu_f(x))^2 \mu(dx,dv).
	\end{align*}
	By Lemma \ref{lem:isserlis} we have
	\begin{align*}
	\int_{\cH^2} \left\langle v, \nabla^2_xu_f(x) v\right\rangle_{\cH}^2\mu(dx,dp) =2 \int_{\cH}\Tr\left(\nabla^2_xu_f(x)\Sigma^2\nabla^2_xu_f(x)\right) \pi(dx)+\int_{\cH}\Tr(\Sigma\nabla^2_xu_f(x))^2\pi(dx).
	\end{align*}
	Thus
	\begin{align*}
	\lVert \left\langle v, \nabla^2_xu_f(x) v\right\rangle_{\cH}-\Tr(\Sigma\nabla^2_xu_f(x)) \rVert_{L_{\mu}^2}^2=2 \lVert \Sigma\nabla^2_xu_f(x)\rVert_{L_\pi^2(\cH;\cH\otimes\cH)}.
	\end{align*}
	We can treat the second term similarly, by Lemma \ref{lem:isserlis} we write the norm of this term as
	\begin{align*}
	&\lVert\frac{1}{4}\sum_{n=1}^\infty \langle\nabla_x\Phi(x),(1-R_n)v\rangle_{\cH}\left\langle (1-R_n)v, \nabla_xu_f(x) \right\rangle_{\cH}- \left\langle\Sigma\nabla_x\Phi(x), \nabla_xu_f(x) \right\rangle_{\cH} \rVert_{L_{\mu}^2}^2 \\
	&=\int_{\cH}\!\int_{\cH} \frac{1}{16}\sum_{n,m=1}^\infty \langle\nabla_x\Phi(x),(1-R_n)v\rangle_{\cH}\left\langle (1-R_n)v, \nabla_xu_f(x) \right\rangle_{\cH}\langle\nabla_x\Phi(x),(1-R_m)v\rangle_{\cH}\left\langle (1-R_m)v, \nabla_xu_f(x) \right\rangle_{\cH}\\
	&- \langle\Sigma\nabla_x\Phi(x), \nabla_xu_f(x) \rangle_{\cH}^2 \nu_0(dv)\pi(dx)\\
	&=\int_{\cH} \frac{1}{16}\sum_{n,m=1}^\infty \langle\Sigma(1-R_n)^*\nabla_x\Phi(x), (1-R_n)^*\nabla_xu_f(x) \rangle_{\cH}\langle\Sigma(1-R_m)^*\nabla_x\Phi(x), (1-R_m)^*\nabla_xu_f(x) \rangle_{\cH}\pi(dx)\\
	&+\int_{\cH}\frac{1}{16}\sum_{n,m=1}^\infty \langle \Sigma(1-R_n)^*\nabla_x\Phi(x),(1-R_m)^*\nabla_x\Phi(x)\rangle_{\cH}\langle \Sigma(1-R_n)^*\nabla_xu_f(x), (1-R_m)^*\nabla_xu_f(x)\rangle_{\cH}\pi(dx)\\
	&+\int_{\cH} \frac{1}{16}\sum_{n,m=1}^\infty \langle\Sigma(1-R_n)^*\nabla_x\Phi(x), (1-R_m)^*\nabla_xu_f(x) \rangle_{\cH}\langle  \Sigma(1-R_n)^*\nabla_xu_f(x) ,(1-R_m)^*\nabla_x\Phi(x)\rangle_{\cH}\pi(dx)\\
	&- \int_{\cH}\langle\Sigma\nabla_x\Phi(x), \nabla_xu_f(x) \rangle_{\cH}^2 \pi(dx)\\
	&=\int_{\cH}\frac{1}{16}\sum_{n=1}^\infty \langle \Sigma(1-R_n)^*\nabla_x\Phi(x),(1-R_n)^*\nabla_x\Phi(x)\rangle_{\cH}\langle \Sigma(1-R_n)^*\nabla_xu_f(x), (1-R_n)^*\nabla_xu_f(x)\rangle_{\cH}\pi(dx)\\
	&+\int_{\cH} \frac{1}{16}\sum_{n=1}^\infty \langle\Sigma(1-R_n)^*\nabla_x\Phi(x), (1-R_n)^*\nabla_xu_f(x) \rangle_{\cH}\langle  \Sigma(1-R_n)^*\nabla_xu_f(x) ,(1-R_n)^*\nabla_x\Phi(x)\rangle_{\cH}\pi(dx)\\
&\leq 2\int_{\cH} \lVert \Sigma^{\frac{1}{2}}\nabla_x\Phi(x)\rVert_{\cH}^2\lVert \Sigma^{\frac{1}{2}}\nabla_xu_f(x)\rVert_{\cH}^2\pi(dx).
	\end{align*}
	Thus,
	\begin{align*}
	\lVert(1-\Pi)T^2u_f\rVert_{L_\mu^2}&\leq\sqrt{2} \lVert \Sigma\nabla^2_xu_f\rVert_{L_\pi^2}+\sqrt{2}\left\lVert  \lVert 
		\Sigma^\frac{1}{2}\nabla_xu_f(x) \rVert_{\cH}\Sigma^\frac{1}{2}\nabla_x\Phi(x) \right\rVert_{L_\pi^2} .
	\end{align*}
	By \eqref{eq:d2Resolventbound} and \eqref{eq:innerprodresolventbound} we can bound this by
	\begin{align*}
	\lVert(1-\Pi)T^2u_f\rVert_{L_\mu^2}&\leq\sqrt{2}(\kappa_1+\kappa_2)\lVert \Pi f\rVert_{L_\mu^2}.
	\end{align*}
\end{proof}

\appendix

\section{Differentiability in the direction $\sX$}

Define the set $C$ to be the set of all $f:\cH^2\to \R$ for which the following limit is well-defined and bounded:
\begin{equation*}
\LX f(x,p):=\lim_{t\to 0}\frac{1}{t}[f(\varphi_t(x,v))-f(x,v)] = \frac{d}{dt}\lvert_{t=0} f( \varphi_t(x,v)).
\end{equation*}

\begin{lemma}\label{lem:varphiinvundermu0}
	Assume that \eqref{eq:cons} holds, then flow $\varphi_t$ is invariant under the probability measure $\mu_0$.
\end{lemma}

\begin{proof}[Proof of Lemma \ref{lem:varphiinvundermu0}]
	Fix $y,w\in \cH$ and define the function 
	$$
	\psi(x,v) = \exp(i\langle x,y\rangle_{\cH} + i\langle v,w\rangle_{\cH}).
	$$
	Then since $\varphi_t$ is a bounded linear operator on the space $\cH^2$ we have
	\begin{equation*}
	\psi(\varphi_t(x,v)) = \exp(i\left\langle \varphi_t^*(y,w),(x,v) \right\rangle_{\cH^2}).
	\end{equation*} 
	Here $\varphi_t^*$ denotes the $\cH^2$ adjoint of $\varphi_t$, when viewed as a linear operator.
	Now integrating with respect to $\mu_0$ we have
	\begin{equation*}
	\int_{\cH^2} \psi(\varphi_t(x,v)) \mu_0(dx,dv) = \hat{\mu_0}(\varphi_t^*(y,w))
	\end{equation*}
	where $\hat{\mu_0}$ denotes the Fourier transform of the measure $\mu_0$. As $\mu_0$ is a Gaussian measure we have
	\begin{equation*}
	\int_{\cH^2} \psi(\varphi_t(x,v)) \mu_0(dx,dv) = \exp(-\frac{1}{2}\langle \left(\begin{array}{cc}
	\Sigma_x & 0\\
	0 & \Sv
	\end{array}\right)\varphi_t^*(y,w),\varphi_t^*(y,w) \rangle_{\cH^2}).
	\end{equation*}
	Differentiating this expression we have
	\begin{align*}
	&\frac{d}{dt}\int_{\cH^2} \psi(\varphi_t(x,v)) \mu_0(dx,dv) = -\frac{1}{2}\left(\left\langle \left(\begin{array}{cc}
	\Sigma_x & 0\\
	0 & \Sv
	\end{array}\right)\frac{d}{dt}\varphi_t^*(y,w),\varphi_t^*(y,w) \right\rangle_{\cH^2} \right.\\
	&\left.+ \left\langle \left(\begin{array}{cc}
	\Sigma_x & 0\\
	0 & \Sv
	\end{array}\right)\varphi_t^*(y,w),\frac{d}{dt}\varphi_t^*(y,w) \right\rangle_{\cH^2}\right)\exp\left(-\frac{1}{2}\left\langle \left(\begin{array}{cc}
	\Sigma_x & 0\\
	0 & \Sv
	\end{array}\right)\varphi_t^*(y,w),\varphi_t^*(y,w) \right\rangle_{\cH^2}\right)\\
	&=-\frac{1}{2}\left(\left\langle \left(\begin{array}{cc}
	\Sigma_x & 0\\
	0 & \Sv
	\end{array}\right)\left(\begin{array}{cc}
	0 & -\sX\\
	\sP & 0
	\end{array}\right)\varphi_t^*(y,w),\varphi_t^*(y,w) \right\rangle_{\cH^2} \right.\\
	&\left.+ \left\langle \left(\begin{array}{cc}
	\Sigma_x & 0\\
	0 & \Sv
	\end{array}\right)\varphi_t^*(y,w),\left(\begin{array}{cc}
	0 & -\sX\\
	\sP & 0
	\end{array}\right)\varphi_t^*(y,w) \right\rangle_{\cH^2}\right)\exp\left(-\frac{1}{2}\left\langle \left(\begin{array}{cc}
	\Sigma_x & 0\\
	0 & \Sv
	\end{array}\right)\varphi_t^*(y,w),\varphi_t^*(y,w) \right\rangle_{\cH^2}\right)\\
	&=-\frac{1}{2}\left(\left\langle \left(\begin{array}{cc}
	0 & -\Sigma_x\sX\\
	\Sv\sP & 0
	\end{array}\right)\varphi_t^*(y,w),\varphi_t^*(y,w) \right\rangle_{\cH^2} \right.\\
	&\left.+ \left\langle \left(\begin{array}{cc}
	0 & \sP\Sv\\
	-\sX\Sigma_x & 0
	\end{array}\right)\varphi_t^*(y,w),\varphi_t^*(y,w) \right\rangle_{\cH^2}\right)\exp\left(-\frac{1}{2}\left\langle \left(\begin{array}{cc}
	\Sigma_x & 0\\
	0 & \Sv
	\end{array}\right)\varphi_t^*(y,w),\varphi_t^*(y,w) \right\rangle_{\cH^2}\right)\\
	&=-\left(\left\langle \left(\begin{array}{cc}
	0 & -\Sigma_x\sX\\
	\Sv\sP & 0
	\end{array}\right)\varphi_t^*(y,w),\varphi_t^*(y,w) \right\rangle_{\cH^2}\right)\exp\left(-\frac{1}{2}\left\langle \left(\begin{array}{cc}
	\Sigma_x & 0\\
	0 & \Sv
	\end{array}\right)\varphi_t^*(y,w),\varphi_t^*(y,w) \right\rangle_{\cH^2}\right)
	\end{align*}
	Now for any $x',v'\in\cH$ we have
	\begin{equation*}
	\left\langle \left(\begin{array}{cc}
	0 & -\Sigma_x\sX\\
	\Sv\sP & 0
	\end{array}\right)\left(\begin{array}{c}
	x' \\
	v'
	\end{array}\right),\left(\begin{array}{c}
	x' \\
	v'
	\end{array}\right) \right\rangle_{\cH^2} = \langle (\sP\Sv-\Sigma_x\sX)v',x'\rangle_{\cH}
	\end{equation*}
	which is equal to zero by \eqref{eq:cons}.
\end{proof}

\begin{lemma}\label{lem:L0skewsym}
	Assume \eqref{eq:cons} holds, then for any $f,g$ such that $t\mapsto f(\varphi_t(x,v))$ and $t\mapsto g(\varphi_t(x,v))$ are differentiable $\mu_0$-a.e. then
	\begin{equation*}
	\int_{\cH^2} \LX f(x,v) g(x,v) \mu_0(dx,dv) = - \int_{\cH^2} f(x,v) \LX g(x,v) \mu_0(dx,dv).
	\end{equation*}
\end{lemma}

\begin{proof}
	Recall that $\LX $ is the directional derivative with respect to the integral curve $\varphi_t$ so we have
	\begin{equation*}
	\int_{\cH^2} \LX f(x,v) g(x,v) \mu_0(dx,dv) = \lim_{t\to 0}\frac{1}{t} \left(\int_{\cH^2} f(\varphi_t(x,v)) g(x,v) \mu_0(dx,dv)-\int_{\cH^2} f(x,v) g(x,v) \mu_0(dx,dv)\right).
	\end{equation*}
	Using Lemma \ref{lem:varphiinvundermu0}
	\begin{align*}
		\int_{\cH^2} \LX f(x,v) g(x,v) \mu_0(dx,dv) &= \lim_{t\to 0}\frac{1}{t} \left(\int_{\cH^2} f(x,v) g(\varphi_{-t}(x,v)) \mu_0(dx,dv)-\int_{\cH^2} f(x,v) g(x,v) \mu_0(dx,dv)\right)\\
		&=-\int_{\cH^2} f(x,v) \LX g(x,v) \mu_0(dx,dv).
\end{align*}	
\end{proof}

\begin{corollary}\label{cor:Loadjformu}
	Assume Hypothesis \ref{hyp:IMass} holds. Then for any $f,g$ such that $t\mapsto f(\varphi_t(x,v))$ and $t\mapsto g(\varphi_t(x,v))$ are differentiable $\mu_0$-a.e. then
	\begin{equation*}
	\int_{\cH^2} \LX f(x,v) g(x,v) \mu(dx,dv) = - \int_{\cH^2} f(x,v) \LX g(x,v) \mu(dx,dv)+ \int \langle \nabla_x\Phi(x),\sP p\rangle f(x,v)g(x,v) \mu(dx,dv).
	\end{equation*}
\end{corollary}

\subsection*{Acknowledgements}

We wish to thank Andrew Duncan, Michela Ottobre and Andreas Eberle for helpful discussions.  

This work is part of the research programme ‘Zigzagging through computational barriers’ with project number 016.Vidi.189.043, which is financed by the Dutch Research Council (NWO).

\bibliography{bibliography}
\bibliographystyle{plain}

\end{document}